\documentclass{article}

\usepackage[english]{babel}
\usepackage[T1]{fontenc}
\usepackage[english]{babel}
\usepackage{hyperref}
\usepackage{authblk}

\usepackage[a4paper,top=2cm,bottom=2cm,left=3cm,right=3cm,marginparwidth=1.75cm]{geometry}
\usepackage{enumitem}
\setlist[itemize]{label=--}

\usepackage{amsmath,amsfonts,amssymb,mathtools,stmaryrd,amsthm, esint}
\usepackage{mathabx}
\usepackage{relsize}
\usepackage{bbold}

\newtheorem{theorem}{Theorem}[section]

\newtheorem{lemma}[theorem]{Lemma}
\newtheorem{proposition}[theorem]{Proposition}

\newtheorem{definition}[theorem]{Definition}
\theoremstyle{definition}
\newtheorem{example}[theorem]{Example}

\theoremstyle{remark}
\newtheorem{remark}[theorem]{Remark}

\usepackage{tabularx}
\usepackage{nicematrix}
\usepackage{arydshln}

\usepackage{graphicx,float}

\usepackage{tikz}
\usetikzlibrary{patterns}
\usepackage{minted}
\usetikzlibrary{shapes,decorations,arrows,calc,arrows.meta,fit,positioning}

 \usepackage{csquotes}

\title{Multilinear Mikhlin Multipliers with Degenerate Singularities}
\author{Hana\"e Vandanjon \\
\small Nantes Universit\'e, Laboratoire de Math\'ematiques Jean Leray, LMJL, UMR 6629, F-44000 Nantes, France  \\
\small {\em E-mail address:} {hanae.vandanjon@univ-nantes.fr} }
\date{\today}

\begin{document}
\maketitle
\begin{abstract}
We prove $L^p$ bounds for multilinear multipliers whose symbols are singular along degenerate subspaces. When the singularity has a sufficiently mild degeneracy, we show that the wavepacket analysis of the nondegenerate theory can be adapted to this setting. The admissible range of exponents depends on the geometry of the singularity.
\end{abstract}
\section{Introduction}
Let $n \geq 2$ and let $m$ be a function defined on the hyperplane 
\[ \Gamma:= \{\xi \in \mathbb{R}^n \; |\; \xi_1+\cdots+\xi_n=0 \}. \]
We consider the multilinear operator $T_m$ associated with the symbol $m$:
\[ T_m(f_1,\dots, f_{n-1})(x):= \int_{\mathbb{R}^{n-1}}m(\xi)e^{2i\pi x(\xi_1+\cdots+\xi_{n-1})} \prod_{j=1}^{n-1}\hat{f}_j(\xi_j) \; d\xi_1\cdots d \xi_{n-1},\]
where $\xi=(\xi_1, \dots, \xi_{n-1}, -(\xi_1 + \cdots + \xi_{n-1}))\in \Gamma$.
\\ It was shown in \cite{MuscaluTaoThiele2002} that the operator $T_m$ extends to a bounded operator 
\[T_m: L^{p_1} \times \cdots \times L^{p_{n-1}} \rightarrow L^{p_n'}\]
provided that: 
\begin{enumerate}
\item $m$ is a symbol singular along a nondegenerate subvector space $\Gamma' \subset \Gamma$, of dimension $k$ with 
\[ k< \frac{n}{2},\] and  satisfies the following Mikhlin estimates 
\begin{equation}\label{Mikhlin}
    |\partial^\alpha m(\xi)| \lesssim \left(\text{dist}(\xi, \Gamma') \right)^{-|\alpha|}
\end{equation}
for enough partial derivatives. 
    \item For any $j=1, \dots, n-1$
\begin{equation}\label{eq:range}
    1<p_j\leq \infty,
    \end{equation}
\begin{equation}\label{eq:Holder}
    \frac{1}{p_1}+\cdots +\frac{1}{p_n}=1,\end{equation}
\begin{equation} \label{eq:droites}
    \frac{1}{p_{i_1}} +\cdots +\frac{1}{p_{i_r}} < \frac{n-2k+r}{2},
\end{equation}
for any $1\leq i_1< \cdots < i_r \leq n$ and $r=1, \dots, n$.
\end{enumerate}
Here, $\Gamma'$ nondegenerate means that it is a graph over every choice of $k$ variables among $\xi_1, \dots, \xi_n$. Notice that since there is no condition \eqref{eq:range} on the last exponent $p_n$, its dual exponent $p_n'=\frac{p_n}{p_n-1}$ can be lower than $1$. 
\\ This theorem covers the case of the nondegenerate bilinear Hilbert transform, previously studied by Lacey and Thiele in \cite{LaceyThiele1997}, \cite{LaceyThiele1999}, for $\alpha\neq 0, 1$
\[ BHT_{1, \alpha}(f,g)(x)=p.v. \int_\mathbb{R}f(x+ t) g(x+ \alpha t)\frac{dt}{t}\]
which can be rewritten as an operator $T_m$ with symbol 
\[m(\xi)=i\pi \text{sgn}( \xi_1+\alpha \xi_2).\]
The symbol $m$ is singular along $\Gamma'_{1,\alpha}=\{\xi \in \mathbb{R}^3\;|\;  \xi_1 = -\alpha \xi_2 \}$ which is degenerate whenever $\alpha=0$ or $\alpha=1$. When $\Gamma'_{1,\alpha}$ is degenerate, the bilinear Hilbert transform reduces, up to a constant, either to a product involving the Hilbert transform or to the Hilbert transform of a product. In consequence, boundedness follows directly from Hölder's inequality together with the $L^p$ boundedness of the Hilbert transform. 
The estimates obtained by using the approach from \cite{LaceyThiele1997}, \cite{LaceyThiele1999} blow up as $\alpha$ approaches the degenerate values and they don’t include the degenerate cases. A related question is whether these estimates can be made uniform in $\alpha$. This was first achieved for weak type estimates by \cite{Thiele2002}, in the local $L^2$ range by \cite{GrafakosLi2004}, and later extended outside of the local $L^2$ range in \cite{Li2006} and to the full known range of boundedness by \cite{UralstevWarchalski2022}.
\\The result of \cite{MuscaluTaoThiele2002} was extended in \cite{DemeterPramanikThiele2010} to operators defined on $\mathbb{R}^d$, in the local $L^2$ range. In this setting, the symbol $m$ is defined on $(\mathbb{R}^d)^n$ and is singular along a $k$-dimensional subspace $\Gamma'\subset(\mathbb{R}^d)^n$. 
A new difficulty arises for operators of fractional rank, that is, when $\frac{k}{d}$ is not an integer, since $\frac{k}{d}$ can be arbitrarily close to the strict limiting value $\frac{n}{2}$. As in the one-dimensional setting, the singularity $\Gamma'\subset(\mathbb{R}^d)^n$ is assumed to be a graph over every choice of $k$ canonical variables, together with an additional hypothesis accounting for the ambient dimension $d$ and the possibility of fractional rank.
In higher dimensions $d$, when $\Gamma'$ fails to be a graph over some choices of canonical variables, the resulting degeneracies seem to have different degrees of complexity.  This phenomenon appears in particular in the study of the two-dimensional bilinear Hilbert transform in \cite{DemeterThiele2010}.
\\Notice that the Mikhlin estimate \eqref{Mikhlin} is invariant under translation in the direction of $\Gamma'$, which can be seen as a modulation invariance property on the class of operators $T_m$. Together with the translation invariance of $T_m$, this suggests the use of tools adapted to both the space and frequency symmetries, namely a wave packet decomposition. Indeed, in \cite{MuscaluTaoThiele2002}, \cite{LaceyThiele1997}, it was shown that it is enough to study the following discretized model:
\begin{equation} \label{eq:opmodel}
    T_{\mathbb{P}}(f_1, \dots, f_n)(x):= \sum_{P\in \mathbb{P}}|I_P|^{1-n/2}\left(\prod_{j=1}^{n-1}\left<f_j, \phi_{P,j}^j\right> \right)\phi_{P,n}^n(x) \end{equation}
where $\mathbb{P}$ is a finite collection of multi-tiles (multi-rectangles of area 1) $P=P_1 \times \cdots \times P_n$, with a special rank property implied from the nondegenerate assumption. This rank property encodes the geometric relations between the tiles and is the key ingredient to organize the summation over $\mathbb{P}$ and recover the information carried by each function $f_j$. In fact, without an appropriate rank property, there is no reason to believe that the operator above \eqref{eq:opmodel} is bounded. This idea produces an explicit example of a symbol $m$ in \cite{Kesler2019}, singular along the degenerate line 
\[\{(\xi, -\xi, 0) \; | \; \xi \in \mathbb{R}\}, \]
such that the operator $T_m$ does not satisfy any $L^p$ estimates in the Banach range.
\\ One motivation to study operators with degenerate symbols is provided by the operator introduced in \cite{Palsson2012}:
\[ T_\beta (f_1, f_2, f_3)(x):= p.v. \int_\mathbb{R} \left(\int_0^1 f_1(x+\alpha t)d\alpha \right)f_2(x+\beta t)f_3(x+t)\frac{dt}{t}.\]
It can be rewritten as an operator $T_{m_\beta}$ with symbol 
\[m_{\beta}(\xi_1, \xi_2, \xi_3)=\int_0^1\text{sgn}(\alpha \xi_1 +\beta \xi_2 + \xi_3)d\alpha\]
which is discontinuous along the degenerate line 
\[\{ (\xi_1, \xi_2, \xi_3, -(\xi_1+ \xi_2 + \xi_3)) \; | \; \xi_1 =0, \beta\xi_2+ \xi_3=0\}.\]
This operator is derived from Calderón's second commutator, dropping the average in $\beta$.
Without the averaging in $\alpha$, one recovers the trilinear Hilbert transform, for which no $L^p$ boundedness results are known. In particular, the result in \cite{MuscaluTaoThiele2002} does not apply since the singularity is of dimension $k=2=\frac{n}{2}$. Keeping the average over $\alpha$ inside the integral lowers the dimension of the singularity but it introduces some degeneracy and reduces the regularity of the symbol. Although the model operator in \cite{Palsson2012} differs from \eqref{eq:opmodel}, mainly because the symbol satisfies weaker regularity assumptions, it still admits an adapted rank property allowing one to prove boundedness in a smaller range of $L^p$ exponents. This suggests that the nondegeneracy assumption can be weakened at the expense of a smaller range of admissible $L^p$ exponents. 
\\ Beyond Palsson’s operator, another motivation to consider degenerate singularities arises from the study of multilinear operators along curves, as in \cite{BeneaBernicotLie}. For example, if one considers a hybrid trilinear Hilbert transform along curves given by 
\begin{equation*}
\label{eq:HTHT:1*}
\int_{\mathbb{R}} f_1(x-t^3) f_2(x+t) f_3(x-t) \frac{1}{t} dt,
\end{equation*}
an important step in the approach introduced in \cite{BeneaBernicotLie} is to prove the boundedness of multilinear anisotropic operators associated to the symbol
\[\sum_{k \in \mathbb{Z}}  \psi\big( \frac{\xi_1}{2^{3k}} \big) \psi\big( \frac{\xi_3-\xi_2}{2^k} \big).\]
Its isotropic version is the multilinear operator associated to 
\[m_1(\xi_1, \xi_2, \xi_3 ):=\sum_{k \in \mathbb{Z}}  \psi\big( \frac{\xi_1}{2^{k}} \big) \psi\big( \frac{\xi_3-\xi_2}{2^k} \big),\]
which is a Mikhlin symbol with respect to the subspace $\Gamma'_1=Span( (0, 1, 1, -2)^t)$.
\\Similarly, a Mikhlin symbol singular along $\displaystyle \Gamma'_2:=Span\{ (1, 1, -2, 0, 0, 0)^t, (0, 0, 0, 1, 1, -2)^t\}$ is given by
\[m_2(\xi_1, \ldots, \xi_5 ):=\sum_{k \in \mathbb{Z}}  \psi\big( \frac{\xi_1-\xi_2}{2^{k}} \big) \psi\big( \frac{\xi_3+2\xi_1}{2^k} \big)\psi\big( \frac{\xi_4-\xi_5}{2^k} \big).\]
An anisotropic version of this symbol is associated to the operator
\begin{equation*}
\label{eq:HTHT:2*}
\int_{\mathbb{R}} f_1(x-t^3-2t^5) f_2(x+t^3) f_3(x-t^5) f_4(x-t) f_5(x+t) \frac{1}{t} dt,
\end{equation*}
whose boundedness is, for the moment, an open problem.
\\The boundedness of the trilinear Hilbert transform, and of operators associated to symbols for which $n \leq 2k$ remains an important open problem in time-frequency analysis. By considering variants of such operators along curves, the dimension of the singularity is reduced but (anisotropic) degeneracies naturally appear. 
\\In this paper we show that, in certain situations at least, the presence of degeneracies is not a major obstruction. More precisely, the purpose of this paper is to extend the result from \cite{MuscaluTaoThiele2002} to symbols singular along certain degenerate subspaces, which in particular include the singularity considered in \cite{Palsson2012} and the ones appearing in the examples above.
\begin{definition} Let $(i_1,\dots, i_k)$ be a tuple such that $1\leq i_1< \cdots < i_k \leq n$. We say that $(i_1,\dots, i_k)$ is a nondegenerate direction if there exists a linear map $h \;: \mathbb{R}^{k} \rightarrow \Gamma$ (depending on the tuple $(i_1,\dots, i_k)$)  such that $\Gamma'$ can be parametrized via $h$ by these indices, i.e.
\[ \Gamma'= \left\{ h(\xi_{i_1}, \dots, \xi_{i_k}) \; | \; (\xi_{i_1},\dots, \xi_{i_k}) \in \mathbb{R}^k \right\}.\]
We will also describe them as good directions. 
\\ When this is not the case, we will say that $(i_1,\dots, i_k)$ is a degenerate or bad direction. 
\\ We say that a set of indices $\{i_1, \dots, i_r\}$ forms a nondegenerate set if any direction $(i_{j_1},\dots, i_{j_k})$, with $i_{j_1}, \dots, i_{j_k} \in \{i_1, \dots, i_r\}$ is nondegenerate.
\end{definition}
The assumption "$\Gamma'$ nondegenerate" in \cite{MuscaluTaoThiele2002} means that the set $\{1, \dots, n\}$ is a nondegenerate set. 
\\As an example, if we consider the singularity from the example above, $\Gamma'_1=\{ (0, \xi, \xi, -2 \xi)) \; | \; \xi\in \mathbb{R}\}$, then $(1)$ is a degenerate direction, $(2),(3)$, and $(4)$ are nondegenerate directions and $\{ 2,3,4\}$ is a nondegenerate set. 
\begin{definition}
 We say that $i \in \{ 1, \dots, n\}$ is a \emph{strongly degenerate index} if all directions $(i_1,\dots, i_k)$ with $i \in \{i_1,\dots, i_k\}$ are degenerate. We will denote by $s$ the number of \textit{strongly degenerate indices} for $\Gamma'$ and assume, if there are some, that they are the first coordinates $1, \dots,s$.
\end{definition}
\textit{Strongly degenerate indices} will be considered “bad” indices in the sense that they don't provide any information, in no direction and therefore we will in some sense leave them apart. 
\\If we again consider the same $\Gamma'_1$ as above, then $(1)$ is a strongly degenerate index. 
\\ With these definitions, we can now state our main results. 
 \begin{theorem}\label{th1}
 Let $\Gamma'$ be a $k$-dimensional subspace of $\Gamma$, and let $G\subset \{1, \dots, n\}$ be a maximal nondegenerate set. Denote $\sharp G:=n-l$ and assume that 
    \[  2k < n-l.\]
Suppose that $m$ satisfies the Mikhlin estimate \eqref{Mikhlin} relative to $\Gamma'$ for sufficiently many derivatives. Then 
    \[ T_m: L^{p_1 }\times \cdots \times L^{p_{n-1}} \rightarrow L^{p_n'}\]
    whenever \eqref{eq:range}, \eqref{eq:Holder} hold and
 \[ \frac{1}{p_{i_1}}+\cdots+\frac{1}{p_{i_r}}<\frac{n-l-2k+r}{2}\]
    for all $1\leq i_1<\dots <i_r\leq n$, $i_1, \dots, i_r \in G$ and $1 \leq r\leq n-l$.
\end{theorem}
Thus provided that we can find a large enough set of indices which parameterize the singularity in any direction generated by this set, the operator $T_m$ is bounded.
\\ One can notice the asymmetry between the constraints for Lebesgue exponents. This comes from the way we deal with the possible degenerate directions. In some sense the indices in $G$ play a nondegenerate part as in \cite{MuscaluTaoThiele2002} whereas the indices not in $G$ are left aside. The lower bound on $p_n'$ will therefore depend on whether $n\in G$ or not. Indeed if $n\in G$, then we get 
\[p_n'> \frac{1}{n-k-1/2}\]
whereas if $n \notin G$, we get 
\[p_n' > \frac{1}{n-k-1}.\]
Although we cannot apply Theorem \ref{th1} to the degenerate bilinear Hilbert transform, the same phenomenon can already be  observed in this setting. Indeed, if $H$ denotes the linear Hilbert transform, 
\[BHT_{1,0}(f,g)=H(f)g,\]
where $(1)$ is the degenerate direction, is bounded from $L^{p_1}\times L^{p_2}$ to $L^{\frac{p_1p_2}{p_1+p_2}}$ for all $p_1, p_2 \in (1, \infty)$. On the other hand,
\[BHT_{1,1}(f,g)=H(fg),\] 
for which $(3)$ is the degenerate direction, has a strictly smaller boundedness range since the linear Hilbert transform cannot be bounded below $L^1$.
\\ When we can chose $G=\{ 1, \dots, n\}$, i.e. $\Gamma'$ is nondegenerate, we recover the theorem from \cite{MuscaluTaoThiele2002}. When this is not the case, the conditions above are more restrictive than the ones in the nondegenerate setting. For example, if $m$ is a symbol singular along $\Gamma'_1$, satisfying the Mikhlin assumption \eqref{Mikhlin}, there is only one maximal set given by $G=\{2,3,4\}$, with $l=1$. In this case, we get that the operator $T_m$ is bounded from $L^{p_1}\times L^{p_2}\times L^{p_3}$ to $L^{p_4'}$ whenever $(p_1,p_2, p_3, p_4')$ satisfy \eqref{eq:Holder}, $p_1, p_2,p_3>1$, $\frac{1}{p_2}+\frac{1}{p_3}<\frac{3}{2}$ and $p_4'> \frac{2}{5}$. 
\\For the singularity $\Gamma'_1$, the choice of the set $G$ is unique. In general, however, this is not always the case. For instance, if we consider the singularity $\Gamma'_2$, one can see that $\{1,4\}$ and $\{1,5\}$ are both admissible choices for $G$. This leads to two observations.  
First, when there exist several admissible sets $G$ of sufficiently large cardinality, the above theorem can be applied using each of these choices. The corresponding boundedness results may then be interpolated to obtain a larger range of admissible $L^p$ exponents.
Second, observe that the parameter $l$ satisfies $l\geq s$ and they are not always equal, except for the case $k=1$. Indeed, for  $\Gamma'_2$, we have $l=4$ and yet $s=0$. 
In particular in this theorem we do not use the fact that some indices might belong to some nondegenerate directions and yet not be in any large enough admissible $G$.  This is why we introduce our second result, which makes use of this property. 
\\We denote  \begin{align*}
\Xi_{\Gamma'}:= \{(\alpha_{s+1},\dots, \alpha_n) \in (0,\frac{1}{2})^{n-s} \:| \: &\exists(\theta_{i_1,\dots, i_k})\in [0, 1]^{\sharp \{(i_1, \dots, i_k) \text{ good direction}\}}\: s.t. \\ & \alpha_j= \sum_{\substack{(i_1, \dots, i_k) \text{ good} \\ \exists r \text{ s.t. } i_r=j}} \theta_{i_1,\dots, i_k},  
 \sum_{(i_1, \dots, i_k) \text{ good}} \theta_{i_1, \dots i_k}=1 \} .\end{align*}
\begin{theorem}\label{th2}
    Let $\Gamma'$ be a subspace of $\Gamma$ of dimension $k$ such that
    \[2k < n-s.\] 
     Suppose that $m$ satisfies the estimate 
    \[ | \partial^\alpha m(\xi)| \lesssim \text{dist}(\xi, \Gamma ')^{-|\alpha|}\]
    for enough derivatives. Then 
        \[ T_m: L^{p_1 }\times \cdots \times L^{p_{n-1}} \rightarrow L^{p_n'}\]
        for all tuples such that \eqref{eq:range} and \eqref{eq:Holder} hold and  such that for all $j=1, \dots, s$:
    \[ p_{j}>1,\] and there exists $\alpha \in \Xi_{\Gamma'}$ such that for all $j=s+1, \dots, n$
    \begin{equation}\label{eq:alpha1}
        (1-\alpha_j)>\frac{1}{p_j}. \end{equation} 
\end{theorem}
 In particular, using \eqref{eq:Holder}, \eqref{eq:alpha1} and the fact that for $\alpha \in \Xi_{\Gamma'}$, $\alpha_j<\frac{1}{2}$, one can notice that 
\[ p_n' > \frac{1}{n-k-\frac{1}{2}}.\]
We assumed for readability purposes that the strongly degenerate indices are $1, \dots, s$ but this theorem still remains valid for arbitrary strongly degenerate indices $i_1, \dots, i_s \in \{1, \dots, n\}$. In particular if $n$ is a strongly degenerate index, then the operator $T_m$ is still bounded but with a more restrictive condition on $p_n'$, which is 
\[ p_n'> \frac{1}{n-k-1}\]
as in the previous theorem.
As we said above, this approach uses the fact that when $k>1$, an index can belong both to a degenerate direction and a nondegenerate one, which means that together with $k-1$ other indices, it can parameterize the singularity while with some other indices, it does not. In the first approach, we decided to put those indices aside and provided that the number of indices we were left with was large enough (the exact condition was $n-l>2k$), we obtained the boundedness of the operator $T_m$.
Here, we only put aside the indices which do not parameterize the singularity in any direction and provided that we have enough indices left and that the set $\Xi_{\Gamma'}$ is not empty, we get a boundedness result. This theorem allows us to deal with more singularities than Theorem \ref{th1}, an example being $\Gamma_2'$, but the boundedness now depends on whether $\Xi_{\Gamma'}$ is not empty and the boundedness range is less explicit. Notice that if $\Xi_{\Gamma'}\neq \emptyset$, Theorem \ref{th2} gives boundedness at least in the local $L^2$ range. The matter of nonemptiness of $\Xi_{\Gamma'}$ and the $L^p$ range obtained will be discussed in detail in Section 6, in particular we will make explicit the range obtained for the singularity appearing in our second example $\Gamma_2'$. In the case $k=2$, the particular structure of the degenerate directions allows us to fully characterize when $\Xi_{\Gamma'}$ is nonempty in terms of the number of nondegenerate directions. More precisely, denoting by $\mathfrak{d}$ the number of nondegenerate directions, we have the following.
\begin{proposition} \label{caracterisationsolutions_k=2}
    Let $k=2$ and $n-s\neq 6$, then  $\Xi_{\Gamma '} \neq \emptyset$ if and only if $\displaystyle \mathfrak{d}\geq 2\binom{n-s-1}{k-1}=2(n-s-1)$.
\end{proposition}
\begin{remark}
In particular this proposition implies that, in the case $k=2$, provided there are enough nondegenerate directions, and $m$ is a symbol as in Theorem \ref{th2}, then $T_m$ is a bounded operator at least in the local $L^2$ range. 
\\ We give an overview of what happens in the case $n-s=6$ and $k=2$ in Section 6. 
\end{remark}
 If the singularity $\Gamma'$ is not degenerate, Theorem \ref{th2} retrieves the result from \cite{BeneaMuscalu2021} and the range of boundedness obtained is the same as in the classical result of \cite{MuscaluTaoThiele2002}. 
\par We give a short overview of the framework of the paper. In Section 2, we give a bit more explanations on what a degeneracy is and provide a few examples to illustrate this notion. In Section 3, we roughly explain how the boundedness of the operator $T_m$ can be studied through its wave packet decomposition introduced above in \eqref{eq:opmodel}. Section 4 presents some notions and results from \cite{MuscaluTaoThiele2002} that will be necessary in this paper. Section 5 consists in the proof of Theorem \ref{th1} following the approach of \cite{MuscaluTaoThiele2002} and refining the arguments to get some local estimates more suitable to allow in the analysis some degeneracies of $\Gamma'$. In Section 6 we give the proof of Theorem \ref{th2}, this time following \cite{BeneaMuscalu2021}, as well as some criteria for the nonemptiness of $\Xi_{\Gamma'}$. Finally, in Section 7 we extend the counterexample from \cite{Kesler2019} of an unbounded symbol singular along a degenerate line in $\mathbb{R}^3$ to $\mathbb{R}^n$ for any $n \geq 3$.

\section{Different kinds of degeneracies}
In this section we detail some remarks on the definitions of degeneracies we stated above and give a short overview of the situations that can occur, by illustrating with some examples. 
\\ First notice that since $\Gamma'$ is of dimension $k$, there must exist at least one nondegenerate direction. 
\\The nondegenerate or degenerate direction property can be restated using minors of the matrix formed by a base of the vector space $\Gamma'$. Indeed if 
$\Gamma'=\mathrm{Span}(u_1,\dots,u_k)$, the direction $(i_1, \dots, i_k)$ is nondegenerate if and only if 
\[ \det\left( [u_1,\dots,u_k]_{i_1,\dots,i_k} \right)\neq0\]
 where $[u_1,\dots,u_k]_{i_1,\dots,i_k}$ is the matrix composed of the coordinates $i_1,\dots,i_k$ of the vectors $u_1,\dots, u_k$ in the canonical base.
\\ From this we notice that there is a non trivial difference between the case $k=1$, where a degenerate direction $(i)$ is a zero on the $i$-th coordinate of a spanning vector. The case $k=1$ is thus easier to handle than the $k>1$ case where a $(i_1,\dots,i_k)$ degeneracy does not necessarily  mean that these coordinates are all equal to zero. Another noticeable difference, which will also motivate the second method presented here is that when $k>1$, an index $i$ can belong to a degenerate direction and a nondegenerate one. Some indices will therefore play a more important role than others depending on how many nondegenerate directions they take part in.
\begin{example} Let us look at the three subspaces $\Gamma' \subset \Gamma$ given by:
\begin{itemize}
    \item $\Gamma'_1= \text{Span}\left(\begin{bmatrix}
    0 \\ 1 \\ 1 \\-2
\end{bmatrix}\right) $, here $n=4, \:k=1$ and $l=s=1$, 
\item $\Gamma'_2= \text{Span}\left(\begin{bmatrix}
    1 \\ 1\\ -2  \\0 \\ 0 \\ 0 
\end{bmatrix} \:, \; \begin{bmatrix}
    0  \\ 0\\ 0  \\ 1 \\ 1 \\ -2 
\end{bmatrix}\right) $, here $n=6, \: k=2$, $s=0$ and $l=4$,  
\item $\Gamma_3'= \text{Span}\left(\begin{bmatrix}0 \\ 0 \\ 1  \\ 2 \\-3  \end{bmatrix}\; , \; \begin{bmatrix}
    0\\ 1\\ 1  \\ 3 \\ -5 
\end{bmatrix}\right)$, here $n=5,\: k=2$, and $l=s=1$. 
\end{itemize}
In the first case $(1)$ is a degenerate direction and we can easily see it as the vector spanning $\Gamma_1'$ has a zero on its first component.  In the second case, if we look at the first components, they are not all equal to zero and yet the direction $(1,2)$ is degenerate as the two lines formed by the first and second components of the two vectors are linearly dependent. One can however notice that $(1,j)$ is nondegenerate for $j=4,5,6$. In the third case, the first components of the two spanning vectors are zero and this implies that $(1,j)$ is a degenerate direction for all $j=2,\dots,5$. In particular $1$ is a \textit{strongly degenerate index}. In some sense the index $1$ is participating more in the second case than in the third. 
\end{example}
 One can notice that if $i$ is a \textit{strongly degenerate index} then for any $u \in \Gamma'$, $u_i=0$. Indeed consider a matrix formed of $k$ spanning vectors $u_1, \dots,u_k \in \mathbb{R}^n$ of $\Gamma'$. Then since $\Gamma'$ is of dimension $k$ there exists at least one invertible submatrix of size $k \times k$. Assume that such a submatrix is made of the first $k$ lines $l_1,\dots, l_k$ and assume that $k+1$ is a \textit{\textit{strongly degenerate index}}. Then there exists $\lambda_1,\dots, \lambda_k$ such that $l_{k+1}= \sum_{j=1}^k \lambda_jl_j$ and on another hand, we have for $j=1,\dots, k$
\[ 0=\det \left( \begin{bmatrix} l_1 \\ \vdots \\ l_{j-1} \\l_{j+1} \\ \vdots \\ l_k \\
    l_{k+1}
\end{bmatrix}\right)=\det\left( \begin{bmatrix} l_1 \\ \vdots \\ l_{j-1} \\l_{j+1} \\ \vdots \\ l_k \\ 
    \sum_{i=1}^k \lambda_i l_i
\end{bmatrix} \right)= \lambda_j \underbrace{\det \left(\begin{bmatrix}l_1 \\ \vdots \\ l_{j-1} \\l_{j+1} \\ \vdots \\ l_k \\ 
   l_j \end{bmatrix} \right)}_{\neq 0} \]
which implies that each $\lambda_j=0$ and in particular that $l_{k+1}=0$.
\begin{definition}
An index $i$ is said to be \emph{\textit{strongly nondegenerate}} if for every direction $(i_1, \dots, i_k)$ such that $i\in \{ i_1, \dots, i_k\}$, the direction $(i_1, \dots, i_k)$ is nondegenerate.
\end{definition}
The set $G$ defined in Theorem \ref{th1} always contains all \textit{strongly nondegenerate indices}.
\\ Notice also that if there exists a \textit{strongly degenerate index}, there cannot be any \textit{strongly nondegenerate index} and vice versa.
\begin{example} We consider the singularity
\[\Gamma'_4= \text{Span}\left(\begin{bmatrix}
    1 \\-1\\0 \\0 \\ 0\\0 
\end{bmatrix}, \begin{bmatrix}
    1  \\ 1 \\ 1\\1 \\1 \\ -5
\end{bmatrix}\right). \]
In this case, one can see that $1$ and $2$ are \textit{strongly nondegenerate indices} as any directions $(1,i)$ and $(2,i)$ are nondegenerate. 
\end{example}
\begin{remark} \label{remarquedeg}
We point out that in the case $k=2$, the degenerate directions have a specific structure. Indeed,  if $(i,j)$ and $(i,j')$ are two degenerate directions and $i$ is not a \textit{strongly degenerate index},  then $(j,j')$ is also degenerate. This is simply because the corresponding lines, in the matrix formed by the spanning vectors, are linearly dependent. 
\\ In the general case $k>2$, the set of $k\times k$ minors of the matrix has a particular algebraic structure that is difficult to take into account.  
\end{remark}
\begin{example} We consider two examples in the case $k=3$ to illustrate our remark: 
\[\Gamma'_5= \text{Span}\left(\begin{bmatrix}
    1 \\0 \\0 \\1 \\1 \\1 \\-4
\end{bmatrix}, \begin{bmatrix}
    0 \\ 1  \\ 0 \\ 0 \\1 \\2 \\ -4
\end{bmatrix},\begin{bmatrix}
    0 \\ 0 \\1 \\ 0 \\ 0 \\ 3 \\-4
\end{bmatrix}\right), \quad
 \Gamma'_6 = \text{Span}\left(\begin{bmatrix}
    1 \\0 \\0 \\1 \\1 \\1 \\-4
\end{bmatrix}, \begin{bmatrix}
    0 \\ 1  \\ 0 \\ 2 \\1 \\2 \\ -6
\end{bmatrix},\begin{bmatrix}
 0 \\ 0 \\1 \\ 0 \\ 0 \\ 3 \\-4
\end{bmatrix}\right). \]
In the case of $\Gamma_5'$, one can easily see that $l_1=l_4$ which implies that any direction $(1,4, j)$ is degenerate. On another hand, in the case of $\Gamma_6'$, one can also notice that $(1,2,5)$ and $(2,4,5)$ are degenerate directions since $l_5=l_1+l_2$ and $l_4=l_1+2l_2$. This implies that the directions $(1, 2, 4)$ and $(1,4,5)$ are degenerate. In some sense the degenerate directions $(1,4,j)$ in the case of $\Gamma_5'$ are "stronger" than the same ones in the case of $\Gamma_6'$. 
\end{example}
\section{Model form}
In this section, we briefly discretize our operator via a wave packet decomposition into a model operator in the spirit of \cite{MuscaluTaoThiele2002} and \cite{ThieleBook}. We first introduce a few definitions.
\begin{definition}
A tile $P=I_P\times \omega_P$ is a rectangle in $\mathbb{R}^2$ where $I_P$ is a dyadic interval and $\omega_P$ is a shifted dyadic intervals such that $|I_P|=|\omega_P|^{-1}$. $I_P$ is called the space interval and $\omega_P$ the frequency interval. 
\\ A multitile is a tuple $P=(P_1,\dots, P_n)$ such that each $P_j$ is a tile with same space interval, $I_P:=I_{P_1}=\cdots =I_{P_n}$. 
\end{definition}
\begin{definition}
A wave packet $\phi_P$ adapted to a tile $P$ is a smooth function such that
\[ \text{supp}(\hat{\phi_P})\subset \omega_P, \]
and $\phi_P$ is a bump function adapted to $I_P$, in the sense that 
\[ | (\phi_P (x) e^{-2i\pi c(\omega)x})^{(k)}(x)|\lesssim_{k, M} |I_P|^{-\frac{1}{2}-k} \left( 1 +\frac{\text{dist}(x, I_P)}{|I_P|}\right)^{-M}\]
for enough derivatives $k$ and a large enough $M$. 
\end{definition}
Notice that with this choice of normalization, the wave packets satisfy 
\[ \|\phi_P\|_2 \lesssim 1,\]
we will say that they are $L^2$-normalized. 
\\ Let us first notice that $\mathbb{R}^n\backslash \Gamma'$ can be covered by Whitney  cubes $Q$, in other words, $$\displaystyle \mathbb{R}^n\backslash \Gamma'=\bigcup_{Q\in \cal{Q}} \frac{9}{10}Q,$$ where for each $Q \in \cal{Q}$, $Q$ is a shifted dyadic cube such that 
\begin{equation} \label{eq:Whitneycubes}
    \mathrm{dist}(Q, \Gamma')\sim C_0\, \mathrm{diam}(Q)
\end{equation}
with $C_0$ a large constant and where the collection $\left(\frac{9}{10}Q\right)_{Q\in\cal{Q}}$ finitely overlaps. Therefore we can split our symbol $m$ into:
\[ m = \sum_{Q\in \mathcal{Q}}m_Q\]
where each $m_Q$ is supported on $Q$. Notice here that since $m$ is supported on $\Gamma$, we only consider the cubes $Q$ such that $Q\cap \Gamma \neq \emptyset$. By the Whitney property and the fast decay outside of $\Gamma'$ of the symbol $m$, we have for any $\xi \in Q$: 
\[ |\partial^\alpha m(\xi)|\lesssim \mathrm{diam}(Q)^{-|\alpha|}.\]
Now each $m_Q$ can be periodized, tensorized via a Fourier decomposition and then relocalized so that 
\[ m_Q(\xi)=\sum_{m\in\mathbb{Z}^n}c_m^Q \prod_{j=1}^n\psi_{Q_j}^{j,m}(\xi_j)\]
 where $Q=Q_1 \times\cdots \times Q_n$, $\psi_{Q_j}^{j,m}$ is a bump function adapted to $Q_j$ and $(c_m^Q)_{m}$ is a fast decaying sequence uniformly in $Q$. With the fast decaying property of $(c_m^Q)_m$, it is enough to consider only one term in the last summand.  
 \\ We are now left with a symbol of the form 
 \[\sum_{Q\in \mathcal{Q}} \prod_{j=1}^n \psi_{Q_j}^{j}\]
 where $Q$ is part of a collection $\cal{Q}$ of shifted dyadic cubes, with a Whitney property, such that $Q\cap \Gamma\neq \emptyset$. We can also add the assumption that $\cal{Q}$ is sparse in the sense that: for $Q,Q' \in \cal{Q}$ $|Q|<|Q'|$, then there exists a large constant $C$ such that $C|Q|<|Q'|$ and if $|Q|=|Q'|$, and for some $j=1, \dots, n$, with $Q_j\neq Q'_j$, then $CQ_j\cap CQ_j'= \emptyset$. 
 \\  We now consider the multilinear form associated with one term of the summand $\prod_j\psi_{Q_j}^j$:
 \begin{align*}
     \Lambda_Q(f_1,\dots,f_n)&= \int_\Gamma \prod_{j=1}^n\psi_{Q_j}^j(\xi_j)\hat{f}_j(\xi_j) d\xi \\
     &= \sum_{t \in \mathbb{Z}} \underbrace{\int_{\xi_1 +\cdots +\xi_n=C_nt|Q_1|}\prod_{j=1}^n\psi_{Q_j}^j(\xi_j)\hat{f}_j(\xi_j) d\xi}_{F(t)} \\
    &= \sum_{t \in \mathbb{Z}} \hat{F}(t) \\
    &= C_n^{-1} \sum_{t \in \mathbb{Z}}\int_{\mathbb{R}^n} |Q_1|^{-1} \prod_{j=1}^n \hat{f}_j(\xi_j) \psi_{Q_j}^j(\xi_j)e^{-2i\pi t \xi_j/(|Q_j|C_n)}d\xi \\
    &=C_n^{-1}\sum_{t \in \mathbb{Z}}  |Q_1|^{-1}\prod_{j=1}^n \left<\hat{f}_j, \tilde{\psi}_{Q_j,t}^j \right> \\
    &=C_n^{-1}\sum_{t \in \mathbb{Z}}  |Q_1|^{-1+\frac{n}{2}} \prod_{j=1}^n \left<f_j, \check{\psi}_{Q_j,t}^j \right> \\
\end{align*}
where $ \phi_{Q_j}^j=\check{\psi}_{Q_j,t}^j$ is a $L^2$ normalized wave packet adapted to a tile $P_j:=I_t \times Q_j$, with $I_t$ dyadic interval depending on $t \in \mathbb{Z}$. Here we used that since we are only considering cubes $Q$ such that $Q \cap \Gamma \neq \emptyset$, we have for all $t\in \mathbb{Z}^*$, that $Q \cap \{ (\xi_1, \dots, \xi_n) \; | \; \xi_1+\cdots +\xi_n=C_nt|Q_1| \}=\emptyset$ and applied the Poisson summation formula. As the constant $C_n$ only depends on the dimension and doesn't play any part, we will omit it in the following. We will also ommit the $t$ dependancy for readability purposes. 
\\ We will now denote $P:=(P_1, \dots,P_n)$, $I_P:=I_t$, $\omega_{P_j}:=Q_j$ and $Q_P=(Q_1,\dots, Q_n)$. 
\\ To sum up, we are considering a multilinear form of the type 
\[ \Lambda_\mathbb{P}(f_1, \dots, f_n)=\sum_{P\in \mathbb{P}}|I_P|^{1-\frac{n}{2}} \prod_{j=1}^n \left< f_j, \phi_{P_j}^j\right>\]
and the associated operator 
\[ T_{\mathbb{P}}(f_1,\dots ,f_{n-1})(x)= \sum_{P \in \mathbb{P}}|I_P|^{1-\frac{n}{2}} \left(\prod_{j=1}^{n-1}\left<f_j,\phi_{P_j}^j \right>\right)\phi_{P_n}^n(x)\]
where $\mathbb{P}$ is a collection of multitiles satisfying some specific properties that we will explain in the following. For a more detailed discussion, we refer the reader to \cite{MuscaluTaoThiele2002} and \cite{DemeterPramanikThiele2010}.
\\ \begin{itemize}
    \item ($p_1$) First we may assume that the collection $\mathbb{P}$ is finite as the reasoning we use allows a limiting argument.
    \item ($p_2$) (Whitney) For any $P \in \mathbb{P}$, $\text{dist}(Q_P, \Gamma') \sim C_0|Q_P|$.
    \item ($p_3$) (scale separation) If $P$, $P' \in \mathbb{P}$, with $|Q_P|<|Q_{P'}|$, then there exists a large constant $C$ such that $C|Q_P|<|Q_{P'}|$.
    \item ($p_4$) (distance separation) If $P$, $P' \in \mathbb{P}$, with $|Q_P|=|Q_{P'}|$,  and if for some $j=1, \dots,n $, $\omega_{P,j}\neq \omega_{P',j}$ then  $C\omega_{P,j} \cap C\omega_{P',j}= \emptyset$.
\end{itemize}
The two last properties are exactly the sparseness assumption on the Whitney cubes.  
\\ As in \cite{MuscaluTaoThiele2002}, the collection $\mathbb{P}$ will have a sort of rank property. To make it clearer, we first give some notations on tiles. 
\begin{definition} Let $P$, $P'$ be two tiles, then we denote 
\begin{itemize}
    \item $P<P'$ if $I_P\subsetneq I_{P'}$ and $\omega_{P'} \subset 3 \omega_{P}$,
    \item $P\leq P'$ if $P<P'$ or $P'=P$,
    \item $P \lesssim P'$ if $I_P \subset I_{P'}$ and $\omega_{P'}\subset C_1C_0 \omega_{P}$, 
    \item $P \lesssim' P'$ if $P\lesssim P'$ and $P\not \leq P'$.
    \item  $P \lesssim \lesssim P'$ if $I_P \subset I_{P'}$ and $\omega_{P'} \subset C_2 \omega_{P}$.
\end{itemize}
\end{definition}
Here $C\gg C_2 \gg C_1C_0 \gg C_0 \gg 1$. 
This ordering is the same one as in \cite{MuscaluTaoThiele2002}, and it allows us with the above properties to get a very similar result on our collection $\mathbb{P}$. The only difference is that we have to deal with the degenerate directions. 
\begin{lemma} \label{lemmarank}
Let $(i_1, \dots, i_k)$ be a nondegenerate direction. 
\begin{itemize}
    \item ($p_5$) (rank) Let $P$, $P' \in \mathbb{P}$, such that for all $r=1,\dots k$, $P_{i_r}'=P_{i_r}$, then $P=P'$. 
    \item ($p_6$) (overlapping) Let $P, P'\in \mathbb{P}$ such that for all $r=1, \dots, k$, $P_{i_r}' \leq P_{i_r}$, then 
 $ P_i' \lesssim P_i $ for all $i=1, \dots, n$.
    \item ($p_7$) (lacunary) Furthermore, under the same conditions as ($p_6$), if $|I_P|\ll |I_{P'}|$, then there exists at least two indices $j$ and $j' \in \{1, \dots, n\}$ such that
\[ P_j' \lesssim' P_j, \quad \text{and} \quad P_{j'}'\lesssim' P_{j'}.\]
\end{itemize}
\end{lemma}
The proof is similar as the one in \cite{MuscaluTaoThiele2002} as it only relies on the fact that a direction $(i_1, \dots, i_k)$, nondegenerate in our case, parametrizes the singularity and then uses the Whitney property of the multitiles. 
\\ Notice that an analogue of ($p_6$) holds for $\lesssim$. Indeed if $P, P'\in \mathbb{P}$ are such that for all $r=1, \dots, k$, $P_{i_r}'\lesssim P_{i_r}$, then $P_{i}' \lesssim \lesssim P_{i}$ for all $i$. 
\\ When $\Gamma'$ is nondegenerate, ($p_5$),($p_6$) and $(p_7)$ hold for any direction $(i_1, \dots, i_k)$. In particular, when ($p_5$) holds for any direction $(i_1, \dots, i_k)$, we say that $\mathbb{P}$ is of rank $k$.
\\When $\Gamma'$ is degenerate, ($p_5$) does not hold anymore for all directions but only for the nondegenerate ones, and so $\mathbb{P}$ is not of rank $k$. However, under the assumptions of Theorem \ref{th1}, we can still have a similar rank $k$ property if we only consider the indices in $G$. We propose a definition of rank to try to match this new property, that appeared in \cite{Palsson2012}.
\begin{definition}
We say that a collection $\mathbb{P}$ is of rank $k$ with respect to $\{ i_1, \dots, i_{n-l}\}$ if the collection \[\tilde{\mathbb{P}}=\{ \tilde{P}=(\tilde{P}_{i_1} ,\dots,\tilde{P}_{i_{n-l}})\;| \;\exists P \in \mathbb{P},\; P_{i_j}=\tilde{P}_{i_j}\}\] is of rank $k$.
\end{definition}
\begin{remark}
In the case $k=1$, the nondegenerate directions are all the indices $i>s$, therefore the collection $\mathbb{P}$ is of rank $1$ with respect to $\{s+1,\dots, n\}$. 
\\ In the case $k>1$, it is less clear as directions are not indices and an index can belong to degenerate and nondegenerate directions. If the singularity is as in Theorem \ref{th1}, then the collection $\mathbb{P}$ is of rank $k$ with respect to $G$. 
\end{remark}
Consequently, the operator $T_m$ can be seen as a superposition of discrete model operators. It therefore suffices to establish bounds for these models. In the following, we will focus on $T_\mathbb{P}$, or, equivalently, on its associated multilinear form $\Lambda_{\mathbb{P}}$, where $\mathbb{P}$ is a finite sparse collection of multitiles satisfying Lemma \ref{lemmarank}. 
\section{Trees, size, energy}
In this section we present tools from \cite{MuscaluTaoThiele2002} that will be needed in the next sections. We begin by defining our key notions.  
\begin{definition} Let $i, j$ be integers. 
    A subcollection $T \subset \mathbb{P}$ is said to be an $i$-tree, if there exists a multitile $P_T$ called the top of the tree such that for all $P\in T$,
    \[P_i\lesssim P_{T,i}.\]
    We denote $I_T:=I_{P_T}$, the space interval of $P_T$.
\\ A $j$-tree is said to be $j$-overlapping if for all $P\in T$, $P_j\leq P_{T,j}$.
\\ It is said to be $j$-lacunary if for all $P\in T$, $P_j\lesssim' P_{T,j}$.
\end{definition}
Notice that if the collection is of rank $1$, then an $i$-overlapping tree, $T$, is $j$-lacunary for at least two indices $j\neq i$. Furthermore, the rank 1 property implies that $T$ is a $j$-tree for every $j$.
\\ When $k>1$ and the collection is of rank $k$, this property does not hold anymore. Nevertheless, in this case an $i$-overlapping tree is a collection of rank $k-1$. Thus, if a collection was an $i$-overlapping tree for $k$ distinct indices, then it would preserve the tree structure in every index and still enjoy two lacunary indices. This leads us to define $(i_1,\dots, i_k)$-trees.
\begin{definition} Let $T \subset \mathbb{P}$. We say that $T$ is a $(i_1,\dots, i_k)$ tree if for each $j=i_1, \dots, i_k$ the collection $T_j=\left\{P_j \;|\; P\in T \right\}$ is a tree in the sense that there exists $I_{T_j}\times \omega_{T_j}$ a tile such that $P_j \lesssim I_{T_j}\times \omega_{T_j} $ for all $P \in T$. 
\\ We say that $T$ is a $(i_1, \dots, i_k)$-overlapping tree whenever it is $i_j$-overlapping for each $j=1, \dots, k$.
\end{definition}
When the collection $\mathbb{P}$ is of rank $k$, any $(i_1,\dots, i_k)$-overlapping tree is a $j$-tree for every $j$ and it is lacunary for at least two indices. 
\\ Since we consider singularities with degeneracies, the collection $\mathbb{P}$ is not of rank $k$ anymore and thus the previous implication does not hold. However, if $(i_1, \dots, i_k)$ is a good direction, i.e. a nondegenerate direction, then by Lemma \ref{lemmarank}, if $T$ is a $(i_1, \dots, i_k)$-overlapping tree, it is a $j$-tree for any $j=1, \dots,n$, and has at least two indices for which it is lacunary. Because of the sparseness assumption and the Whitney property, restricting the summation to a $(i_1, \dots, i_k)$-tree when $(i_1, \dots, i_k)$ is a nondegenerate direction can be seen as summing over dyadic intervals. 
\\ We will therefore only consider trees on indices that are not \textit{strongly degenerate} as these indices dot not belong to any nondegenerate direction and thus trees in these indices will provide no additional information on the collection.
\\ We recall that $\{1, \dots, s\}$ are strongly degenerate indices and therefore overlapping trees in one of these indices do not provide any information. Thus, from now on, we will only consider $j$-overlapping trees for $j \in \{ s+1, \dots, n\}$.
\begin{definition}
    Let $j \in \left\{1,\dots,n\right\}$, we call $j$-size of a collection $\mathbb{P}$, the quantity 
    \[ \mathrm{size}_{\mathbb{P},j}(f_j)=\sup_T\left( \frac{1}{|I_T|} \sum_{\substack{P_j \in T(j)}}|\langle f_j, \phi_{P_j}^j\rangle|^2\right)^{1/2}\]
    where $T$ ranges over all $j$-lacunary trees that are also $i$-overlapping for an index $i\in \{s+1, \dots, n\}$ and $T(j):=\{P_j \; | \; \exists \tilde{P} \in T \; s.t. \; \tilde{P}_j=P_j\}$ is the projection on the $j$-th coordinate of $T$.
\end{definition}
\begin{remark}
     Notice that the size depends on the sequence $(\phi_{P_j}^j)_{P \in \mathbb{P}}$ but we will omit this in the notation as it can be bounded uniformly. 
   \\ We have by definition 
    \begin{equation} \label{eq:suplessthansize}
    \sup_{P\in \mathbb{P}}|I_P|^{-1/2} |\langle f_j, \phi_{P_j}^j\rangle| \leq \mathrm{size}_{\mathbb{P},j}(f_j).\end{equation} 
   \\Traditionally, the size is a quantity that only sees what happens in the $j$ variable, and it can be seen as a function on the projection on the $j$ index of the collection $\mathbb{P}$. Here we however ask for a little more, that is that the trees we consider are $i$-overlapping trees in another index. This does not change much in the reasoning but helps simplify some arguments from \cite{MuscaluTaoThiele2002}. 
\end{remark}
We now present the key lemma, that allows us to get estimates on a tree. 
\begin{lemma} \label{lemmaprod} Let $T$ be an $(i_1, \dots, i_k)$-tree for $(i_1, \dots, i_k)$ a nondegenerate direction such that $T$ is lacunary for two indices, then: 
    \[ |\Lambda_T(f_1,\dots,f_n)|\leq |I_T| \prod_{j=1}^n \mathrm{size}_{T,j}(f_j).\]
\end{lemma}
In particular if $k=1$, $T$ is a $j$-overlapping tree for $(j)$ a nondegenerate direction, the above bound holds. 
\\ Since there are two lacunary indices, this is obtained by using the bound \eqref{eq:suplessthansize} and Cauchy Schwarz inequality. 
\\ However, we need the assumption that $T$ is lacunary in two indices. We therefore introduce a lemma which gives a sufficient condition to get two lacunary indices. 
\begin{lemma} \label{lemmafinite}
   Let $T$ be a $j$-tree for every $j=s+1, \dots, n$. Then $T$
can be split into finitely many subtrees with at least two lacunary indices.
\end{lemma}
\begin{proof}
 For each index $i=s+1, \dots, n$, $T$ can be split into a lacunary and an overlapping subtree, we denote them $T_i^{\text{lac}}$ and $T_i^{\text{ov}}$. Now 
 \[T=\bigcup_{i=s+1}^n \bigcup_{t_i=\text{ov}, \; t_i =\text{lac}} \left( T_{s+1}^{t_{s+1}} \cap \cdots \cap T_n^{t_n} \right).\]
 It is already clear that as soon as there exists two $t_i$ such that $t_i=\text{lac}$, $T_{s+1}^{t_{s+1}} \cap \cdots T_n^{t_n}$ has at least two lacunary indices.
 \\ $T_{s+1}^{\text{ov}}\cap \cdots \cap T_n^{\text{ov}}$ is a $(i_1, \dots, i_k)$-overlapping tree for any nondegenerate direction and therefore, by property ($p_7$) from Lemma \ref{lemmarank}, it has at least two lacunary indices. 
 \\ If now there is only one index, assume $n$, such that $t_n=\text{lac}$, then, we want to show that there exists a second index $i\neq n$ for which $t_i=\text{lac}$. It is enough to prove that in $\{s+1, \dots, n-1\}$, there exists a nondegenerate direction, as we can then conclude as above. 
 Assume by contradiction that for every  direction $(i_1, \dots, i_k)$, such that $\{i_1, \dots, i_k\} \subset \{s+1, \dots, n-1\}$, $(i_1, \dots, i_k)$ is degenerate. Then since $n$ is not a \textit{strongly degenerate index}, there exists a nondegenerate direction $(i_1, \dots, i_{k-1},n)$, with $i_1\geq s+1$. Thus, if we denote by $l_j$ the lines in the matrix formed by the spanning vectors of $\Gamma'$, we get: 
 \[ 0 \neq \text{det} \left( \begin{bmatrix} l_{i_1} \\ \vdots \\ l_{i_{k-1}} \\l_n \end{bmatrix}\right)=\text{det} \left( \begin{bmatrix}  l_{i_1} \\ \vdots \\ l_{i_{k-1}}  \\ -\sum_{j=s+1}^{n-1} l_j \end{bmatrix} \right)=0,\]
which is absurd. Here we used that since $\Gamma'\subset \Gamma$, $\displaystyle 0=\sum_{j=1}^n l_j =\sum_{j=s+1}^n l_j$, by the discussion in Section 2.
\\ This proves that there exists a nondegenerate direction in $\{s+1, \dots, n-1\}$ which is enough to conclude. 
\end{proof}
If we can estimate each size and estimate the sum over $|I_T|$ for some trees that cover our collection $\mathbb{P}$, then with Lemma \ref{lemmaprod} and \ref{lemmafinite} we would be able to get an estimate for $\Lambda_{\mathbb{P}}$. We state the results from \cite{MuscaluTaoThiele2002} to do so. 
\\ We first deal with estimates on the size.
\begin{lemma} \label{sizeestimate} Let $\mathbb{P}$ be a finite sparse collection of multitiles, $j=1,\dots, n$, then we have: 
     \[ \mathrm{size}_{\mathbb{P}, j}(f_j) \lesssim \sup_{P\in \mathbb{P}} \frac{1}{|I_P|}\|f_j \tilde{\chi}_{I_P}^N\|_1.\]
\end{lemma}
Notice that here the rank of the collection does not matter as the size is only considering the projections onto the $j$ coordinate. Furthermore in our definition of size, the supremum is taken over a more restrictive set of trees and therefore the estimate still holds. 
We now focus on estimating $\sum_{T}|I_T|$ as $T$ ranges over a specific collection that will soon be detailed. We only consider $j$ trees, for $j\in \left\{ s+1,\dots,n\right\}$ as we said earlier that they were the only directions we will need to focus on. 
\begin{definition}
   Let $j$ be in $\{s+1, \dots, n\}$. Let $T_1,\dots, T_m$ be $j$-lacunary trees, and respectively $i-1$,\dots, $i_m$-overlapping for indices $s<i_1, \dots, i_m \neq j$. We say that $T_1,\dots, T_m$ is a chain of strongly $j$-disjoint trees if:
   \begin{itemize}
       \item $P_j\neq P_j'$ if $P \in T_{l_1}$ and $P'\in T_{l_2}$, $l_1\neq l_2$;
       \item if $l_1\neq l_2$, $P\in T_{l_1}$, $P'\in T_{l_2}$ are such that $2 \omega_{P_j}\cap 2 \omega_{P'_j}\neq \emptyset$, then if $|\omega_{P_j}|<|\omega_{P_j'}|$, $I_{T_{l_1}}\cap I_{P'}=\emptyset $ and if $|\omega_{P_j}|>|\omega_{P_j'}|$, then $I_{T_{l_2}}\cap I_{P}=\emptyset $;
       \item if $l_1<l_2$, $P\in T_{l_1}$, $P'\in T_{l_2}$, $2 \omega_{P_j}\cap 2 \omega_{P_j'}\neq \emptyset$, then if $|\omega_{P_j}|=|\omega_{P_j'}|$, one has $I_{P'}\cap I_{T_{l_1}}=\emptyset$.
   \end{itemize}
\end{definition}
\begin{remark}
       If $\mathbb{P}$ is a collection of $j$-disjoint multitiles (i.e. the $j$-tiles, $P_j$, are disjoint), then it can be seen as a chain of strongly $j$ disjoint trees where each tree is a single multitile.
\end{remark}
\begin{definition}
   Let $j\in \left\{ s+1,\dots,n\right\}$, we call $j$-energy of the collection $\mathbb{P}$, the quantity:
   \[ \mathrm{energy}_j(\mathbb{P}):=\sup_{m\in \mathbb{Z}}\sup_{\mathbb{T}} 2^m \left( \sum_{T\in \mathbb{T}} |I_T| \right)^{1/2}\]
   where $\mathbb{T}\subset \mathbb{P}$ ranges over all chains of $j$-strongly disjoint trees such that for all $T\in \mathbb{T}$,
   \[ \left(\sum_{\substack{P_j \in T(j)}} |\langle f_j, \phi_{P_j}^j \rangle |^2\right)^{1/2} \geq 2^m|I_T|^{1/2}\]
   and for all subtrees $T' \subset T$, we have:
   \[\left(\sum_{\substack{P_j \in T'(j)}} |\langle f_j, \phi_{P_j}^j \rangle |^2\right)^{1/2} \leq 2^{m+1}|I_{T'}|^{1/2}.\]
\end{definition}
The energy quantity depends on the functions $f_j$ and the sequences $(\phi_{P_j}^j)_{P\in \mathbb{P}}$ but we allow ourselves to forget about them in the notation as they are fixed. 
\\ In the case of the bilinear Hilbert transform and more generally when $k=1$, the energy can be used to obtain a size-energy estimate of the multilinear form $\Lambda_{\mathbb{P}}$, see for example Proposition 6.12 from \cite{MuscaluSchlag2013}. Although the corresponding estimate is no longer available when $k>1$, the energy quantity remains useful in the inductive argument, where it is used in order to establish a less explicit estimate required in the proof of Theorem \ref{th1}.
\\ As in the size definition, we again consider $j$-lacunary trees that are also overlapping in another index. Again as this is a more restrictive condition, it does not impact the estimates and we get the following result: 
\begin{proposition} \label{energielocalisee}
   Assume that there exists a dyadic interval $I_0$ such that $I_P\subset I_0$ for all $P \in \mathbb{P}$ and let $j\in \left\{s+1,\dots,n\right\}$, then 
    \[ \mathrm{energy}_j(\mathbb{P})\lesssim\|f_j\tilde{\chi}_{I_0}^{M/2}\|_2.\]
\end{proposition}
This is obtained in a different form in \cite{MuscaluTaoThiele2002} and detailed in \cite{BeneaMuscalu2021} via a careful Bessel type argument making use of the quasi orthogonality of the wave packets whenever the tiles are in two different trees that are strongly disjoint. 
\begin{remark}
In the following section, we will make the additional assumption that there exist $\lambda>0$ and $a_j>0$ such that:
\[\|f_j\|_2 \lesssim a_j,\]
\[ \|f_j \tilde{\chi}_{I_0}^{M/2}\|_2 ^2\lesssim \lambda|I_0| a_j^2.\]
Combining these with Proposition \ref{energielocalisee}, we then  get the energy estimate: 
\begin{equation} \label{eq:energielocalisee}\mathrm{energy}_j^2(\mathbb{P})\lesssim \min(a_j^2,a_j^2 \lambda |I_0|) \lesssim a_j^2(1+\lambda)\min(1, |I_0|) \end{equation}
where for the first case we estimated $\|f_j\tilde{\chi_{I_0}^{M/2}}\|_2$ by $\|f_j\|_2$ and we used the assumption above for the second estimate. 
Later on we will use this particular bound and we will split our collection of multitiles depending on such $\lambda$.
\end{remark}

We are now left with splitting the collection $\mathbb{P}$ into levels of size, to be able to estimate $\sum_{I_T}|I_T|$ by the energy of the sequence. For readability purposes we denote:
\[ S_j=\mathrm{size}_{\mathbb{P}, j}(f_j), \;\: j=1,\dots,n\]
\[ E_j=\mathrm{energy}_j(\mathbb{P}), \;\: j=s+1,\dots,n.\]
\begin{lemma}\label{lemmedec1}
    Let $s+1 \leq j\leq n$ and $\mathbb{P}'\subset \mathbb{P}$ such that:
    \[ \mathrm{size}_{\mathbb{P}', j}(f_j) \leq 2^{-m}E_j\]
    then one can split $\mathbb{P}'=\mathbb{P}''\cup \mathbb{P}'''$ such that  
    \[  \mathrm{size}_{\mathbb{P}''',j}(f_j) \leq 2^{-m-1}E_j\]
    and $ \displaystyle\mathbb{P}''=\bigcup_{T \in \mathbb{T}}T$ where $\mathbb{T}$ is a collection of $j$-trees, which are overlapping in an index $i\in \{s+1, \dots, n\}$ such that: 
    \[ \sum_{T\in \mathbb{T}}|I_T| \lesssim 2^{2m}.\]
\end{lemma}
This is again a result from \cite{MuscaluTaoThiele2002}, in a different form.
The idea is essentially to remove the trees that have “large” size from $\mathbb{P}'$ and put them in $\mathbb{P}''$. Since we want to be able to estimate $\sum_{T \in \mathbb{P}''} |I_T|$, we try to construct a chain of strongly $j$-disjoint trees. 
\\ Usually, a tree in $\mathbb{T}$ is either $j$-lacunary or $j$-overlapping. Here, we ask that even the $j$-lacunary trees in $\mathbb{T}$ are overlapping in another direction. This will be important in the proof of Theorem \ref{th1}, as in this case any tree $T \in \mathbb{T}$ is of rank $k-1$. 
\\ By iterating the last lemma we get the following proposition. 
\begin{lemma} \label{lemmedec2}
   Let $s+1\leq j\leq n$, then there exists $m_0 \in \mathbb{Z}$ such that:
   \[2^{-m_0}E_j \geq S_j \geq 2^{-m_0-1}E_j \]
Furthermore one can decompose 
   \[ \mathbb{P}=\bigcup_{m \geq m_0} \mathbb{P}^{j}_m\]
   such that when $\mathbb{P}_{m}^j\neq \emptyset$: 
   \[ 2^{-m-1}E_j\leq \mathrm{size}_{\mathbb{P}_m^j, j}(f_j) \leq  \min (2^{-m}E_j,S_j)\]
   and one can cover $\displaystyle \mathbb{P}_m^j=\bigcup_{T\in \mathbb{T}_m^j} T$ where $\mathbb{T}_m^j$ is a collection of overlapping trees such that 
   \[ \sum_{T\in \mathbb{T}_m^j}|I_T| \lesssim 2^{2m}.\]
\end{lemma}
\begin{remark} Notice that since we are considering $m\geq m_0$, we have: 
\[ 2^{-m} E_j \leq 2^{-m_0}E_j \leq 2S_j.\]
We will later on ignore the factor 2 as it doesn't impact the reasoning.
\end{remark}
For some technical purposes we state a similar yet more precise result which will be used in the proof of Theorem \ref{th2}.
\begin{lemma} \label{lemmedec3}
   Let $s+1\leq j\leq n$, then there exists $m_0 \in \mathbb{Z}$ such that:
   \[2^{-m_0} \geq S_j \geq 2^{-m_0-1} \]
Furthermore one can decompose 
   \[ \mathbb{P}=\bigcup_{m \geq m_0} \mathbb{P}^{j}_m\]
   such that  whenever $\mathbb{P}_m^j\neq \emptyset$
   \[ \mathrm{size}_{\mathbb{P}_m^j, j}(f_j) \sim 2^{-m}\]
   and one can cover $\displaystyle \mathbb{P}_m^j=\bigcup_{T\in \mathbb{T}_m^j} T$ where $\mathbb{T}_m^j$ is a collection of trees containing a chain of strongly $j$-disjoint trees $T_1, \dots, T_M$ such that for $i=1, \dots, M$: 
\[\mathrm{size}_{T_i,j}(f_j) \sim 2^{-m},\]
and
\[ \sum_{T \in \mathbb{T}_{m}^j} |I_{T}| \lesssim \sum_{i=1}^M |I_{T_i}|.\]
\end{lemma}
We now have all the tools in hand to prove Theorems \ref{th1} and \ref{th2}.

\section{$L^p$ bounds for singularities with a large nondegenerate set of indices}
We are now ready to prove Theorem \ref{th1}, which will be deduced from some interpolation results. We first recall some definitions for interpolation and explain how the main result is obtained. We then detail a key estimate which is used in the proof of the main result. 
Before diving more closely into the proof, we first recall the context of Theorem \ref{th1}. We consider a symbol $m \in \mathcal{C}^\alpha(\Gamma \backslash \Gamma')$,
for enough $\alpha$ with $\Gamma' $ a subvector space of dimension $k$, such that there exists $l \in \mathbb{N}$ with $n-l> 2k$ and $G \subset \{1, \dots, n\}$, with $\sharp G=n-l$, such that any direction $(i_1, \dots, i_k)$ with $i_1, \dots, i_k \in G$ is nondegenerate. 
\\ For readability purposes, we will assume that $G=\{l+1, \dots, n\}$. 
The only part where this choice of $G$ has an impact is the interpolation as the results will depend on whether the last index $n\in G$ or not. We will detail it whenever it is necessary.  
\\ In particular with our choice of $G$, it implies that the collection $\mathbb{P}$ for the discretized model is of rank $k$ with respect to $\{ l+1, \dots, n\}$. 
\subsection{Short reminders on interpolation}
We state some multilinear interpolation results from \cite{MuscaluTaoThiele2002} with the small difference being that the domains we interpolate between are not symmetric since we have to deal with degenerate and nondegenerate directions. 
We first begin with some definitions. 
\begin{definition}
    A tuple $\alpha=(\alpha_1,\dots,\alpha_n)$ is called admissible if for all $i=1,\dots,n$: 
    \[ -\infty<\alpha_i<1,\]
     \[ \sum_{i=1}^n\alpha_i=1\]
     and there is at most one index $j$ such that $\alpha_j<0$. If it is the case, then we say that the tuple $\alpha$ is bad with bad index $j$ and we call the other indices good. If there is no such index, we say that $\alpha$ is a good tuple. 
\end{definition}
\begin{definition}
    Let $F$, $F'$ be two finite measure subsets of $\mathbb{R}$ such that $F'\subset F$. We say that $F'$ is a major subset of $F$ if 
    \[ |F'|\geq \frac{1}{2}|F|.\]
\end{definition}
\begin{definition}
    Let $F\subset \mathbb{R}$ be of finite measure. We denote by $X(F)$, the space of all complex-valued functions $f$ supported on $F$ such that $\|f\|_\infty\leq 1$.
\end{definition}
\begin{definition}
    Let $\alpha=(\alpha_1,\dots,\alpha_n)$ be an admissible tuple. We say that a $n$-linear form $\Lambda$ is of restricted type $\alpha$ if for every $F_1,\dots, F_n$ subsets of $\mathbb{R}$ of finite measure, there exists a major subset $F_j'$ of $F_j$ for every bad index $j$ such that
    \[ |\Lambda(f_1,\dots,f_n)|\lesssim |F|^\alpha\]
    for all $f_1\in X(F_1'),\dots,f_n\in X(F_n')$. Above we use the convention $F_i'=F_i$ for the good indices and 
    \[|F|^\alpha: =\prod_{i=1}^n|F_i|^{\alpha_i}.\]
\end{definition}
In our case, we will prove in the next subsection the following result. 
\begin{theorem} \label{threstricted type2}
Let \(\mathbb P\) be a finite sparse collection of multitiles such that there exist
\(l\) and \(k\), with \(n-l>2k\), such that \(\mathbb P\) is of rank \(k\)
with respect to \(\{l+1,\ldots,n\}\). Then \(\Lambda_{\mathbb P}\) is of
restricted type \(\alpha\) for all \(\alpha\in\mathbb R^n\) satisfying either
\[\begin{aligned} &\left\{ \begin{array}{l} \displaystyle \sum_{j=1}^n \alpha_j =1, \\
\text{The bad index } s\in\{l+1,\ldots,n\},\\
\alpha_j\in(0,1),\qquad \text{for all } j=1,\ldots,l,\\
\alpha_j\in\left(\frac12,1\right),\qquad \text{for all }j=l+1,\ldots,n,\ j\neq s,\\
\alpha_s \in \left( \frac{3}{2}-n+k+l-\sum_{j=1}^l\alpha_j \;,\; 2-n+k+l-\sum_{j=1}^l\alpha_j \right),
\end{array}\right.
\\[1ex] &\hspace{2.5em}\text{\emph{or}}
\\[1ex] &\left\{ \begin{array}{l} \displaystyle \sum_{j=1}^n \alpha_j =1, \\
\text{The bad index }s\in\{1,\ldots,l\},\\
\alpha_j\in(0,1),\qquad \text{for all }j=1,\ldots,l,\ j\neq s, \\
\alpha_j\in\left(\frac12,1\right),\qquad \text{for all }j=l+1,\ldots,n,\\
\displaystyle\sum_{j=l+1}^{n}\alpha_j=n-l-k.
\end{array} \right.
\end{aligned} \]
We denote by $R_1$ the first set of $\alpha$s described above, $R_2$ the second one and $R=R_1\cup R_2$. 
\end{theorem}
\begin{remark}Notice that here, we obtain two different kinds of sets depending on where the bad index lies, whereas in the nondegenerate setting, we would only obtain something similar to $R_1$. The second range $R_2$ plays no extra role in the boundedness of the operator  $T_\mathbb{P}$ when $n\in G$ as the $\alpha$ tuples that we are interested in have to satisfy $\alpha_j\in [0,1)$ for all $j=1, \dots, n-1$.
\\One can feel confused as to why when we have degenerate directions, we get wider range for the exponent $\alpha_j$ but one should remember that these indices can only be bad in a very specific way ($R_2$ has strict constraints), and therefore when we take the convex hull of these tuples we get less estimates than in the nondegenerate case. 
\end{remark}
We are now interested in describing the convex hull of $R$ as we want to apply a multilinear interpolation result. We will focus on the convex hull of $R_1$ as $R_2$ only plays a part when $n$ has to be removed. 
In the spirit of \cite{MuscaluTaoThiele2002}, we first introduce a small lemma that will help us describe the convex hull. 
\begin{lemma} \label{lemmaconv}
    Let $\alpha\in \mathbb{R}^n$, with $\alpha_{l+1}\geq \cdots \geq \alpha_n$, then the convex hull of the points $(\alpha_1,\dots,\alpha_n)$ and its permutations of the $n-l$ last coordinates is the set of points $(\alpha_1,\dots,\alpha_l,x_{l+1},\dots, x_n)$ such that:
    \[ x_{l+1}+\cdots +x_n=\alpha_{l+1}+\cdots +\alpha_n, \] 
    \[ x_{i_1}+\cdots +x_{i_r} \leq \alpha_{l+1}+\cdots +\alpha_{r+l},\]
  for all $l+1\leq i_1<\cdots <i_r \leq n$ and all $1\leq r\leq n-l$.
\end{lemma}
The proof is similar to Lemma $3.6$ from \cite{MuscaluTaoThiele2002}. The only difference is that we keep the first $l$ indices fixed, rather than allowing them to be permuted.
\\This leads us to consider the set $Q$ of all admissible tuples $\alpha$ such that 
\[ \alpha_{i_1}+\cdots +\alpha_{i_r}< \frac{n-2k-l+r}{2}\]
for all $l+1\leq i_1<\cdots <i_r \leq n$, and all $1\leq r\leq n-l$. We also impose that if $\alpha$ is a bad tuple in $Q$, then its bad index is in $\{l+1,\dots,n\}$. Then we have: 
\begin{lemma} \label{lemmaconv2}
    The set $Q$ is contained in the convex hull of $R_1$ and contains all good tuples $\alpha$. Furthermore if $\alpha \in Q$ with bad index $j \in \left\{l+1,\dots, n\right\} $, then there exists $\tilde{\alpha}\in R_1$, with bad index $j$, such that for all $i\neq j$, $n\geq i\geq l+1$, $\tilde{\alpha}_i>\alpha_i$ and such that $\alpha$ is in the convex hull of $\tilde{\alpha}$ and the points of $R_1$ with bad index $i\neq j$, $i\in \{l+1, \dots, n\}$. 
\end{lemma}
The proof is similar to the one in \cite{MuscaluTaoThiele2002}, with some careful handling of the $l$ first indices which cannot be 
bad. However once these coordinates are fixed, by picking some $\alpha_1, \dots, \alpha_l \in (0,1)$ the proof is essentially the same. Indeed if $\alpha \in Q$, then $\alpha$ is in the convex hull of the extremal points 
  \[ \alpha^{(1)}= (\alpha_1,\dots, \alpha_l, \underbrace{1,\dots,1}_{n-2k-l},\underbrace{\frac{1}{2},\dots, \frac{1}{2}}_{2k-1}, \frac{3}{2}+k+l-n-\sum_{j=1}^l \alpha_j) \in \bar{R}\]
    and all permutations of $\alpha$ with $\alpha^{(i)}_j=\alpha_j$ for $j=1,\dots, l$.

\begin{remark}
In the case where $n\notin G$, the set we consider is the following: 
    \[S= \left\{ \alpha\in\mathbb{R}^n \;\middle|\; \begin{array}{l}
\alpha_j\in(0,1)\qquad \forall j \neq n,\\[0.3em]
\alpha_{i_1}+\cdots+\alpha_{i_r}<\dfrac{n-2k+r-l}{2},\quad \forall\, 1\le i_1<\cdots<i_r\le n, \quad i_1, \dots, i_r \in G \quad \forall\, r=1,\ldots,n-l,\\[0.3em]
\displaystyle\sum_{j=1}^{n}\alpha_j=1
\end{array}
\right\}.
\]
And the equivalent statement of Lemma \ref{lemmaconv2} is the following: \textit{ $S$ is in the convex hull of $R$. All good tuples ($\alpha_j\geq0$ for all $j$) are contained in $S$. If $\alpha \in S$ has bad index $n$, then $\alpha=\lambda \alpha^{(1)}+ (1-\lambda) \alpha^{(2)}$ with $\alpha^{(1)} \in \text{Conv}(R_1)$, $\alpha^{(2)}\in R_2$ with bad index $n$ and $\alpha^{(2)}$ is a majorant of $\alpha$, i.e. $\alpha^{(2)}_j>\alpha_j$ for all $j\in G$.} 
\\ The proof of this statement is very similar to the proof of Lemma \ref{lemmaconv2}, the only difference being the construction given a tuple $\alpha \in S$ of the associated $\alpha^{(1)}$ and $\alpha^{(2)}$. In this case $\alpha^{(2)}$ is defined, up to enlarging a bit $\alpha_j^{(2)}$ for $j \in G$, by:
\[ \alpha^{(2)}_j=\max(\frac{1}{2}, \alpha_j) \text{ for j}\in G, \quad \alpha^{(2)}_j=\alpha_j \text{ for j}\notin G, j\neq n, \quad \text{and } \alpha_n=1-(n-l-k)-\sum_{j\notin G} \alpha_j.\]
By picking a $\lambda$ satisfying 
\[\frac{\alpha_n-1}{\alpha_n^{(2)}-1}>1-\lambda > \frac{\alpha_n}{\alpha_n^{(2)}}\]
and defining $\alpha^{(1)}=\frac{1}{\lambda}\left(\alpha-(1-\lambda)\alpha^{(2)}\right)$, we get the desired statement. 
\end{remark}
Theorem \ref{threstricted type2} and Lemma \ref{lemmaconv2} imply the following. 
\begin{theorem}
    Let $\alpha \in Q$ and $m$ as in Theorem \ref{th1}. Then $\Lambda_m$ is of restricted type $\alpha$. 
\end{theorem}
This is enough to prove Theorem \ref{th1}, as in \cite{MuscaluTaoThiele2002}.
\\ We are only left with proving Theorem \ref{threstricted type2} which will be the object of the next subsection. 

\subsection{A localized estimate and restricted type estimate for Theorem \ref{threstricted type2}}
We recall some notations from the last section: 
\[ S_j=\mathrm{size}_{\mathbb{P}, j}(f_j), \;\: j=1,\dots,n,\]
\[ E_j=\mathrm{energy}_j(\mathbb{P}), \;\: j=l+1,\dots,n.\]

We first begin by proving an estimate which is the key to proving Theorem \ref{threstricted type2}.
\begin{theorem} \label{threcurrence}
     Let $\mathbb{P}$ be a finite, sparse collection of multitiles of rank $k$ with respect to $\{ l+1,\dots, n\}$, $I_0$ a dyadic interval such that for all $P \in \mathbb{P}$, $I_P \subset I_0$. Let $f_1, \dots, f_n $ measurable functions such that there exists $(a_j)\in \mathbb{R}^n$ such that for all $j=1, \dots, n$
     \[ \|f_j\|_2 \leq a_j.\]
     Assume moreover that there exist $\lambda_{l+1},\dots,\lambda_n$ strictly positive real numbers such that for all $P \in \mathbb{P}$ and all $j=l+1,\dots,n$, we have:
    \[ \|f_j \tilde{\chi}_{I_P}^{N/2}\|_2^2 \lesssim \lambda_j a_j^2 |I_P|.\]
    Then, we have the estimate 
    \[ \sum_{P\in \mathbb{P}} |I_P|^{1-\frac{n}{2}}\prod_{j=1}^n |\langle f_j, \phi_{P_j}^j \rangle | \lesssim \min(1,|I_0|)\left(\prod_{j=1}^l S_j \right) \times  \prod_{j=l+1}^n S_j^{\theta_j}a_j^{1-\theta_j}(1+\lambda_j)^{(1-\theta_j)/2}\]
    for any $\theta_{l+1},\dots, \theta_n \in (0,1)$ such that:
    \[ \theta_{l+1}+\cdots+\theta_n=n-2k-l.\]
\end{theorem}
\begin{remark} 
 We first prove the estimate for $k=1$ and  then prove the general case by induction. 
 \\ One can at first be surprised by the $\min(1,|I_0|)$ quantity that appears above. It is here for technical purposes since, when $k>1$, an $i$-overlapping tree $T$ does not have enough structure to apply Lemma \ref{lemmaprod} and therefore we will need to make the quantity $|I_T|$ appear with this $\min$. 
\\ With this in mind, to get the $\min(1,|I_0|)$, we will need to use the localization estimate \eqref{eq:energielocalisee}, for which we need the above hypothesis on the $\lambda_j$ to get the $|I_0|$ bound. This also explains the $a_j$ and $(1+\lambda_j)$ quantities that appear in our theorem. One can think of them as the reminiscence of the energy. 
\\ Also take note that the assumption on the existence of such $\lambda_j$ is very light as we will later on decompose our collection $\mathbb{P}$ into $\mathbb{P}_{\lambda_1,\dots,\lambda_n}$ with such $\lambda_1,\dots,\lambda_n$.
\\ Last but not least we relate the result presented above to Theorem 6.4 in \cite{MuscaluTaoThiele2002}. The first noticeable difference is the product of the sizes for the degenerate directions. This product is obtained by “ignoring” the indices $\{1, \dots, l\}$ and is the reason why we then obtain restricted type estimates for two separate sets of $\alpha$s, $R_1$ and $R_2$. As we have seen in the previous subsection, this weakens our interpolation result. 
\\The second difference is that our estimate no longer requires several technical assumptions that were introduced in \cite{MuscaluTaoThiele2002} solely for the induction argument. Moreover, we obtain a simpler estimate, which does not involve the auxilary quantity $A$, and only uses the exponents $\theta_{l+1}, \dots, \theta_n$ relevant to our setting. This estimate is closer in spirit to the size-energy estimate obtained in the case $k=1$, as the size quantity is preserved and the $a_j(1+\lambda_j)$ play the role of the energy. This is a consequence of the adjustments made in the last section. 
\end{remark}
\begin{proof}
The proof follows the same pattern as the one in \cite{MuscaluTaoThiele2002} for symbols singular along nondegenerate spaces acting only on $\{l+1, \dots, n\}$. 
\\ We denote:
\[ \tilde{\Lambda}_{\mathbb{P}}(f_1, \dots, f_n):= \sum_{P\in \mathbb{P}} |I_P|^{1-\frac{n}{2}}\prod_{j=1}^n |\langle f_j, \phi_{P_j}^j \rangle |.\]
\textbf{The $k=1$ case: } Assume that $k=1$. 
 Let $\theta_{l+1},\dots,\theta_n \in \left(0,1\right)$ such that $\displaystyle \sum_{j=l+1}^n \theta_j=n-2-l$. Let also $\varepsilon>0$ such that $\theta_j-\varepsilon>0$ for all $j=l+1,\dots,n$. One should not be concerned about the $\varepsilon$ parameter as it is only here to help us compute a certain sum. 
\\ Using Lemma \ref{lemmedec2}, for each index, $j=l+1,\dots,n$, we decompose our collection $\mathbb{P}$, into:
\[ \mathbb{P}=\bigcup_{m_j\in\mathbb{Z}} \mathbb{P}_{m_j}^{j}.\]
Then, 
\[\mathbb{P}=\bigcup_{m_{l+1},\dots,m_n\in\mathbb{Z}} \bigcap_{j=l+1}^n\mathbb{P}_{m_j}^{j}.\]
Without loss of generality we can assume that that $m_{l+1}\leq \cdots\leq m_n$ (by symmetry of the cases). For each tuple $m_{l+1}\leq \cdots \leq m_n$, we use the decomposition into trees from Lemma \ref{lemmedec2}:
\[\mathbb{P}_{m_{l+1}}^{l+1}=\bigcup_{T\in\mathbb{T}_{m_{l+1}}^{l+1}}T.\]
We denote 
\[ \mathbb{P}_{m_{l+1},\dots,m_n}:=\displaystyle \bigcap_{j=l+1}^n\mathbb{P}_{m_j}^{j} .\]
Since $k=1$, for each tree $T\in \mathbb{T}^{l+1}_{m_{l+1}}$, the tree $T\cap \mathbb{P}_{m_{l+2}}^{l+2} \cap \cdots \cap \mathbb{P}_{m_n}^n$ has at least two lacunary indices. Thus, using Lemma \ref{lemmaprod}, we get: 
\begin{align*}
     \tilde{\Lambda}_{\mathbb{P}_{m_{l+1},\dots,m_n}}(f_1\dots,f_n)&\leq \sum_{\substack{T':=T\cap \mathbb{P}_{m_{l+1}}^{l+1}\cap\cdots \cap \mathbb{P}_n^n \\T\in\mathbb{T}_{m_{l+1}}^{l+1} }} \tilde{\Lambda}_{T'}(f_1,\dots,f_n)\\
    &\leq \sum_{T\in\mathbb{T}_{m_{l+1}}^{l+1}}|I_T| \prod_{j=1}^n\mathrm{size}_{T',j}(f_j) \\
    &\leq \left(\prod_{j=1}^l\mathrm{size}_{\mathbb{P},j}(f_j)\right)\sum_{T\in\mathbb{T}_{m_{l+1}}^{l+1}}|I_T| \prod_{j=l+1}^n \min(2^{-m_j}E_j,S_j) \\
    &\lesssim 2^{2m_{l+1}} \left(\prod_{j=1}^l\mathrm{size}_{\mathbb{P},j}(f_j)\right)  \left( \prod_{j=l+2}^n (2^{-m_j}E_j)^{1-\theta_j}S_j^{\theta_j}\right) (2^{-m_{l+1}}E_{l+1})^{1-\theta_{l+1} +\varepsilon}S_{l+1}^{\theta_{l+1}-\varepsilon}
\end{align*}
where we used the estimate $\sum_{T \in \mathbb{T}_{m_{l+1}}^{l+1}} |I_T| \lesssim 2 ^{2m_{l+1}}$.
We now sum over  $m_{l+1}\leq \cdots \leq m_n$, using that for $m_{l+1}$, we have the stopping condition $2^{-m_{l+1}}E_{l+1} \leq S_{l+1}$ (up to a factor 2). Thus, we get: 
\begin{align*}
     \sum_{m_{l+1}\leq \cdots \leq m_n}\tilde{\Lambda}_{\mathbb{P}_{m_{l+1},\dots,m_n}}(f_1,\dots,f_n)&\lesssim  \left(\prod_{j=1}^l\mathrm{size}_{\mathbb{P},j}(f_j)\right)  \left( \prod_{j=l+2}^n E_j^{1-\theta_j}S_j^{\theta_j}\right)E_{l+1}^{1-\theta_{l+1} +\varepsilon}S_{l+1}^{\theta_{l+1}-\varepsilon}\\
     & \times\sum_{m_{l+1}} \underbrace{2^{2m_{l+1}-m_{l+1}\sum_{j=l+1}^n(1-\theta_j)-m_{l+1}\varepsilon}}_{2^{{2m_{l+1}}-m_{l+1}(n-l)+m_{l+1}(n-l-2)-m_{l+1}}=2^{-m_{l+1}\varepsilon}} \\
     &\lesssim \left(\prod_{j=1}^l\mathrm{size}_{\mathbb{P},j}(f_j)\right)  \left( \prod_{j=l+2}^n E_j^{1-\theta_j}S_j^{\theta_j}\right)E_{l+1}^{1-\theta_{l+1} +\varepsilon}S_{l+1}^{\theta_{l+1}-\varepsilon} \left(\frac{S_{l+1}}{E_{l+1}} \right)^{\varepsilon}\\
     &\lesssim  \left(\prod_{j=1}^l\mathrm{size}_{\mathbb{P},j}(f_j)\right)  \left( \prod_{j=l+1}^n E_j^{1-\theta_j}S_j^{\theta_j}\right).
     \end{align*}
We prove the same estimate in a similar way for the other subcollections of $\mathbb{P}$, and therefore get the estimate: 
\[ \tilde{\Lambda}_{\mathbb{P}}(f_1,\dots,f_n)\lesssim \left(\prod_{j=1}^l\mathrm{size}_{\mathbb{P},j}(f_j)\right)  \left( \prod_{j=l+1}^n E_j^{1-\theta_j}S_j^{\theta_j}\right). \]
Notice that we are almost done, we just need to use the local energy estimate Proposition \ref{energielocalisee}, and more precisely the estimate \eqref{eq:energielocalisee}, to conclude. Indeed, we get: 
\begin{align*}
    \tilde{\Lambda}_{\mathbb{P}}(f_1,\dots,f_n) &\lesssim \left(\prod_{j=1}^l\mathrm{size}_{\mathbb{P},j}(f_j)\right)  \left( \prod_{j=l+1}^n (\min(1,|I_0|)^{1/2}a_j(1+\lambda_j)^{1/2})^{1-\theta_j}S_j^{\theta_j}\right)\\
    &\lesssim \min(1,|I_0|)\left(\prod_{j=1}^l\mathrm{size}_{\mathbb{P},j}(f_j)\right)  \left( \prod_{j=l+1}^n (a_j(1+\lambda_j)^{1/2})^{1-\theta_j}S_j^{\theta_j}\right)
\end{align*}
which concludes the proof in the case $k=1$.

\begin{remark}
    For the case $k=1$, we could have stopped at the size-energy estimate:
    \[ \left|\Lambda_{\mathbb{P}}(f_1,\dots,f_n)\right|\lesssim \left(\prod_{j=1}^l\mathrm{size}_{\mathbb{P},j}(f_j)\right)  \left( \prod_{j=l+1}^n E_j^{1-\theta_j}S_j^{\theta_j}\right),\]
    which doesn't require the local energy estimates and therefore doesn't require the assumptions on the $I_0$ and $\lambda_j$ in the theorem. Indeed this bound is enough to conclude to restricted weak type bounds. We chose to write the theorem with the localized bound as we will need it for the $k>1$ case. 
\end{remark}
\textbf{The induction on $k$: }
Let now $k>1$ and assume the theorem is proven for $k-1$. The idea is to reduce the number of directions we need to fix in order to get a tree lacunary in at least two directions. To do so, we will use the decomposition lemma to decompose our collection into trees overlapping in at least one direction and then we will use the induction hypothesis on these trees. 
    \\ With this in mind let $\theta_{l+1},\dots,\theta_n\in (0,1)$ such that $\theta_{l+1}+\cdots +\theta_n=n-2k-l$.
    Let also $\varepsilon>0$ be such that for all $j=l+1,\dots,n$, $\theta_j-\varepsilon>0$. We use Lemma \ref{lemmedec2} to split our collection for every index $j=l+1,\dots,n$: 
\[\mathbb{P}=\underset{m_{l+1},\dots,m_n \in \mathbb{Z}}{\bigcup} \underset{j=l+1}{\overset{n}{\bigcap}}\mathbb{P}_{m_j}^j,\]
and we restrict ourselves again to study the subcollections such that $m_{l+1}\leq \cdots \leq m_n$. We now decompose for all $m_{l+1}$ $\mathbb{P}_{m_{l+1}}^{l+1}=\underset{T\in\mathbb{T}_{m_{l+1}}^{l+1}}{\bigcup}T$ and we consider:
\[\sum_{m_{l+1}\leq \cdots\leq m_n} \sum_{\substack{T'=T\cap\mathbb{P}_{m_{l+2}}^{l+2}\cap \cdots \cap \mathbb{P}_{m_n}^n \\ T\in \mathbb{T}_{m_{l+1}}^{l+1}}}  \tilde{\Lambda}_{T'}(f_1,\dots,f_n)\]
where $T\in \mathbb{T}_{m_{l+1}}^{l+1}$ is either an $l+1$-overlapping tree or it is $l+1$-lacunary and a $j$-overlapping tree, for an index $j>l+1$. We will separate the argument into these two cases, as they are somewhat similar but they need some slight adjustments.
\\ First we consider the $T'=T\cap\mathbb{P}_{m_{l+2}}^{l+2}\cap \cdots \cap \mathbb{P}_{m_n}^n$ trees that are $l+1$-overlapping trees and we estimate:
\[  \tilde{\Lambda}_{T'}(f_1,\dots,f_n)\lesssim 2^{-m_{l+1}(1- \theta_{l+1} +\varepsilon)}S_1^{\theta_{l+1}-\varepsilon}E_1^{1-\theta_{l+1} +\varepsilon} \tilde{\tilde{\Lambda}}_{T'}(f_1,\dots,f_l, f_{l+2},\dots,f_n)\]
where $ \displaystyle \tilde{\tilde{\Lambda}}_{T'}(f_1,\dots,f_l,f_{l+2}, \dots f_n)=\underset{P\in T'}{\sum} |I_P|^{1-n/2+1/2}\underset{\substack{j=1 \\ j\neq l+1}}{\overset{n}{\prod}}|\langle  f_j,\phi_{P_j}^j\rangle  |$.
\\ Notice that if we look at the subcollection $T'$ without the $l+1$ tile $P_{l+1}$, it is a collection of rank $k-1$ with space dimension $n-1$ since we fixed the $l+1$-th coordinate $P_{l+1} \leq P_{T,l+1}$ for all $P\in T'$. We would like to apply the induction hypothesis to the quasi multilinear form $\tilde{\tilde{\Lambda}}_{T'}$. To do so let's choose $0<\tilde{\theta}_j<1$ and $0<\theta_j'<1$ such that:
\begin{align*}
\begin{dcases}
    &\tilde{\theta_j}\theta_j'=\theta_j,\quad \text{for }l+2\leq j\leq n,\\
   &\sum_{j=l+2}^n \tilde{\theta_j}=(n-l-1)-2(k-1)=n-l-2k+1.
\end{dcases}
\end{align*}
As proving that such numbers exist is just a technical difficulty, we first finish the induction proof and then we will explain at how to pick such numbers. 
\\ Applying the induction result to the form $\tilde{\tilde{\Lambda}}_{T'}$, with the numbers $\tilde{\theta}_{l+2},\dots,\tilde{\theta}_n$ and $I_0=I_T$, we get: 
\[ \tilde{\tilde{\Lambda}}_{T'}(f_1,\dots,f_n)\lesssim \min(1,|I_T|) \prod_{j=1}^l S_j \times \prod_{j=l+2}^na_j^{1-\tilde{\theta}_j} (1+\lambda_j)^{\frac{1-\tilde{\theta}_j}{2}}\tilde{S}_j^{\tilde{\theta}_j}\]
where $\tilde{S}_j:=\mathrm{size}_{T',j}(f_j)$. Now since $T'\subset \mathbb{P}_{m_j}^j$, its $j$ size is bounded by the quantity $\min(S_j, 2^{-m_j}E_j)$ and interpolating between the two bounds with $\theta_j'$, we get: 
\[\tilde{\tilde{\Lambda}}_{T'}(f_1,\dots,f_n)\lesssim |I_T| \prod_{j=1}^lS_j \times \prod_{j=l+2}^n a_j^{1-\tilde{\theta}_j}(1+\lambda_j)^{\frac{1-\tilde{\theta}_j}{2}} \underbrace{S_j^{\theta_j'\tilde{\theta}_j} E_j^{(1-\theta_j')\tilde{\theta_j}}2^{-m_j(1-\theta_j')\tilde{\theta}_j}}_{=S_j^{\theta_j}E_j^{\tilde{\theta}_j-\theta_j}2^{-m_j(\tilde{\theta}_j-\theta_j)}}\]
Let's now look at the quantities we got from the last computation. The $|I_T|$ factor will be helpful with the sum over $T\in \mathbb{T}_{m_{l+1}}^{l+1}$ by using the estimate given in Lemma \ref{lemmedec2}. We want to keep the $a_j,\; 1+\lambda_j,\; S_j$ factors as they are in our induction hypothesis. The energy factors are going to be used to introduce the $\min(1, |I_0|)$ quantity as we did in the $k=1$ case. Finally, it remains to estimate geometric series in $m_{l+1}\leq \cdots\leq m_n$.
\\ With this in mind, we estimate: 
\begin{align*}
    \sum_{m_{l+1}\leq\cdots\leq m_n}
    \tilde{\Lambda}_{\mathbb{P}_{m_{l+1},\dots,m_n}}(f_1,\dots,f_n) &
    \lesssim S_{1+1}^{\theta_{l+1}-\varepsilon}E_{l+1}^{1-\theta_{l+1} +\varepsilon}\prod_{j=l+2}^n a_j^{1-\tilde{\theta}_j}(1+\lambda_j)^{\frac{1-\tilde{\theta}_j}{2}} S_j^{\theta_j}E_j^{\tilde{\theta}_j-\theta_j}  \times \prod_{j=1}^l S_j\\
    &\times \sum_{m_{l+1}\leq\cdots\leq m_n} 2^{-m_{l+1}(1- \theta_{l+1} +\varepsilon)} \prod_{j=l+2}^n 2^{-m_j(\tilde{\theta}_j-\theta_j)} \underbrace{\sum_{T\in \mathbb{T}_{m_{l+1}}^{l+1}} |I_T|}_{\lesssim 2^{2m_{l+1}}}\\
    & \lesssim \prod_{j=1}^l S_j \times S_{l+1}^{\theta_{l+1}-\varepsilon}E_{l+1}^{1-\theta_{l+1} +\varepsilon}\prod_{j=l+2}^n a_j^{1-\tilde{\theta}_j}(1+\lambda_j)^{\frac{1-\tilde{\theta}_j}{2}} S_j^{\theta_j}E_j^{\tilde{\theta}_j-\theta_j}  \\
    &\times \underbrace{\sum_{\substack{m_{l+1}\\ 2^{-m_{l+1}}E_{l+1}\leq S_{l+1}}}\underbrace{2^{2m_{l+1} -m_1{l+1}(1-\theta_{l+1}+\varepsilon)-m_{l+1}\underset{j=l+2}{\overset{n}{\sum}}(\tilde{\theta}_j-\theta_j)}}_{=2^{-m_{l+1}\varepsilon}}}_{\lesssim \left(\frac{S_{l+1}}{E_{l+1}}\right)^\varepsilon}\\
    &\lesssim \prod_{j=1}^lS_j \times  S_{l+1}^{\theta_{l+1}}E_{l+1}^{1-\theta_{l+1}}\prod_{j=l+2}^n a_j^{1-\tilde{\theta}_j}(1+\lambda_j)^{\frac{1-\tilde{\theta}_j}{2}} S_j^{\theta_j}E_j^{\tilde{\theta}_j-\theta_j}.
\end{align*}
Now we use Proposition \ref{energielocalisee} and more precisely \eqref{eq:energielocalisee} to estimate all energies by: 
\[ E_j \lesssim \min(1, |I_0|)^{1/2} a_j (1+ \lambda_j)^{1/2}.\]
This gives us the bound: 
\[
    E_{l+1}^{1-\theta_{l+1}} \prod_{j=+2}^n a_j^{1-\tilde{\theta}_j}(1+\lambda_j)^{\frac{1-\tilde{\theta}_j}{2}}E_j^{\tilde{\theta}_j-\theta_j} \lesssim \min(1, |I_0|)\prod_{j=l+1}^n a_j^{1-\theta_j}(1+\lambda_j)^{\frac{1-\theta_j}{2}}.\]
Finally we get the bound: 
\[\tilde{\Lambda}_\mathbb{P}(f_1,\dots,f_n)\lesssim \min(1, |I_0|) \prod_{j=1}^l S_j \times \prod_{j=l+1}^n a_j^{1-\theta_j}(1+\lambda_j)^{\frac{1-\theta_j}{2}}S_j^{\theta_j}\]
which was the expected induction quantity.
\\ We now need to deal with the case when $T'$ is $l+1$-lacunary. Then by construction, it is a $j$-overlapping tree for another index $j>l+1$. We will assume to simplify that it is a $n$-overlapping tree. Then we want to use the induction hypothesis on the collection $T'$ where we removed the $n$-th tile $P_n$ as this one is fixed by $P_{T,n}$. The reasoning is the same as before the only difference is how we deal with the $\varepsilon$ parameter which is why we separated the two cases.
\\ As we did before, we estimate
\[ \tilde{\Lambda}_{T'}(f_1, \dots, f_n) \leq \mathrm{size}_{T',n}(f_j) \tilde{\tilde{\Lambda}}_{T'}(f_1,\dots,f_{n-1}) \leq S_n^{\theta_n}E_n^{1-\theta_n} \tilde{\tilde{\Lambda}}_{T'}(f_1\dots,f_{n-1}).\]
Now we would like to apply the induction hypothesis on $\tilde{\tilde{\Lambda}}_{T'}$. To do so, we need to find as we did before some $\theta_j',\tilde{\theta}_j \in (0,1)$ such that: 
\begin{align*}
\begin{dcases}
    &\tilde{\theta_j}\theta_j'=\theta_j, \; j=l+2,\dots,n-1\\
    & \tilde{\theta}_{l+1}\theta_{l+1}'=\theta_{l+1}-\varepsilon \\
   &\sum_{j=l+2}^n \tilde{\theta_j}=(n-1)-2(k-1)-l=n-2k-l+1.
\end{dcases}
\end{align*}
This $\varepsilon>0$ that appears here is the same as the one from the first case, which we need in order to sum over $m_{l+1},\dots,m_n$. We again leave the proof of such $\tilde{\theta}$ and $\theta '$ existing for later. We are now ready to apply the induction hypothesis:
\[ \tilde{\tilde{\Lambda}}_{T'}(f_1,\dots,f_n) \lesssim \min(1,|I_T|) \prod_{j=1}^{l}\tilde{S}_j \times \prod_{j=l+1}^{n-1} a_j^{1-\tilde{\theta}_j}(1+\lambda_j)^{(1-\tilde{\theta}_j)/2}\tilde{S}_j^{\tilde{\theta}_j}\]
where $\tilde{S}_j:=\mathrm{size}_{T',j}(f_j)$. As before since $T'\subset \mathbb{P}_{m_j}^j$, we get the size estimate, for $j=l+1,\dots,n-1$:
\[\tilde{S}_j\leq S_j^{\theta_j'}2^{-m_j(1-\theta_j')}E_j^{1-\theta_j'}. \]
This gives us: 
\[ \tilde{\tilde{\Lambda}}_{T'}(f_1,\dots, f_{n-1}) \lesssim |I_T| \prod_{j=1}^l S_j \times \prod_{j=l+1}^{n-1}a_j^{1-\tilde{\theta_j}} (1+\lambda_j)^{(1-\tilde{\theta_j})/2}\underbrace{S_j^{\theta_j'\tilde{\theta}_j}2^{-m_j \tilde{\theta}_j(1-\theta_j')}E_j^{\tilde{\theta_j}(1-\theta_{j}')}.}_{\substack{=S_j^{\theta_j} E_j^{\tilde{\theta}_j-\theta_j }2^{-m_j (\tilde{\theta}_j-\theta_j)} \: \mathrm{if} \: j\geq l+2\\ = S_{l+1}^{\theta_{l+1}-\varepsilon} E_{l+1}^{\tilde{\theta}_{l+1}-\theta_{l+1}+\varepsilon}2^{-m_{l+1} (\tilde{\theta}_{l+1}-\theta_{l+1}+\varepsilon)} \: \mathrm{else}}}\]
As before, the $|I_T|$ quantity will be used to sum over the trees $T \in \mathbb{T}_{m_{l+1}}^{l+1}$, the $a_j$, $1+\lambda_j$ will be helpful to recover the induction hypothesis, the energy will allow us to get a new $\min(1,|I_0|)$ and we will use the powers of 2 to sum over $m_{l+1},\dots,m_n$. 
\\ We therefore estimate: 
\begin{align*}
\sum_{m_{l+1}\leq\cdots\leq m_n}\tilde{\Lambda}_{\mathbb{P}_{m_{l+1},\dots,m_n}}(f_1,\dots,f_n) 
    &\lesssim S_{l+1}^{\theta_{l+1}-\varepsilon}E_{l+1}^{\tilde{\theta}_{l+1}-\theta_{l+1} +\varepsilon} (a_1(1+\lambda_{l+1})^{1/2})^{1-\tilde{\theta}_{l+1}} \times E_n^{1-\theta_n}S_n^{\theta_n} \\
    &\times \prod_{j=l+2}^{n-1} a_j^{1-\tilde{\theta}_j}(1+\lambda_j)^{\frac{1-\tilde{\theta}_j}{2}} S_j^{\theta_j}E_j^{\tilde{\theta}_j-\theta_j}  \times \prod_{j=1}^l S_j\\
    &\times \sum_{m_{l+1}\leq\cdots\leq m_n}2^{-m_{l+1}(\tilde{\theta}_{l+1}- \theta_{l+1} +\varepsilon)}  2^{-m_n(1-\theta_n)} \prod_{j=l+2}^n 2^{-m_j(\tilde{\theta}_j-\theta_j)} \underbrace{\sum_{T\in \mathbb{T}_{m_{l+1}}^{l+1}} |I_T|}_{\lesssim 2^{2m_{l+1}}}\\
    & \lesssim\times  S_{l+1}^{\theta_{l+1}-\varepsilon}E_{l+1}^{\tilde{\theta}_{l+1}-\theta_{l+1} +\varepsilon} (a_1(1+\lambda_{l+1})^{1/2})^{1-\tilde{\theta}_{l+1}} \times E_n^{1-\theta_n}S_n^{\theta_n} \\
    & \times \prod_{j=1}^l S_j \prod_{j=l+2}^{n-1} a_j^{1-\tilde{\theta}_j}(1+\lambda_j)^{\frac{1-\tilde{\theta}_j}{2}} S_j^{\theta_j}E_j^{\tilde{\theta}_j-\theta_j}  \\
    &\times \underbrace{\sum_{\substack{m_{l+1}\\ 2^{-m_{l+1}}E_{l+1}\leq S_{l+1}}}\underbrace{2^{2m_{l+1} -m_{l+1}(\tilde{\theta}_{l+1}-\theta_{l+1}+\varepsilon)-m_{l+1}\underset{j=l+2}{\overset{n-1}{\sum}}(\tilde{\theta}_j-\theta_j)-m_{l+1}(1-\theta_n)}}_{=2^{-m_{l+1}\varepsilon}}}_{\lesssim \left(\frac{S_{l+1}}{E_{l+1}}\right)^\varepsilon}\\
    &\lesssim \prod_{j=1}^l S_j \times  S_n^{\theta_n}E_n^{1-\theta_n}\prod_{j=l+2}^{n-1} a_j^{1-\tilde{\theta}_j}(1+\lambda_j)^{\frac{1-\tilde{\theta}_j}{2}} S_j^{\theta_j}E_j^{\tilde{\theta}_j-\theta_j}.
\end{align*}
One can see that the estimate we obtain is similar to  the first case and as it now doesn't depend on the $\theta_j'$, the end of the reasoning is the same as the last case with the indices $l+1$ and $n$ exchanged. Using the local energy estimate \eqref{eq:energielocalisee} and computing all these we get the desired bound: 
\[\tilde{\Lambda}_\mathbb{P}(f_1,\dots,f_n)\lesssim \min(1, |I_0|) \prod_{j=1}^l S_j \times \prod_{j=l+1}^n a_j^{1-\theta_j}(1+\lambda_j)^{\frac{1-\theta_j}{2}}S_j^{\theta_j}.\]
\\ \textbf{We are now left with proving the existence of the $\tilde{\theta_j}$ and $\theta_j'$}. 
We want to find some $\tilde{\theta}_{l+2},\dots,\tilde{\theta}_n, \theta_{l+2}', \dots,\theta_n' \in \left(0,1\right)$ such that: 
\begin{align*}
\begin{dcases}
    &\tilde{\theta_j}\theta_j'=\theta_j,\\
   &\sum_{j=l+2}^n \tilde{\theta_j}=n-2k-l+1.
\end{dcases}
\end{align*}
We reduce this problem into finding for all $j=l+2,\dots,n$, $\varepsilon_j>0$ such that $\theta_j+\varepsilon_j<1$ and $\displaystyle \sum_{j=l+2}^n\varepsilon_j=1+\theta_{l+1}$, and then define:
\[\tilde{\theta}_j:=\theta_j+\varepsilon_j \;\;\mathrm{and}\;\;\theta_j':=\frac{\theta_j}{\theta_j+\varepsilon_j}.\]
Notice that for the second case, it will be enough to define $\tilde{\theta}_{l+1}$ in the same way and just change $\theta_{l+1}':= \frac{\theta_{l+1}-\varepsilon}{\theta_{l+1}+\varepsilon_{l+1}}$ which is why we focus on the first case. 
First choose, for any $j=l+2,\dots,n$, a $\delta_j>0$ such that 
\[ 1>\theta_j+\delta_j>\frac{n-l-2}{n-l-1}.\]
Then: 
\begin{align*}
    \sum_{j=l+2}^n\delta_j&=\sum_{j=l+2}^n(\delta_j+\theta_j) -\sum_{j=l+2}^n\theta_j\\
    &>\frac{n-l-2}{n-l-1}\times(n-l-1)-\left(\sum_{j=l+1}^n\theta_j\right)+\theta_{l+1}\\
    &>n-l-2-(n-2k-l)+\theta_{l+1} \\
    &>\underbrace{2(k-1)}_{\geq 1 \;\mathrm{since}\;k>1}+\theta_{l+1} \\
    &>\theta_{l+1} +1
\end{align*}
For $j=l+2, \dots,n$, define $\varepsilon_j=\delta_j\times\frac{\theta_{l+1}+1}{\sum\delta_l}$, then $0<\varepsilon_j<\delta_j$, therefore $\theta_j+\varepsilon_j<1$ and $\sum \varepsilon_j=1+\theta_{l+1}$ which is what we wanted. This concludes the proof of Theorem \ref{threcurrence}.
\end{proof}
We now have all the tools needed to prove Theorem \ref{threstricted type2}. 
\begin{proof}\textit{(Theorem \ref{threstricted type2})} Let $F_1, \dots, F_n \subset \mathbb{R}$, with finite measure. Let $\alpha \in R_1$, without loss of generality, we can assume that the bad index is $n$. By a scaling argument we can assume that $|F_n| \sim 1$. 
\\ Define the exceptionnal set: 
    \[ \Omega:=\bigcup_{j=1}^n \left\{ x \: | \: M(\chi_{F_j})(x)\geq C|F_j|\right\},  \]
and 
\[ F_n'=F_n\backslash \Omega. \]
Notice that if $C$ is large enough we have $|\Omega|< \frac{1}{2}$ and therefore, $|F_n'|> \frac{1}{2}|F_n|$. So $F_n'$ is a major subset of $F_n$. Let $f_1, \dots, f_{n-1} \in X(F_1),\dots, X(F_{n-1})$ and $f_n \in X(F_n')$. We want to prove that 
\[ |\Lambda_{\mathbb{P}}(f_1,\dots, f_n)| \lesssim |F'|^\alpha.\]
We denote for $j=1, \dots, n-1$, 
$a_j:=|F_j|^{1/2}$ and $a_n:=|F_n'|^{1/2}$.
\\We decompose \[ \mathbb{P}=\bigcup_{d \in \mathbb{N}} \mathbb{P}_{d}\]
where 
\[\mathbb{P}_d=\left\{P \in \mathbb{P} \; |\; (2^d-1)|I_P|\leq \text{dist}(I_P,\Omega^\complement )<(2^{d+1}-1)|I_P|\right\}.\]
Fix one $d \in \mathbb{N}$, we would like to apply Theorem \ref{threcurrence} to the subcollection $\mathbb{P}_d$. First let's notice a few properties of such a collection. 
For $j=1,\dots, n-1$, we have the size estimate:
\[ \mathrm{size}_{\mathbb{P}_d,j}(f_j)\lesssim \sup_{P\in \mathbb{P}_d}\frac{1}{|I_P|}\|f_j\tilde{\chi}_{I_P}^N\|_1 \lesssim\sup_{\substack{I_P \\ P \in \mathbb{P}_d}}\frac{1}{|I_P|}\int_{F_j}\tilde{\chi}_{I_P}^N(x)dx \lesssim 2^d a_j^2 \]
where we used Lemma \ref{sizeestimate} for the first inequality. We then used that for $P \in \mathbb{P}_d$, $2^{d+1}I_P \cap \Omega^\complement\neq \emptyset$ and that on $\Omega^\complement$, $M(\frac{\chi_{F_j}}{|F_j|})\leq C$.
\\ For $j=n$, on another hand we have: 
\[ \mathrm{size}_{\mathbb{P}_d,n}(f_n)\lesssim\sup_{P\in \mathbb{P}_d} \frac{1}{|I_P|}\int_{F_n'} \tilde{\chi}_{I_P}^N(x)dx \lesssim2^{-Nd} a_n^2 \]
where we used that $F_n' \subset \Omega^\complement$ and that $I_P$ is at distance $\sim 2^d |I_P|$ from $\Omega^\complement$.
\\ Last but not least, we estimate for $j=1,\dots,l$:
\[\mathrm{size}_{\mathbb{P}_{d},j}(f_j)\lesssim \min(1,a_j^2 2^{d}) \lesssim a_j^{2\beta_j}2^{d\beta_j}\]
for any $\beta_j \in [0,1]$.
 \\ Now to apply Theorem \ref{threcurrence}, we have to check that $\mathbb{P}_d$ satisfies the hypothesis. First, $\mathbb{P}_d$ is a sparse finite collection of tiles of rank $k$ with respect to $\{l+1, \dots,n\}$ and therefore, there exists an $I_0$ which contains all $I_P$, for $P\in \mathbb{P}_d$. Now for any $I_P$, such that $P \in \mathbb{P}_d$, we have for $l+1\leq j \leq n-1 $: 
\[ \|f_j\tilde{\chi}^{N/2}_{I_P}\|_2^2 = \int_\mathbb{R}|f_j(x)|^2\tilde{\chi}^{N}_{I_P}(x)dx\lesssim\int_{F_j}\tilde{\chi}_{I_P}^N(x)dx \lesssim2^da_j^2|I_P|\]
by the same reasoning as before. And similarly, we have for $j=n$, 
\[ \|f_n\tilde{\chi}^{N/2}_{I_P}\|_2^2 \lesssim \int_{F_n'}\tilde{\chi}_{I_P}^N(x)dx \lesssim2^{-Nd}a_n^2|I_P|.\]
We are now ready to apply Theorem \ref{threcurrence}, with $\lambda_{l+1},\dots,\lambda_{n-1}=2^d$, $\lambda_n=2^{-Nd}$ to get: 
\begin{align*}
    \left| \Lambda_{\mathbb{P}_{d}}(f_1,\dots,f_n)\right| &\lesssim \left( \prod_{j=1}^l \mathrm{size}_{\mathbb{P}_{d},j}(f_j) \right) \times \left( \prod_{j=l+1}^n a_j^{\theta_j}(1+\lambda_j)^{1/2(1-\theta_j)} \mathrm{size}_{\mathbb{P}_{d},j}(f_j)^{\theta_j}\right) \\
    &\lesssim  \left(\prod_{j=1}^l a_j^{2\beta_j} \lambda_j^{\beta_j} \right)\times \left( \prod_{j=l+1}^n a_j^{1+\theta_j}\lambda_j^{\theta_j}(1+\lambda_j)^{1/2(1-\theta_j)} \right) \\
    &\lesssim 2^{dC_{\beta_1,\dots, \beta_l,\theta_{l+1},\dots, \theta_{n-1} }} 2^{-Nd C_{\theta_n}}\left(\prod_{j=1}^la_j^{2\beta_j} \right)\times \left( \prod_{j=l+1}^n a_j^{1+\theta_j}\right).
\end{align*}
Provided that $N$ is large enough, this allows us to sum over $d$. 
\\ In the end, we get: 
\[  \left| \Lambda_\mathbb{P}(f_1,\dots,f_n)\right| \lesssim \left(\prod_{j=1}^l a_j^{2\beta_j}\right) \times \left( \prod_{j=1=l+1}^na_j^{1+\theta_j} \right),\]
and this for all $\beta_j \in \left(0,1\right)$, $\theta_j \in \left(0,1\right)$, such that $\displaystyle \sum_{j=l+1}^n \theta_j=n-l-2k$. 
\\ We now only need to prove that we can choose, for any $\alpha \in R_1$, $\beta_j=\alpha_j$ for $j=1, \dots, l$ and $\frac{1}{2}(1+ \theta_j)=\alpha_j$ for $j=l+1, \dots, n-1$. Indeed if this is the case then we would have: 
\[ \left|\Lambda_\mathbb{P}(f_1, \dots, f_n)\right| \lesssim \prod_{j=1}^l a_j^{2\beta_j} \prod_{j=l+1}^n a_j^{1+\theta_j} \lesssim \prod_{j=1}^l |F_j|^{\beta_j} \prod_{j=l+1}^n |F_j|^{(1+\theta_j)/2} \lesssim |F_n|^{(1-\theta_n)/2} \prod_{j=1}^{n-1}|F_j|^{\alpha_j} \lesssim |F|^\alpha\] 
where we used that $|F_n| \sim 1$ and that $a_j= |F_j|^{1/2}$.
   \\ First, for $j=l+1,\dots, n-1$, if we define $\theta_j=2 \alpha_j-1$, then 
     \[ \frac{1}{2}<\alpha_j <1 \iff 0<\theta_j<1.\]
     \\ Then, we define $\displaystyle \theta_n= n-l-2k- \sum_{j=l+1}^{n-1} \theta_j$. The only thing we need to check is that \\ $\theta_n \in \left(0,1\right)$. Let's rewrite $\theta_n$ depending on our $\alpha$:
     \begin{align*}
         \theta_n&=n-l-2k-\sum_{j=l+1}^{n-1}\theta_j \\
         &=n-l-2k- \sum_{j=l+1}^{n-1}(2\alpha_j-1) \\
         &=n-l-2k-\underbrace{2\sum_{j=1}^n \alpha_j}_{=2} +  2\sum_{j=1}^l\alpha_j + 2\alpha_n +n-l-1 \\ 
         &=2n-2l-2k+2 \alpha_n +2\sum_{j=1}^l\alpha_j-3.
     \end{align*}
    And therefore 
    \[ \alpha_n \in \left( \frac{3}{2}-n+l+k-\sum_{j=1}^l\alpha_j  \; ,\; 2-n+l+k -\sum_{j=1}^l \alpha_j\right) \iff 0<\theta_n<1.\]
    This concludes the proof when $\alpha \in R_1$. 
    \\ When $\alpha \in R_2$, if we for example assume that $1$ is the bad index, then one can repeat the same argument exchanging the indices $1$ and $n$. We get the estimate:
 \[ \left| \Lambda_{\mathbb{P}}(f_1,\dots,f_n)\right|\lesssim \left(\prod_{j=1}^l a_j^{2\beta_j} \right)\times \left( \prod_{j=l+1}^n a_j^{1+\theta_j}\right).\]
 for all $\beta_j \in (0,1)$, $\theta_j \in (0,1)$, such that $\displaystyle \sum_{j=l+1}^n \theta_j=n-l-2k$. 
 \\ By definition of $R_2$, we can choose $\beta_j=\alpha_j$ for $j=2, \dots,l $ and $\frac{1}{2}(1+\theta_j)=\alpha_j$ for $j=l+1, \dots, n$.
\end{proof}
As a conclusion, we proved by induction a localized estimate which replaces the size-energy estimate of the $k=1$ case. This is the key point to get restricted type estimates for the multilinear form and conclude to Theorem \ref{th1} by convex interpolation. 
\section{$L^p$ bounds for singularities with fewer strongly degenerate indices}
In this section, we present the proof of our second result, Theorem \ref{th2}, and complement it with some criteria determining when $\Xi_{\Gamma'}\neq \emptyset$ and two detailed examples.
\\ We now recall the context of Theorem \ref{th2}. Let $m \in \mathcal{C}^\infty(\Gamma \backslash \Gamma')$ with $\Gamma '$ a vector subspace of dimension $k$. We assume that there exists $s\in \mathbb{N}$, such that $1, \dots, s$ are the only \textit{strongly degenerate} indices of $\Gamma'$ and that $n-s>2k$. Then the multitile collection $\mathbb{P}$ of the discretized model operator satisfies Lemma \ref{lemmarank} and all the nondegenerate directions $(i_1, \dots, i_k)$ satisfy $i_1>s$. The indices $1, \dots, s$ will be treated in a similar way as in the last section as they don't parametrize in any way the singularity. Recall that since $1, \dots,s$ are \textit{strongly degenerate indices}, for all $u \in \Gamma'$, $u_1=\cdots=u_s=0$. 
\\The approach for the proof of Theorem \ref{th2} follows the proof presented in \cite{BeneaMuscalu2021} for vector valued multilinear multipliers singular along a nondegenerate vector space, as it allows us to consider the set of all nondegenerate directions. This method also gives us a sparse domination for the discretized model operator. 

\subsection{Some definitions and a localized estimate}
We begin by setting the frame for our approach and prove a similar estimate as Theorem 19 from \cite{BeneaMuscalu2021} which we will use in the next subsections to obtain the sparse domination and $L^p$ boundedness. As we seek to obtain a sparse domination, we will need to have some estimates on localized collections.  
\begin{definition}
    Let $I_0$ be a dyadic interval. We denote by $\mathbb{P}(I_0)$ the subcollection of multitiles $P \in \mathbb{P}$ such that $I_P\subset I_0$.
\end{definition}
We will also need to sometimes include the interval $I_0$ in the collection we consider and therefore, we need to slightly change our definition of size.
\begin{definition} Let $g$ be a function, $I_0$ a dyadic interval, then we define
    \[\widetilde{\mathrm{size}}_{I_0}(g)=\max \left( \sup_{P \in \mathbb{P}(I_0)}\frac{1}{|I_P|}\int |g| \tilde{\chi}_{I_P}^{M}, \frac{1}{|I_0|}\int |g| \tilde{\chi}_{I_0}^M   \right).\]
    For any $\mathfrak{s}>0$, we also define 
    \[ \widetilde{\mathrm{size}}_{I_0}^\mathfrak{s}(g):=\left(\widetilde{\mathrm{size}}_{I_0}(g^\mathfrak{s})\right)^{1/\mathfrak{s}}. \]
\end{definition}
\begin{remark}
    Notice that we also changed the usual $L^2$ size to a maximal average over the dyadic space intervals in $\mathbb{P}(I_0)$. This will not be an issue as we have from Lemma \ref{sizeestimate}:
    \[ \mathrm{size}_{\mathbb{P}(I_0),j}(f_j) \lesssim \widetilde{\mathrm{size}}_{I_0}(f_j).\]
    In particular, this new definition of size does not depend on the frequency intervals from the tiles but only on the collection of space intervals. 
\end{remark}
We are now ready to state the key estimate needed for the method in \cite{BeneaMuscalu2017}.
We recall that:
\begin{align*}
\Xi_{\Gamma'}:= \{(\alpha_{s+1},\dots, \alpha_n) \in (0,\frac{1}{2})^{n-s} \:| \: &\exists(\theta_{i_1,\dots, i_k})\in [0, 1]^{\sharp \{(i_1, \dots, i_k) \text{ good direction}\}}\: s.t. \\ & \alpha_j= \sum_{\substack{(i_1, \dots, i_k) \text{ good} \\ \exists r \text{ s.t. } i_r=j}} \theta_{i_1,\dots, i_k},  
 \sum_{(i_1, \dots, i_k) \text{ good}} \theta_{i_1, \dots i_k}=1 \}.\end{align*}
\begin{proposition} \label{thm19} 
Let $I_0$ be a dyadic interval. Then for any $(\alpha_{s+1},\dots, \alpha_n) \in \Xi_{\Gamma'}$
and for any $\beta_1,\dots, \beta_s\in [0,1]$, 
we have 
\begin{equation} \label{eq:localizedestimate2}
    |\Lambda_{\mathbb{P}(I_0)}(f_1,\dots, f_n)| \lesssim |I_0| \prod_{j=1}^s \widetilde{\mathrm{size}}_{I_0}(\mathbb{1}_{F_j})^{\beta_j} \times \prod_{j=s+1}^n \widetilde{\mathrm{size}}_{I_0}(\mathbb{1}_{F_j})^{1-\alpha_j} \end{equation}
    for any $F_1,\dots, F_n \subset \mathbb{R}$, sets of finite measure and any $f_1,\dots,f_n $ functions such that $f_j \in X(F_j)$.
\end{proposition}
\begin{proof}
   For any $j=s+1, \dots, n$, we decompose our collection as before, using Lemma \ref{lemmedec3}, into:
    \[ \mathbb{P}(I_0)=\bigcup_{m_j}\bigcup_{T \in \mathbb{T}_{m_j}^j}T\]
    such that, whenever $\mathbb{P}_{m_j}^j\neq \emptyset$, 
    \[ \mathrm{size}_{\mathbb{P}_{m_j}^j,j}(f_j)\sim 2^{-m_j}\]
    and therefore we can write 
    \[\mathbb{P}(I_0)=\bigcup_{m_{s+1},\dots, m_n}\bigcup_{T_{s+1} \in \mathbb{T}_{m_{s+1}} ^{s+1}} \cdots \bigcup_{T_n \in \mathbb{T}_{m_n}^n} T_{s+1} \cap \dots \cap T_n. \]
    Notice that any $T:=T_{s+1} \cap \cdots \cap T_n$ is a $j$-tree for any $j=s+1,\dots, n$ and a $(i_1, \dots, i_k)$-tree for any $(i_1,\dots, i_k)$  nondegenerate directions. In particular we  can apply Lemmas \ref{lemmafinite} and \ref{lemmaprod} to obtain:
    \begin{align*}
        | \Lambda_{\mathbb{P}(I_0)}(f_1,\dots,f_n)|&\leq \sum_{m_{s+1},\dots, m_n} \sum_{T_{s+1}\in \mathbb{T}_{m_{s+1}}^{s+1}, \dots, T_n \in \mathbb{T}_{m_n}^n } |\Lambda_{T_{s+1}\cap \cdots \cap T_n}(f_1,\dots, f_n)| \\
        &\lesssim \sum_{m_{s+1},\dots, m_n} \sum_{T_{s+1}\in \mathbb{T}_{m_{s+1}}^{s+1}, \dots, T_n \in \mathbb{T}_{m_n}^n }  |I_T|  \prod_{j=1}^n \mathrm{size}_{T_{s+1}\cap \cdots \cap T_n,j}(f_j) \\
        &\leq  \prod_{j=1}^s \mathrm{size}_{\mathbb{P}(I_0),j}(f_j) \sum_{m_{s+1},\dots, m_n} \sum_{T_{s+1}\in \mathbb{T}_{m_{s+1}}^{s+1}, \dots, T_n \in \mathbb{T}_{m_n}^n } |I_T| 2^{-(m_{s+1}+\cdots + m_n)}.
    \end{align*}
    Now since $T_{s+1} \cap \cdots \cap T_n$ is a $(i_1, \dots, i_k)$-tree for any $(i_1,\dots, i_k)$ good direction, we can estimate the sum above in $\sharp \left\{ (i_1,\dots, i_k) \; | \; (i_1, \dots, i_k) \mathrm{\; is \; a \; good \; direction}\right\}$ different ways. We denote by $\mathbb{T}_{m_{s+1},\dots, m_n}$ the collection of trees $T:=T_{s+1} \cap \cdots \cap T_n$ with $T_{s+1} \in \mathbb{T}_{m_{s+1}}^{s+1},\dots, T_n \in \mathbb{T}_{m_n}^n$, and we estimate: 
    \[ \displaystyle \sum_{T\in \mathbb{T}_{m_{s+1},\dots, m_n}} |I_T|  \leq \prod_{(i_1,\dots, i_k)\; \mathrm{good}} \left( \sum_{\substack{T \in \mathbb{T}_{m_{s+1} \dots, m_n} \\ T \mathrm{\; is \; a \; } (i_1,\dots, i_k) \; \mathrm{tree}}} |I_T|\right)^{\theta_{i_1,\dots, i_k}}.\]
Fix any $(\theta_{i_1,\dots,i_k}) \in [0,1]^{\sharp \{ \text{good directions}\}}$ such that $\displaystyle \sum_{(i_1,\dots,i_k) \; \mathrm{good}}\theta_{i_1,\dots, i_k}=1$
and a good direction $(i_1, \dots, i_k)$ and now 
\[ \sum_{\substack{T \in \mathbb{T}_{m_{s+1} \dots, m_n} \\ T \mathrm{\; is \; a \; } (i_1,\dots, i_k) \; \mathrm{tree}}} |I_T| \leq \sum_{T_{i_1} \in \mathbb{T}_{m_{i_1}}^{i_1}} \cdots \sum_{T_{i_k}\in \mathbb{T}_{m_{i_k}}^{i_k}}|I_{T_{i_1}}\cap \cdots \cap I_{T_{i_k}}|.\]
To estimate the right hand side above, we consider the subcollection of such trees such that $I_{T_{i_1}} \subset \cdots \subset I_{T_{i_k}}\subset I_0$ as we can deal with the other symmetric subcollections in a similar way. This will allow us to use a localized energy result. Indeed using Lemma \ref{lemmedec3} and Lemma \ref{energielocalisee}, we have:
\[\sum_{\substack{T_{i_k} \in \mathbb{T}_{m_{i_k}}^{i_k}\\ I_{T_{i_k}}\subset I_{T_{i_{k-1}}}}} |I_{T_{i_k}}| \lesssim 2^{2m_{i_k}}\|f_{i_k}\tilde{\chi}_{I_{T_{i_{k-1}}}}\|_2^2 \lesssim 2^{2m_{i_k}}\widetilde{\mathrm{size}}_{I_0}(\mathbb{1}_{E_{i_k}})|I_{T_{i_{k-1}}}|.\]
This allows us to sum over $T_{i_1} \in \mathbb{T}_{m_{i_1}}^{i_1},\dots, T_{{i_k}} \in \mathbb{T}_{m_{i_k}}^{i_k}$, by iterating the argument in $i_k,\dots, i_1$ and we get:
\begin{align*}
    |\Lambda_{\mathbb{P}(I_0)}(f_1,\dots, f_n)| &\lesssim \prod_{j=1}^s \mathrm{size}_{\mathbb{P}(I_0),j}(f_j) \sum_{m_{s+1},\dots, m_n} 2^{-(m_{s+1}+\cdots +m_n)}\prod_{(i_1,\dots, i_k) \; \mathrm{good}} \left( \prod_{j=1}^k 2^{2m_{i_j}}\widetilde{\mathrm{size}}_{I_0} (\mathbb{1}_{F_{i_j}})|I_0|\right)^{\theta_{i_1,\dots, i_k}} \\
    &\lesssim \prod_{j=1}^s \mathrm{size}_{\mathbb{P}(I_0),j}(f_j) \sum_{m_{s+1},\dots, m_n}  |I_0|\prod_{j=s+1}^n 2^{-m_j(1-2 \alpha_j)} \widetilde{\mathrm{size}}_{I_0}(\mathbb{1}_{F_j})^{\alpha_j}
\end{align*}
where for $j=s+1,\dots, n$, we set: 
\[ \alpha_j:=\sum_{\substack{(i_1,\dots, i_k) \; \mathrm{good}\\ \exists r \; i_r=j}}\theta_{i_1,\dots, i_k}.\]
Notice that to sum over $m_{s+1},\dots, m_n$, we need the assumption $\alpha_j \in (0,\frac{1}{2})$. If that's the case then recalling that we are summing over $m_{j}$ such that \[2^{-m_j}\lesssim \mathrm{size}_{\mathbb{P}(I_0),j}(f_j)\lesssim \widetilde{\mathrm{size}}_{I_0}(\mathbb{1}_{F_j})\]
we get the estimate:
\[ |\Lambda_{\mathbb{P}(I_0)}(f_1, \dots, f_n)| \lesssim |I_0| \prod_{j=1}^s\mathrm{size}_{\mathbb{P}(I_0),j}(f_j) \times \prod_{j=s+1}^n \widetilde{\mathrm{size}}_{I_0}(\mathbb{1}_{F_j})^{1- \alpha_j} .\]
Now to conclude it suffices to use the fact that for $j=1, \dots, s$
\[ \mathrm{size}_{\mathbb{P}(I_0),j}(f_j) \lesssim \widetilde{\mathrm{size}}_{I_0} (f_j) \lesssim \widetilde{\mathrm{size}}_{I_0} (\mathbb{1}_{F_j})\]
and that \[ \mathrm{size}_{\mathbb{P}(I_0),j}(f_j) \lesssim 1.\]
So by a geometric average, we get the following bound for any $\beta_j\in [0,1]$:
\[\text{size}_{\mathbb{P}(I_0),j}(f_j) \lesssim \widetilde{\text{size}}_{I_0}(\mathbb{1}_{F_j})^{\beta_j}\]
\end{proof}
The estimate \eqref{eq:localizedestimate2} is the key to prove Theorem \ref{th2}. In particular the range of boundedness obtained is deduced from the set $\Xi_{\Gamma'}$ which appears in the proof. 
\subsection{From the localized estimate to $L^p$ bounds}
We will now briefly explain how to get $L^p$ boundedness from Proposition \ref{thm19} using the method described in \cite{BeneaMuscalu2021}. 
We first consider the Banach case where $p_n'>1$, we will use the estimate above to get a sparse domination and from the sparse domination conclude to boundedness using a duality argument. 
\\ The first step is to go from a bound with restricted type functions to a bound for all locally integrable functions via the Mock interpolation. 
\begin{proposition} \label{Mock Banach}
    Let $(\alpha_{s+1},\dots, \alpha_n) \in \Xi_{\Gamma'}$, then for any $f_1,\dots, f_n$ locally integrable functions and $s_1,\dots,s_n$ such that $s_j(1-\alpha_j)>1$ for $j=s+1, \dots, n$ and $s_1,\dots, s_s > 1$, we have 
    \[ |\Lambda_{\mathbb{P}(I_0)}(f_1,\dots, f_n) |\lesssim|I_0| \prod_{j=1}^n\widetilde{\mathrm{size}}_{I_0}^{s_j}(f_j).\]
\end{proposition}
The main idea is to split each function $f_j$ into 
$f_j = \sum_{k}2^k f_{j,k} \mathbb{1}_{F_k}$, such that $f_{j,k} \in X(F_k)$ and apply Proposition \ref{thm19} to each $f_{j,k}\mathbb{1}_{F_k}$, provided that one can sum over $k_j$ which is the case under the assumption $s_j(1- \alpha_j)>1$. We refer the reader to the proof of Proposition 14 in \cite{BeneaMuscalu2021} for more details.
\\ With this, we now obtain the sparse domination of the multilinear form. \begin{definition} \label{defsparse}
   Let $\eta\in (0,1)$. A collection of dyadic intervals $\mathcal{S}$ is said to be $\eta$-sparse if there exists a pairwise disjoint collection $E_Q \subset Q$ such that $|E_Q|\geq \eta |Q|$ for all $Q \in \mathcal{S}$.
\end{definition}
We say that $\mathcal{S}$ is sparse if it is $\eta$-sparse for some $\eta \in(0,1)$.
\begin{proposition} \label{sparse domination Banach}
Let $f_1, \dots, f_n$ be locally integrable functions, $s_1, \dots, s_n$ exponents as in the last proposition. Then there exists a sparse family $\mathcal{S}$ of dyadic intervals such that 
\[ |\Lambda_\mathbb{P}(f_1,\dots, f_n)|  \lesssim \sum_{Q \in \mathcal{S}}|Q| \prod_{j=1}^n \left( \frac{1}{|Q|}\int_{\mathbb{R}}|f_j|^{s_j} \tilde{\chi}_Q^{M-1}dx\right)^{\frac{1}{s_j}}.\]
\end{proposition}
Notice that Proposition \ref{Mock Banach} says that locally, the multilinear form is bounded by a product of $n$ maximal averages. If we therefore are able to decompose $\mathbb{P}$ into some good collections of well localized collections $\mathbb{P}(Q)$ with $\mathcal{S}:= \bigcup Q$ sparse, we will obtain the desired sparse domination. 
\\ This result is obtained using Theorem 12 from \cite{BeneaMuscalu2021}. The sparse collection is constructed iteratively by $\mathcal{S}=\bigcup_k \mathcal{S}_k$, where $\mathcal{S}_{k+1}= \bigcup_{Q_0 \in \mathcal{S}_k} ch_{\mathcal{S}}(Q_0)$ and $ch_{\mathcal{S}}(Q_0)$ is a set of maximal dyadics intervals $Q \subsetneq Q_0$ such that there exists $P \in \mathbb{P}$ with $I_P \subset Q$ and 
\begin{equation*}\label{eq:constructionofsparse}
    \frac{1}{|Q|}\int_{\mathbb{R}}|f_j(x)|^{s_j}\tilde{\chi}_Q^M (x) dx > C^{s_j} \frac{1}{|Q_0|} \int_\mathbb{R}|f_j(x)|^{s_j}\tilde{\chi}_{Q_0}^{M-1}(x)dx \end{equation*}
for at least one $j=1, \dots n$.
Provided the constant $C$ is large enough, this define a sparse collection and with a good decomposition of the collection $\mathbb{P}$, we obtain the sparse domination.
\\By duality, we now can conclude some $L^p$ boundedness for the operator $T_\mathbb{P}$. 
\begin{theorem}\label{thbanach2ndmethode}
    Let $(\alpha_{s+1},\dots, \alpha_n) \in \Xi_{\Gamma'}$ and $p_1,\dots, p_n$ such that $p_j(1-\alpha_j)>1$ for $j=s+1,\dots, n$, $p_n'\geq 1$, and $p_j>1$ for $j=1,\dots, s$ and such that $\frac{1}{p_1}+ \dots+ \frac{1}{p_n}=1$. Then $T_\mathbb{P}$ extends to a bounded operator: 
    \[ T_\mathbb{P}: L^{p_1}\times \cdots \times L^{p_{n-1}} \rightarrow L^{p_n'}.\]
\end{theorem}
\begin{remark}
    In particular notice that if  $ \Xi_{\Gamma'} \neq \emptyset$, then $T_\mathbb{P}$ is bounded for a range of Lebesgue exponents always including the local $L^2$ case since the $\alpha_j$ have to satisfy $\alpha_j<\frac{1}{2}$. 
    \\ Using limiting arguments, this theorem implies that the operator $T_m$ is bounded for the same range of exponents as $T_\mathbb{P}$. 
\end{remark}
We now consider the quasi Banach case $p_n'<1$, where we cannot solely rely on the estimates on the multilinear form to conclude to $L^p$ boundedness. We therefore introduce an estimate similar as Lemma 26 in \cite{BeneaMuscalu2017} on the operator norm deduced from the estimate \ref{thm19} which will allow us to conclude using the same method as above. 
\begin{lemma} \label{lemma26}
    If $\tau <1$, then for any $\varepsilon>0$ small enough we have 
    \[ \| T_{\mathbb{P}(I_0)}(f_1, \dots, f_{n-1})\mathbb{1}_{F_n} \|_\tau^\tau \lesssim |I_0| \prod_{j=1}^s \widetilde{\mathrm{size}}_{I_0}(\mathbb{1}_{F_j})^{\beta_j\tau} \times \prod_{j=s+1}^{n-1}\widetilde{\mathrm{size}}_{I_0}(\mathbb{1}_{F_j})^{(1-\alpha_j) \tau} \times \widetilde{\mathrm{size}}_{I_0}(\mathbb{1}_{F_n})^{(1-\alpha_n)(1-\varepsilon)}\]
     for all $\beta_1,\dots, \beta_s \in (0,1)$, $(\alpha_{s+1}, \dots, \alpha_n) \in \Xi_{\Gamma'}$ and $f_1, \dots, f_{n-1} \in X(F_1), \dots, X(F_{n-1})$.
\end{lemma} 
This means that up to losing a bit on the last exponent ($(1-\alpha_n) (1 -\varepsilon)< 1-\alpha_n$), we still have a bound for the $L^\tau$ norm of the localized operator. 
\\Now with this estimate we use the same reasoning as in the Banach case as the Mock interpolation and sparse domination methods can both be applied to $\|T_{\mathbb{P}}(f_1,\dots, f_{n-1})\mathbb{1}_{E_n}\|_\tau$, as it was done in \cite{BeneaMuscalu2021}. 
\begin{proposition}\label{Mock quasi Banach}
Let $(\alpha_{s+1},\dots, \alpha_n) \in \Xi_{\Gamma'}$ and $1>\tau>0$ then for any $f_1,\dots, f_{n-1}$ locally integrable functions , $v$ locally $\tau$ integrable, and $s_1,\dots,s_n$ such that $s_j(1-\alpha_j)>1$ for $j=s+1, \dots, n$ and $s_1,\dots, s_s > 1$, we have 
    \[ \| T_{\mathbb{P}(I_0)}(f_1,\dots, f_{n-1})v \|_\tau^\tau \lesssim|I_0| \prod_{j=1}^{n-1}\widetilde{\mathrm{size}}_{I_0}^{s_j}(f_j) \times \widetilde{\mathrm{size}}_{I_0}^{s_n}(v)^\tau.\]
\end{proposition}
We then obtain in a similar way a sparse domination result. 
\begin{proposition} \label{sparse domination quasi Banach}
    Let $1>\tau >0$, $s_1, \dots, s_n$ be as before, $f_1,\dots, f_{n-1}$ some locally integrable functions and $v$ a locally $\tau$ integrable function. Then there exists a sparse family $\mathcal{S}$ such that 
    \[ \|T_\mathbb{P}(f_1,\dots, f_{n-1})v\|_\tau ^\tau \lesssim \sum_{Q \in \mathcal{S}}|Q| \prod_{j=1}^{n-1} \left( \frac{1}{|Q|}\int_{\mathbb{R}}|f_j|^{s_j} \tilde{\chi}_Q^{M-1}dx\right)^{\frac{\tau}{s_j}} \times \left( \frac{1}{|Q|}\int_\mathbb{R}|v|^{s_n}(x)\tilde{\chi}_Q^{M-1}(x) dx\right)^{\tau / s_n}.\]
\end{proposition}
\begin{remark}
    The family $\mathcal{S}$ is constructed as in the Banach case and we apply a similar reasoning using the subadditivity of $\|\cdot\|_\tau^\tau$.
\end{remark}
We now obtain from this sparse estimate, $L^p$ bounds in the quasi Banach case. 
\begin{theorem}\label{thquasibanachsecondemethode}
        Let $(\alpha_{s+1},\dots, \alpha_n) \in \Xi_{\Gamma'}$ and $p_1,\dots, p_n$ such that $p_j(1-\alpha_j)>1$ for $j=s+1,\dots, n-1$, $p_n'< 1$, and $p_j>1$ for $j=1,\dots, s$ and such that $\frac{1}{p_1}+ \dots+ \frac{1}{p_n}=1$. Then $T_\mathbb{P}$ for any $f_1,\dots, f_{n-1}$ locally integrable functions, we have 
    \[ \|T_\mathbb{P}(f_1,\dots, f_{n-1})\|_{p_n'}\lesssim \prod_{j=1}^{n-1} \|f_j\|_{p_j}.\]
\end{theorem}
By some standard limiting arguments, this concludes the proof of Theorem \ref{th2}.
\subsection{Criteria for boundedness and range of exponents}
\subsubsection{Nonemptiness of $\Xi_{\Gamma'}$}
Since Theorem \ref{th2} gives at least local $L^2$ boundedness for the operator $T_m$ once $\Xi_{\Gamma'}$ is not empty, we are interested in what conditions on the singularity allows for that. 
\\ This is only a relevant question when $k\geq 2$ as in the $k=1$ case, one can easily see that an equivalent condition is $n-s>2k$. We also assume that there are no \textit{strongly degenerate indices}, ($s$=0) otherwise one can use the same reasoning by replacing $n$ by $n-s$.
\\ We recall that we denote in what follows by $\mathfrak{d}$ the number of nondegenerate directions. Notice that $\mathfrak{d}\leq \binom{n}{k}$. A first sufficient condition is the following one: 
\begin{proposition} \label{critered}
    If $\displaystyle \mathfrak{d}>  2\binom{n-1}{k-1}$, then $\Xi_{\Gamma'}\neq \emptyset$. 
\end{proposition}
\begin{remark}
    This means that whenever there are enough nondegenerate directions, the operator $T_m$ is bounded at least in the local $L^2$ setting. 
    \\ This is the first step of proving Proposition \ref{caracterisationsolutions_k=2}. 
\end{remark}
\begin{proof}
Assume $\mathfrak{d}>2\binom{n-1}{k-1}$, then one can choose 
    \[\theta_{(i_1,\dots, i_k)}=\frac{1}{\mathfrak{d}}\]
    for any $(i_1,\dots, i_k)$ good direction. 
    \\ Indeed, then for any $j=1,\dots, n$,  \[ \alpha_j=\sum_{\substack{(i_1,\dots,i_k) \text{ good} \\ \exists r\, \; i_r=j}}\theta_{(i_1,\dots, i_k)}=\frac{\sharp \left\{ (i_1,\dots,i_k) \text{ good } |\;\exists r,\, \; i_r=j \right\} }{\mathfrak{d}} \leq \frac{\binom{n-1}{k-1}}{\mathfrak{d}}<\frac{1}{2} \] 
    and therefore the tuple $(\alpha_1,\dots, \alpha_n) \in \Xi_{\Gamma'} $.
  \end{proof}
\begin{remark}
    The above only provides a bound which guarantees that $\Xi_{\Gamma'}$ is not empty. We will see later that that in the case $k=2$, we cannot go any lower than $\mathfrak{d}=\displaystyle 2\binom{n-1}{k-1}=2(n-1)$. This comes from the rigidity of the way the degenerate directions can intersect one another as we explained in Section 2. However in the case $k\geq 3$, the bound given in Proposition \ref{critered} is not optimal as we can see in the following example: 
    \[ \Gamma'_7= \text{Span} \left( \begin{bmatrix}
        1 \\ 0 \\ 0 \\ 1 \\ 0 \\ 0 \\ -2
    \end{bmatrix}, \begin{bmatrix}
        0 \\ 1 \\ 0 \\ 0 \\ 1 \\ 0 \\ -2 
    \end{bmatrix}, \begin{bmatrix}
        0 \\ 0 \\ 1 \\ 0 \\ 0 \\ 1 \\-2
    \end{bmatrix}\right)\]
    In this case $k=3$, $n=7$, so in total there are 35 directions. 15 of them are degenerate directions: $(1,4,j)$ for all $j$, $(2,5,j)$ for all $j$, $(3,6,j)$ for all $j$. This means that the number of nondegenerate directions is $\mathfrak{d}=20<2\binom{n-1}{k-1}=30$. Define
    \begin{equation} \label{eq:thetaexgamma7}
     \theta_{1,2,3}= \theta_{1,5,6}= \theta_{1,2,7}=\theta_{2,4,6}= \theta_{3,4,5}=\theta_{4,6,7}=\theta_{3,5,7}=\frac{1}{7}
     \end{equation}
  and define for the other good directions $\theta_{i_1,i_2,i_3}=0$. 
  \\Then the tuple $(\alpha_1, \dots, \alpha_7)$ defined for $j=1, \dots, 7$ by: 
    \[\displaystyle \alpha_j=\sum_{\substack{(i_1,i_2,i_3) \text{ good}\\ \exists r \; i_r=j }}\theta_{i_1, i_2, i_3}=\frac{3}{7}\] belongs to $\Xi_{\Gamma'_7}$.
  \\ This leads us to define another criterion for which the set $\Xi_{\Gamma'}\neq \emptyset$.  
\end{remark}
\begin{proposition}
    Let $k\geq 3$, $n>2k$. If one can find a set $D$ of $n$ nondegenerate directions such that each index belongs to exactly $k$ directions in $D$, then $\Xi_{\Gamma'} \neq \emptyset$.
\end{proposition}
Notice that this is the case in the example above, as one can see from \eqref{eq:thetaexgamma7}.
\begin{proof}
Denote $D$ such a set. Then define for any $(i_1, \dots, i_k) \in D$, 
\[ \theta_{i_1, \dots, i_k}=\frac{1}{n}\]
and for $(i_1,\dots, i_k)\not \in D$ a good direction, define 
\[ \theta_{i_1,\dots, i_k}=0. \]
Since for any $j=1,\dots, n$, the index $j$ appears exactly $k$ times in the directions in $D$, we have 
\[ \alpha_j:= \sum_{\substack{(i_1, \dots, i_k) \text{ good} \\ \exists r, \; i_r=j} }\theta_{i_1, \dots, i_k} = \frac{k}{n}< \frac{1}{2}.\]
This concludes the proof. 
\end{proof}
To conclude this subsection, we construct an example of $\Gamma' \subset\mathbb{R}^n$ of dimension $k$, such that $\Xi_{\Gamma'} =\emptyset$, for arbitrary large $n$ and $k$ with $n>2k$. 
\begin{example} Let $n$ and $k\geq 2$ be such that $n>2k$. Let

\[\Gamma' = \text{Span} (u_1,\dots, u_k) \subset \mathbb{R}^n\] where $u_1,\dots, u_k$ are the column vectors of the following matrix defined by three blocks:
\[
\begin{array}{@{}c@{\qquad}l@{}}
\left[
\begin{array}{c@{\hspace{2.4cm}}c}
1      &                     \\
\vdots & \mathbf{0}_{(n-k+2)\times(k-1)}         \\
1      &                     \\
\hdashline
0      &                     \\
\vdots & I_{(k-1)}             \\
0      &                     \\
\hdashline
1      &                     \\
\vdots & A+I_{(k-1)}           \\
1      &\\
k-n &
\end{array}
\right]
&
\begin{array}{@{}l@{}}
\left.
\begin{array}{c}
\vphantom{1}\\[-0.15em]
\vphantom{\vdots}\\[-0.15em]
\vphantom{1}
\end{array}
\right\}
\; n-2k+2
\\[0.15em]
\left.
\begin{array}{c}
\vphantom{0}\\[-0.15em]
\vphantom{\vdots}\\[-0.15em]
\vphantom{0}
\end{array}
\right\}
\; k-1
\\[0.15em]
\left.
\begin{array}{c}
\vphantom{1}\\[-0.15em]
\vphantom{\vdots}\\[-0.15em]
\vphantom{1}
\end{array}
\right\}
\; k-1
\end{array}
\end{array}
\]
with $A\in \mathcal{M}_{k-1}(\mathbb{R})$ defined as 
\[
A:=
\begin{bmatrix}
1  & \cdots & 1  \\
\vdots &  & \vdots \\
1  & \cdots & 1  \\
-k & \cdots & -k
\end{bmatrix}.
\]
Here the first block has size $n-2k+2 \times k$, the second one is of size $k-1 \times k$ and the last one is of size $k-1 \times k$. 
\\ What we aimed at for this example, was to have a set of indices that would concentrate all the nondegenerate directions - here the indices that belong to the last two blocks. This yields a lower bound on the corresponding $\alpha_j$. We also required this set of indices to have cardinality $2(k-1)$, as this yields an upper bound on the corresponding $\alpha_j$. Together these bounds imply that $\Xi_\Gamma'=\emptyset$. 
\\ To do so, we first create a block of indices that are very degenerate in the sense that there are no nondegenerate directions with two indices in this block, that is the first block. And then, since the singularity is of dimension $k$, there must at least be one submatrix of rank $k$, which is what we create with the second block (it suffices to consider that block and one line of the above block). In this way we have now $k-1$ indices which concentrate all the nondegenerate directions, which means that we have to add $k-1$ more indices which concentrate the nondegenerate directions. This is how we construct the third block, with the constraint that the last line has to be the opposite of the sum of all the other lines as we want to stay in $\Gamma$. 
\\ With this construction, we obtain that the indices on the last two blocks concentrate all the nondegenerate directions.
\\ Indeed if $(i_1, \dots, i_k)$ is a nondegenerate direction, then we have:
\[ i_2 >n-2k+2\]
If this was not the case then $i_1 ,i_2 \in \left\{1, \dots, n-2k-2\right\}$ but then the lines $i_1$ and $i_2$ in the matrix are the same and therefore $(i_1, \dots, i_k)$ is degenerate. 
\\ In particular, a nondegenerate direction has at least $k-1$ indices in
$\left\{n-2k+3, \dots, n \right\}$. 
We will show that for such a subspace $\Gamma '$, the set $\Xi_{\Gamma'}$ is empty. Indeed if we assume that there exists $(\alpha_1,\dots, \alpha_n)\in \Xi_{\Gamma'}$, then we have that there exists for each good direction $(i_1,\dots, i_k)$, a $\theta_{i_1, \dots, i_k } \in [0, 1]$ such that 
\[ \alpha_j=\sum_{\substack{(i_1, \dots, i_k) \text{ good} \\ \exists r  \; ; \; i_r=j}}\theta_{i_1, \dots, i_k} \in (0, \frac{1}{2}),\]
\[ \sum_{(i_1, \dots, i_k) \text{ good}} \theta_{i_1, \dots, i_k}=1.\]
In particular, if we sum over the $\alpha_j$ for $j$ in the last $2(k-1)$ indices, we get: 
\[k-1 =\frac{2(k-1)}{2}>\sum_{j=n-2k-1}^n \alpha_j \geq (k-1)\times \sum_{(i_1, \dots,i_k) \text{ good}} \theta_{i_1,\dots, i_k}=k-1\]
which is absurd and therefore $\Xi_{\Gamma'}=\emptyset$. 
\end{example}
As a conclusion, we have shown that provided that there are enough nondegenerate directions or provided that $\Gamma'$ has a good structure, then $\Xi_{\Gamma'}\neq \emptyset$. This implies that if $m$ is a Mikhlin symbol for such a singularity $\Gamma'$, $T_m$ is bounded at least in the local $L^2$ range. We have also shown that there exists some singularities $\Gamma'$ such that $\Xi_{\Gamma'}=\emptyset$ and for which Theorem \ref{th2} fails to provide any boundedness. 
\subsubsection{Characterization of the nonemptiness of $\Xi_{\Gamma'}$ for $k=2$}
In this subsection, we completely characterize the nonemptiness of $\Xi_{\Gamma'}$ in the case $k=2$, by proving Proposition \ref{caracterisationsolutions_k=2} and exhibing the possible situations in the case $n-s=6$ and $k=2$. 
\\ As before, we assume that there are no \textit{strongly degenerate indices} to simplify the exposition. 
\\ Proposition~\ref{critered} already shows that if
\[ \mathfrak{d}> 2\binom{n-1}{k-1},\]
then $\Xi_{\Gamma'}\neq\emptyset$. We now prove that, for $k=2$, and $n\neq6$, $\mathfrak{d}= 2\binom{n-1}{k-1}$ is the lowest bound for $\Xi_{\Gamma'}\neq \emptyset$.
\\ The case $k=2$ is easier to study than the general case because of the way the degeneracies can be organized. We first introduce the following definition.
\begin{definition}
    We will say that two directions $(i,j)$ and $(i',j')$ intersect if the sets intersect i.e. $\{ i,j\} \cap \{i',j'\} \neq \emptyset $.
\end{definition}
\begin{remark} \label{intersect}
    If two degenerate directions non trivially intersect (ie $\{ i,j\} \neq \{i',j'\}$), the direction formed by the two indices that didn't coincide is also degenerate. This is a simple observation deduced from Remark \ref{remarquedeg} 
\end{remark}
We are now ready to prove \ref{caracterisationsolutions_k=2}
\begin{proof} (\textit{Proposition \ref{caracterisationsolutions_k=2}})
We split the proof into two cases: $\mathfrak{d}\geq 2(n-1)$ and $\mathfrak{d}<2(n-1)$.
\\ \textbf{The case $\displaystyle \mathfrak{d}\geq 2(n-1)$: } We want to show that $\Xi_{\Gamma'} \neq \emptyset$. By Proposition \ref{critered}, we already have that when $\mathfrak{d}>2(n-1)$, $\Xi_{\Gamma'} \neq \emptyset$. Now assume that $\mathfrak{d}=2(n-1)$. In this case, we use a similar argument as in the proof of Proposition \ref{critered}. We would like to set as before each $\theta$ equal to $\displaystyle \frac{1}{\mathfrak{d}}$, but now some $\alpha_j$ might be equal to $\frac{1}{2}$. Nevertheless, a simple trick of redistributing the weights of specific $\theta$s will allow us to conclude. 
 \\ First, there are at most two indices for which this could happen. Indeed suppose that there are at least three indices $j$ for which 
    \[ \sharp \left\{ (i_1,i_2) \text{ good } |\;\exists r,\, \; i_r=j \right\}\geq \frac{\mathfrak{d}}{2}\]
   and assume that $1, 2, 3$ are such indices. This means that $1,2$ and $3$ are \textit{strongly nondegenerate indices} as the assumption above means that each of them belong to exactly $n-1$ nondegenerate directions. This implies that the number of nondegenerate directions is too large. Indeed it is at least 
   \[ 3(n-1)-3.\]
   This is obtained by simply counting the number of directions containing $1$, $2$ or $3$ and removing the ones counted twice,( $(1,2), (1,3) $ and $(2,3)$ ). Therefore we get 
   \[\mathfrak{d} =2 (n-1) \geq 3(n-1)-3,\]
 which implies
   \[ 4 \leq n.\]
   This is a contradiction since $n>2k=4$. This means that there are at most two indices $j$ which satisfy: 
   \[ \sharp \left\{ (i_1,i-2) \text{ good } |\;\exists r,\, \; i_r=j \right\}\geq \frac{\mathfrak{d}}{2} .\]
   In particular, if there is none, we can define as before 
   $ \displaystyle \theta_{(i_1,i_2)}=\frac{1}{\mathfrak{d}}$ for all $(i_1,i_2)$ good directions and we will get that for all $j=1, \dots, n$
   \[ \alpha_j=\sum_{\substack{(i_1,i_2) \text{ good} \\ \exists r\, \; i_r=j}}\theta_{(i_1,i_2)}=\frac{\sharp \left\{ (i_1,i_2) \text{ good } |\;\exists r,\, \; i_r=j \right\} }{\mathfrak{d}} <\frac{\mathfrak{d}}{2\mathfrak{d}}= \frac{1}{2}\]
   which is enough to conclude to $(\alpha_1, \dots, \alpha_n) \in \Xi_{\Gamma'}$.
   \\ Now if there is only one index, assume it is 1, such that 
  \[ \sharp \left\{ (i_1,i_2) \text{ good } |\;\exists r,\, \; i_r=1\right\}\geq \frac{\mathfrak{d}}{2}=n-1.\]
  This again implies that $1$ is a \textit{strongly nondegenerate index} as the maximum number of nondegenerate directions it can belong to is exactly $n-1$. 
  Then by assumption on the number of nondegenerate directions, there exists a good direction $(i_1,i_2)$ with $i_1 \neq 1$ and $i_2 \neq 1$. Define the following:
  \[ \theta_{1,2}=\frac{1}{2\mathfrak{d}}\]
  \[ \theta_{i_1,i_2}=\frac{1}{\mathfrak{d}}+\frac{1}{2\mathfrak{d}}\]
  \[ \theta_{i_1',i_2'}=\frac{1}{\mathfrak{d}}\]
  for all $(i_1',i_2')$ good direction such that $ (i_1',i_2')\neq (i_1,i_2)$ and $(i_1',i_2') \neq (1,2)$. 
  With this choice we get 
  \[ \alpha_1= \frac{1}{\mathfrak{d}} \times (n-1-1)+\frac{1}{2\mathfrak{d}}= \frac{1}{2}-\frac{1}{2\mathfrak{d}}<\frac{1}{2}\]
  and for all other $j=2,\dots, n$
  \[ \alpha_j\leq  \frac{1}{\mathfrak{d}}(n-1-2)+\frac{1}{\mathfrak{d}}+\frac{1}{2\mathfrak{d}}<\frac{1}{2}.\]
  This is enough to prove that $\alpha\in \Xi_{\Gamma'}$.
  \\ Now if there are two indices $j$, assume $1$ and $2$ such that
  \[  \sharp \left\{ (i_1,i_2) \text{ good } |\;\exists r,\, \; i_r=j\right\}\geq \frac{\mathfrak{d}}{2}=n-1\]
  then the idea is essentially the same. One can restrain a bit $\theta_{1,2}$ and put more weight on another $\theta_{i_1,i_2}$ (which can be achieved because of the number of nondegenerate directions is $\displaystyle 2(n-1)$) so that the sum over all $\theta$ is still 1, and all $\alpha$ are strictly lower than $\frac{1}{2}$. 
  \\ This proves that when $\mathfrak{d}=2(n-1)$, $\Xi_{ \Gamma'} \neq \emptyset$.
 \\ \textbf{The case $\mathfrak{d}<2(n-1)$: } In this case, we aim to prove that $\Xi_{\Gamma'}=\emptyset$. It is enough to prove that there exists one or two indices which concentrate all the nondegenerate directions.   
\\ Indeed assume by contradiction that there exist $\alpha \in \Xi_{\Gamma'}$ and $i<j$ such that $(i,j)$ intersects every nondegenerate directions. Then we have that:
\[ \alpha_i + \alpha_j = \left(\sum_{(i_1, i_2) \text{ good}} \theta_{i_1,i_2}\right) + \theta_{i,j}=1+ \theta_{i,j} \geq 1\]
with the notation $\theta_{i,j}=0$ if $(i,j)$ is a bad direction. This  contradicts the fact that $\alpha_i<\frac{1}{2}$ and $\alpha_j< \frac{1}{2}$. 
\\ \begin{itemize}[label=\textbullet]
    \item We first deal with the case $\mathfrak{d}< 2n-4$. Assume by contradiction that there is no pair of indices that concentrate all the nondegenerate directions. This means for all $1\leq i<j\leq n$, there exists $(i',j')$ nondegenerate direction such that $(i,j)$ and $(i',j')$ do not intersect. We show that in this case $\mathfrak{d}\geq 2n-4$.
\\ Fix a nondegenerate direction $(i,j)$ then for all $r=1,\dots, n$, $r \neq i, j$, either $(i,r)$ is a nondegenerate direction or $(j,r)$ is nondegenerate. Indeed if both $(i,r)$ and $(j,r)$ were degenerate, by Remark \ref{intersect} we would have that $(i,j)$ is degenerate. This means that we at least have $n-1$ nondegenerate directions. 
\\ Let $(i',j')$ be a nondegenerate direction which does not intersect $(i,j)$. Then again, for all $r=1,\dots, n$, $r\neq i', j'$, either $(i',r)$ or $(j',r)$ is nondegenerate. From this, we get $n-1$ new nondegenerate directions minus the ones that already appeared in the first case. Since $(i,j)$ and $(i',j')$ don't intersect, the only directions already possibly counted in the first case are (separately) 
\[ \left[ (i, j') \text{ and } (i, i')\right] \: \text{  or  } \: \left[ (j,j')\text{ and } (j,i') \right]\: \text{ or } \:\left[  (j,j') \text{ and } (i,i')\right] \: \text{ or } \: \left[ (i,j') \text{ and } (j,i')\right].\]
We therefore at most counted two nondegenerate directions twice. This means that we obtain at least $n-1+n-1-2=2n-4$ nondegenerate directions which contradicts the assumption $\mathfrak{d}<2n-4$. 
\\ This proves the following statement: if $\Xi_{\Gamma'}\neq \emptyset$, then $\mathfrak{d}\geq 2n-4$. 
\item We now consider the cases $\mathfrak{d}=2n-4$ and $\mathfrak{d}=2n-3$. We again prove that in these cases two indices concentrate all the nondegenerate directions. To do so we  consider our indices $1, \dots, n$ as vertices of a graph $G_{\Gamma'}$. There exists an edge between two vertices $i$ and $j$ in $G_{\Gamma'}$  if the direction $(i,j)$ is degenerate. The nondegenerate directions are therefore the "missing" edges to obtain a complete graph.
\\ 
We call a connected component of $G_{\Gamma'}$, a set of vertices that are connected by a path to each other. $G_{\Gamma'}$ can be split into disjoint connected components and inside each connected component, there exists an edge between any vertices. Indeed if $i$ and $j$ are two vertices in the same connected component there exists a path $(i, l_1)-(l_1, r_2)-\cdots -(l_r,j) $. In particular it means that $(l_{r-1}, l_{r})$ and $(l_r,j)$ are degenerate directions. Since they intersect, by Remark \ref{intersect} $(l_{r-1},j)$ is also degenerate and therefore there exixts an edge between $l_{r-1}$ and $j$. Iterating the same argument, we obtain that there exists an edge between $i$ and $j$. This proves that each connected component is a complete graph. 
\\ Two indices $i,j$ concentrate all nondegenerate directions if the only missing edges in $G_{\Gamma'}$ are those touching either $i$ or $j$. This means that all the other vertices $\{1, \dots, n\} \backslash \{i,j\}$ must belong in the same connected component. We claim that when  $\mathfrak{d}=2n-4$ or $\mathfrak{d}=2n-3$, this is the case. 
\\ Figure \ref{fig:graphe_n=7} is one example of such graph when for example 
\[ \Gamma'_8= \text{Span}\left( \begin{bmatrix}
    1 \\ -1 \\ 1 \\ 1\\ 1 \\1 \\ -4
\end{bmatrix}  \begin{bmatrix}
    1 \\ -1 \\ 2  \\ 2\\ 2  \\ 2 \\-8
\end{bmatrix}\right)\]
where $n=7$, $k=2$ and  $\mathfrak{d}=2n-3=11$. In this case, we can see that the two vertices $1$ and $2$ concentrate all the nondegenerate directions, so $\Xi_{\Gamma'_8}=\emptyset$.
\begin{figure}[h]
    \centering
\begin{tikzpicture} 
    \node[draw, circle] (5) at ( 2.0 ,  0.0 ) {5}; 
    \node[draw, circle] (4) at ( 1.24,  1.56) {4};
    \node[draw, circle] (3) at (-0.44,  1.94) {3};
    \node[draw, circle] (2) at (-1.80,  0.86){2};
    \node[draw, circle] (1) at (-1.80, -0.86) {1};
    \node[draw, circle] (7) at (-0.44, -1.94) {7};
    \node[draw, circle] (6) at ( 1.24, -1.56){6};
\draw (3)--(4)--(5)--(6)--(7)--(3);
\draw (3)--(5)--(7)--(4)--(6)--(3);
\draw (1)--(2);
\end{tikzpicture}
    \caption{An example: $G_{\Gamma'_8}$}
    \label{fig:graphe_n=7}
\end{figure}
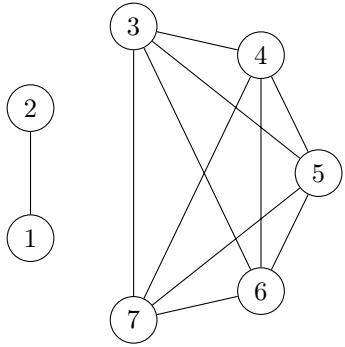
\\ Assume that there are $r$ disjoint connected components $C_1,\dots, C_r$ in $G_{\Gamma'}$ and denote by $s_j$ the number of vertices in the component $C_j$. We have $ \sum_{j=1}^r s_j=n$. Since each pair of vertices in one connected component shares an edge, the number of edges in $G_{\Gamma'}$ is simply 
\begin{equation} \label{countedge1}
    \sum_{j=1}^r\frac{s_j(s_j-1)}{2}.
\end{equation}
On another hand, the number of edges is also the number of degenerate directions given by 
\begin{equation} \label{countedge2}
    \binom{n}{2}-\mathfrak{d}=\frac{1}{2}(n^2-n-2\mathfrak{d}).
\end{equation}
We first prove the statement for $\mathfrak{d}=2n-4$ and then repeat the same process for $\mathfrak{d}=2n-3$. 
\\  Using the fact that $\displaystyle n=\sum_{j=1}^r s_j$, we combine \eqref{countedge1} and \eqref{countedge2} to get:  
\[ \sum_{j=1}^rs_j(s_j-1) = n^2-n-2\mathfrak{d}= n^2-5n+8=\left( \sum_{j=1}^rs_j\right)^2-5\sum_{j=1}^r s_j  \: +  8.\]
Developing 
\[\displaystyle \left(\sum_{j=1}^r s_j \right)^2=\sum_{j=1}^r \left(s_j^2 \:+\: \sum_{i\neq j} s_is_j\right)= \sum_{j=1}^r s_j^2 \:+ \: \sum_{j=1}^r s_j \sum_{i\neq j}s_i= \sum_{j=1}^r s_j^2 \:+\: \sum_{j=1}^r s_j (n-s_j),\]
we get 
\begin{equation}\label{somme_sommets}
\sum_{j=1}^r s_j( 4-n+s_j) =8.\end{equation} 
Notice that the terms in the summand are nonnegative if and only if 
\[ s_j> n-4\]
which leads to consider only very few cases of possible connected components. 
\\ If there are no connected components with $s_j> n-4$, then the right hand side of \eqref{somme_sommets} is negative or vanishes and therefore cannot be equal to 8.
\\ If there is one connected component, assume $C_1$, such that $s_1=n-3$ then the other connected components contain at most $3$ vertices. Thus the possibilities are $s_2=3$, or $s_2=2$ and $s_3=1$ or $s_2=s_3=s_4=1$.
\\ In the first case, there is only one other connected component $C_2$, with $s_2=3$,  \eqref{somme_sommets} becomes 
\[n=5.\]
This implies $s_1=2$ and $s_3=3$ which means that the two vertices in $C_1$ concentrate all the nondegenerate directions.
\\ In the second case, there are now two other connected components $C_2$ and $C_3$ with respectively two and one vertices. Thus, \eqref{somme_sommets} becomes 
\[ n=3\]
which is absurd since $n>2k=4$. 
\\ Finally in the third case, there are three other connected components, each of them composed of only one vertex. In this case \eqref{somme_sommets}) becomes 
\[ n=2\]
which is a contradiction for the same reasons. 
\\ This proves that when there is a connected component with exactly $n-3$ vertices, two indices concentrates the nondegenerate directions and therefore $\Xi_{\Gamma'}=\emptyset$. 
\\ Now if there is a connected component, say $C_1$ such that $s_1=n-2$ or $s_1=n-1$, in particular the one or two vertices left out of $C_1$ concentrate all the nondegenerate directions. This again implies that $\Xi_{\Gamma'}=\emptyset$. One can also do the same computations as above and get that when $\mathfrak{d}=2n-4$, the only graph configuration possible is in fact $s_1=n-2$ and $s_2=2$. 
\\ Assume now that $\mathfrak{d}=2n-3$. A similar reasoning leads us to 
\begin{equation}\label{somme_sommet_2}
    \sum_{j=1}^r s_j( 4-n+s_j) =6.
\end{equation}
Now again, the only achievable graph configurations are when there exists a connected component, assume $C_1$, such that $s_1>n-4$. 
\\ If $s_1=n-3$, then again the other possibilities are $s_2=3$, or $s_2=2$ and $s_3=1$ or $s_2=s_3=s_4=1$. 
\\ In the first case, \eqref{somme_sommet_2} becomes
\[n=6\]
which is a contradiction since we assumed $n\neq 6$. We will explain the configuration that appears in this situation later. 
\\ In the second case, \eqref{somme_sommet_2} becomes \[ n=4\]
which is absurd since $n>2k=4$.
\\ Finally, in the third case \eqref{somme_sommet_2} becomes 
\[n=3\]
which is again a contradiction. 
\\ Once again this implies that one or two vertices concentrate all the nondegenerate directions. In this case, the actual only possible configuration is $s_1=n-2$ and $s_2=s_3=1$. This proves that when $\mathfrak{d}=2n-3$, $\Xi_{\Gamma'}=\emptyset$.
\end{itemize}
In the end we proved that for $\Xi_{\Gamma'}\neq \emptyset$, we need $\mathfrak{d}\geq 2 \binom{n-1}{k-1}$. Combined with Proposition \ref{critered} and the first part of the proof, we get the necessary and sufficient result.
\end{proof}
\begin{remark} \label{remarktriangle}
We would like to discuss what happens in the case $n=6$. 
\\Proposition \ref{critered} states that when $\mathfrak{d}>2(n-1)$, $\Xi_{\Gamma'}\neq \emptyset$. The same proof as above shows that when $\mathfrak{d}=2(n-1)$, $\Xi_{\Gamma'}\neq \emptyset$ and that when $\mathfrak{d}\leq 2n-4$, $\Xi_{\Gamma'}=\emptyset$. The only difference is therefore when $\mathfrak{d}=2n-3=9$. In this particular case two graph configurations are possible. Indeed our proof suggests that we can find some $\Gamma'$ with degeneracies arranged as below either in two triangles ($s_1=3$ and $s_2=3$) or in one component of 4 vertices and two other single vertex components ($s_1=4$, $s_2=1$ and $s_3=1$). 
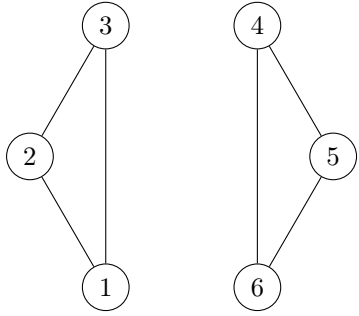
\begin{figure}[h]
    \centering
\begin{tikzpicture} 
    \node[draw, circle] (5) at ( 2.0 ,  0.0 ) {5}; 
    \node[draw, circle] (4) at ( 1.0 ,  1.73) {4};
    \node[draw, circle] (3) at  (-1.0 ,  1.73) {3};
    \node[draw, circle] (2) at (-2.0 ,  0.0 ){2};
    \node[draw, circle] (1) at (-1.0 , -1.73) {1};
    \node[draw, circle] (6) at ( 1.0 , -1.73) {6};
\draw (1)--(2)--(3)--(1);
\draw (4)--(5)--(6)--(4);
\end{tikzpicture}
    \caption{Triangles arrangement n=6}
    \label{fig:graphe_n=6_castriangle}
\end{figure}
\begin{figure}[h]
    \centering
\begin{tikzpicture} 
    \node[draw, circle] (5) at ( 2.0 ,  0.0 ) {5}; 
    \node[draw, circle] (4) at ( 1.0 ,  1.73) {4};
    \node[draw, circle] (3) at  (-1.0 ,  1.73) {3};
    \node[draw, circle] (2) at (-2.0 ,  0.0 ){2};
    \node[draw, circle] (1) at (-1.0 , -1.73) {1};
    \node[draw, circle] (6) at ( 1.0 , -1.73) {6};
\draw (3)--(4)--(5)--(6)--(3)--(5);
\draw (6)--(4);
\end{tikzpicture}
    \caption{Arrangement four sided polygon n=6}
    \label{fig:graphe_n=6_casmarchepas}
\end{figure}
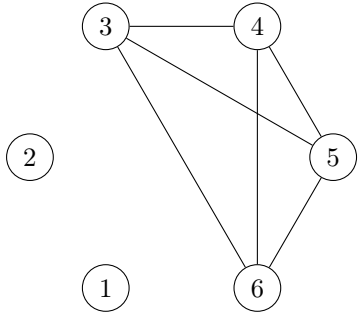
In the case of Figure \ref{fig:graphe_n=6_castriangle}, the set 
$\Xi_{\Gamma '}$ is not empty as each vertex has 3 missing edges and therefore we can choose $\theta_{i,j}=\frac{1}{9}$ which means that each $\alpha_j=\frac{1}{3}<\frac{1}{2}$. 
An example of a singularity for which this is the case is the example $\Gamma_2'$ from the introduction: 
\[ \Gamma'_2= \text{Span}\left(\begin{bmatrix}
   1 \\ 1 \\ -2 \\ 0 \\ 0 \\0 
\end{bmatrix}, \begin{bmatrix}
    0 \\0\\0\\1 \\1\\-2
\end{bmatrix}\right).\]
The boundedness range is computed explicitly in Example \ref{ex2}.
\\On the contrary in the case of Figure \ref{fig:graphe_n=6_casmarchepas}, the vertices $1$ and $2$ concentrate the nondegenerate directions and therefore, the set $\Xi_{\Gamma'}$ is empty. An example of such a singularity $\Gamma '$ is the following 
\[ \Gamma'_4= \text{Span}\left(\begin{bmatrix}
    1 \\-1 \\0 \\0\\0 \\0
\end{bmatrix}, \begin{bmatrix}
    1 \\ 1\\1  \\  1\\ 1\\-5
\end{bmatrix}\right).\]
\end{remark}
This remark and Proposition \ref{caracterisationsolutions_k=2} completely characterize when $\Xi_{\Gamma'}\neq \emptyset$ in the case $k=2$.
\subsubsection{Range of $L^p$ exponents for two examples in the case $k=2$}
In the previous subsection, we were interested in finding some criteria for which the set $\Xi_{\Gamma'}\neq \emptyset$ as this directly gives the local $L^2$ boundedness of the operator $T_m$. We are now focusing on trying to describe the whole range of $L^p$ exponents we obtain. In this subsection, we study two examples for which we can have an explicit range of exponents. 
\\ We will compare the range obtained to the range in the nondegenerate case. To do so we introduce the set $\Xi_{n,k}$ corresponding to $\Xi_{\Gamma'}$ when $\Gamma'$ is nondegenerate. 
\[
\Xi_{n,k}:= \left\{(\alpha_{1},\dots, \alpha_n) \in (0,\frac{1}{2})^{n} \:| \: \exists(\theta_{i_1,\dots, i_k})\in [0,1]^{\binom{n}{k}}\: s.t. \:\: \alpha_j= \sum_{\substack{(i_1, \dots, i_k)  \\ \exists r \text{ s.t. } i_r=j}} \theta_{i_1,\dots, i_k},, 
\sum_{(i_1, \dots, i_k)} \theta_{i_1, \dots i_k}=1 \right\}. \]
\begin{example} \label{ex1}
    First consider the set $\Gamma'_9 \subset\mathbb{R}^5$ defined by: 
    \[\Gamma'_9= \text{Span} \left( \begin{bmatrix}
        1 \\
        1  \\
        1  \\
        1\\
        -4 
    \end{bmatrix}, \begin{bmatrix} 0 \\ 0 \\ 1 \\ 2 \\ -3\end{bmatrix}\right).\]
   Here $n=5$ and $k=2$. In this case, there is only one degenerate direction, which is $(1,2)$ and $\mathfrak{d}=9$.
      \begin{proposition}\label{range1} Let $m$ be a Mikhlin symbol satisfying \eqref{Mikhlin}, for $\Gamma'_9$.
     \\Then,  $T_m$ extends to a bounded operator from $L^{p_1}\times \cdots \times L^{p_{4}} \rightarrow L^{p_5'}$ for all $p_1, \dots, p_5$  such that for all $j=1, \dots, 4$, 
     \[p_j> 1,\]
     and,
    \[ \frac{1}{p_1}+\cdots +\frac{1}{p_5}=1,\]
    \begin{equation} \label{calculXi_1}\frac{1}{p_{i_1}}+\cdots +\frac{1}{p_{i_r}}<\frac{1+r}{2}  \end{equation}
    for all $1\leq i_1<\cdots <i_r\leq 5$ and $r=1, \dots, 5$.
    \end{proposition}
     Notice that this is the same boundedness range as in the nondegenerate case in \cite{MuscaluTaoThiele2002}.
    \\ \begin{proof}
    Proposition \ref{caracterisationsolutions_k=2} and the fact that $\mathfrak{d}=9 \geq 8=2\binom{5-1}{2-1}$ imply that $\Xi_{\Gamma'_9} \neq \emptyset$. Theorem \ref{th2} states that the range in this case are the exponents $p_1, \dots, p_n$ such that 
    \[ \frac{1}{p_1}+\cdots +\frac{1}{p_n}=1,\]
    there exists $\alpha=(\alpha_1, \dots, \alpha_n) \in \Xi_{\Gamma'_9}$ such that for all $j=1,\dots, n$
    \begin{equation} \label{eq:alpha}
     1- \alpha_j>\frac{1}{p_j}.\end{equation}
   We now need to show that this is equivalent to \eqref{calculXi_1}. 
   \\ First, using that $\displaystyle \sum_{j=1}^n\alpha_j=k=2$ and that each $\alpha_j<\frac{1}{2}$, it is easy to see that any Hölder tuple $(p_1, \dots, p_n)$ satisfying \eqref{eq:alpha} satisfies \eqref{calculXi_1}.
   \\ Now to prove the converse, given a Hölder tuple $(p_1, \dots,p_5)$ such that $p_j>1$ for all $j=1, \dots, 4$ and satisfying \eqref{calculXi_1}, we want to show that there exists an $\alpha \in \Xi_{\Gamma'_9}$ such that the tuple satisfies \eqref{eq:alpha} for all $j=1, \dots, 5$. For any $\alpha \in \mathbb{R}^5$ we denote $\tilde{\alpha}:=(1-\alpha_1, \dots, 1-\alpha_5)$. Since $\Xi_{\Gamma_8'}$ is convex and the map $\alpha \mapsto \tilde{\alpha}$ is affine, it is enough to show that the point $(0,\frac{1}{2},\frac{1}{2},\frac{1}{2}, \frac{1}{2})$ and permutations of its coordinates are boundary points of $\Xi_{\Gamma'_9}$. Indeed, assume this is the case, and define for $j=1, \dots, 5$
\[\tilde{\alpha}_j=\max(\frac{1}{2},\frac{1}{p_j}).\]
   Then using \eqref{calculXi_1}, we get 
   \[ \sum_{j=1}^5 \tilde{\alpha}_j<3\]
Therefore, we can enlarge each $\tilde{\alpha}_j$ such that $\tilde{\alpha}_j \in (\frac{1}{2},1)$ with $\tilde{\alpha}_j> \frac{1}{p_j}$ and $\displaystyle \sum_{j=1}^5\tilde{\alpha}_j=3$. Notice that such $\tilde{\alpha}$ will satisfy for any $1\leq i_1 <\cdots< i_r\leq 5$ and any $r=1,\dots,5$:
\[ \tilde{\alpha}_{i_1}+ \cdots +\tilde{\alpha}_{i_r}<\frac{1+r}{2}\]
and $\displaystyle \sum_{j=1}^5 \tilde{\alpha}_j=3$. From Lemma \ref{lemmaconv}, this means that $\tilde{\alpha}$ is in the interior of the convex hull of the point $(1, \frac{1}{2}, \frac{1}{2}, \frac{1}{2}, \frac{1}{2})$. In particular $\alpha:=(1-\tilde{\alpha}_1, \dots, 1-\tilde{\alpha}_5)$ is in the interior of the convex hull of the point $(0,\frac{1}{2}, \frac{1}{2}, \frac{1}{2}, \frac{1}{2})$ and its permutations. Since $\Xi_{\Gamma'_9}$ is convex and we assumed that these points are on its boundary, it shows that $\alpha \in \Xi_{\Gamma'_9}$. 
\\ We are now only left with proving that $(0,\frac{1}{2}, \frac{1}{2}, \frac{1}{2}, \frac{1}{2})$ and its permutations are on the boundary of $\Xi_{\Gamma'_9}$. For an $\alpha \in \Xi_{\Gamma'_9}$, we have: 
\[\left\{\begin{aligned}\alpha_1&= \theta_{1,3}+\theta_{1,4}+\theta_{1,5} ,\\
\alpha_2&= \theta_{2,3}+\theta_{2,4}+\theta_{2,5}, \\
\alpha_3&= \theta_{1,3}+\theta_{2,3}+\theta_{3,4} +\theta_{3,5}, \\
\alpha_4&= \theta_{1,4}+\theta_{2,4}+\theta_{3,4} +\theta_{4,5}, \\
\alpha_5&= \theta_{1,5}+\theta_{2,5}+\theta_{3,5} +\theta_{4,5}, \\       
    \end{aligned}\right.\]
    where $\theta_{i,j}\in [0,1]$ such that $\displaystyle \sum_{\substack{1\leq i<j\leq n \\ (i,j) \neq (1,2)}}\theta_{i,j}=1$. Compared with the nondegenerate case, the system above differs only by the absence of the parameter $\theta_{1,2}$. This does not affect the resulting admissible region, since $\Xi_{5,2}$ and $\Xi_{\Gamma'_9}$ have the same extreme points, namely the permutations of $(\frac{1}{2}, \frac{1}{2}, \frac{1}{2}, \frac{1}{2}, 0)$. Indeed for any $\varepsilon>0$ small enough, one can pick 
    \[ \theta_{1,5}=\theta_{2,5}=\theta_{3,5}=\theta_{4,5}=\varepsilon,\]
    \[ \theta_{1,3}=\theta_{1,4}=\theta_{2,3}=\theta_{2,4}=\frac{1}{4}-\varepsilon,\]
    \[\theta_{3,4}=0\]
    such that 
    \[\alpha_1=\alpha_2=\alpha_3=\alpha_4=\frac{1}{2}-\varepsilon,\]
    and 
    \[\alpha_5=5\varepsilon.\]
    One can do a similar choice of $\theta_{i,j}$ for $(\frac{1}{2},\frac{1}{2},\frac{1}{2},0,\frac{1}{2})$ and $(\frac{1}{2},\frac{1}{2},0,\frac{1}{2},\frac{1}{2})$. Now for $(0,\frac{1}{2},\frac{1}{2},\frac{1}{2},\frac{1}{2})$, one can pick for any $\varepsilon>0$ small enough 
    \[ \theta_{1,3}=\theta_{1,4}=\theta_{1,5}=\varepsilon,\]
    \[ \theta_{2,3}=\theta_{2,4}=\theta_{2,5}=\theta_{3,4}=\theta_{3,5}=\theta_{4,5}=\frac{1}{6}-\frac{1}{2}\varepsilon,\]
    so that 
    \[ \alpha_1=3\varepsilon,\]
    \[\alpha_2=\frac{1}{2}-\frac{3}{2}\varepsilon,\]
    \[\alpha_3=\alpha_4=\alpha_5=\frac{1}{2}-\frac{1}{2}\varepsilon.\]
This concludes the proof. 
\end{proof}
\end{example}
In the first example, we notice that we find the exact same range of $L^p$ exponents as in the nondegenerate case. This is deduced from the fact that the boundary of each $\Xi$ contains the boundary points with as many $\alpha_j=0$ as possible. In some sense, the degeneracy of the first example is "small". In the second example we will however see that it is not always the case as the range of $L^p$ exponents we will exhibit is strictly included in the range for the nondegenerate case. 
\begin{example} \label{ex2}
    We now consider $\Gamma_2'$: 
    \[ \Gamma'_2= \text{Span}\left(\begin{bmatrix}
    1 \\1 \\-2 \\0 \\0 \\0
\end{bmatrix}, \begin{bmatrix}
    0 \\ 0 \\ 0 \\ 1 \\ 1 \\-2
\end{bmatrix}\right), \]
which is our example from the introduction and of the case $n=6$, with the number of nondegenerate directions $\mathfrak{d}=2n-3=9$ that corresponds to the graph \ref{fig:graphe_n=6_castriangle}. 

From this we deduce as we did in the first example from Theorems \ref{thbanach2ndmethode} and \ref{thquasibanachsecondemethode}
\begin{proposition}
Let $m$ be a Mikhlin symbol satisfying \eqref{Mikhlin}, for $\Gamma_2'$. Then the operator $T_m$ extends to a bounded operator
\[T_m: L^{p_1} \times \cdots \times L^{p_5} \rightarrow L^{p_6'}\]
whenever
\[ \frac{1}{p_1}+\cdots+ \frac{1}{p_6}=1,\]
\[ p_j>1 \text{ for }j=1,\dots,5,\]
\begin{equation} \label{calculXi_2}
\frac{1}{p_{i_1}}+\frac{1}{p_{i_2}}<\frac{3}{2} \text{ for all } 1\leq i_1< i_2\leq 3 \text{ and } 4\leq i_1< i_2\leq 6, \end{equation}
\begin{equation} \label{calculXi_3}
    \frac{1}{p_1}+\frac{1}{p_2}+\frac{1}{p_3}<2, \text{ and } \frac{1}{p_4}+\frac{1}{p_5}+\frac{1}{p_6}<2.\end{equation}
\end{proposition}
One can see that in particular, the result does not include the boundedness
\[ L^{1+\varepsilon} \times L^{1+\varepsilon} \times L^2 \times L^2\times L^2\longrightarrow L^{\frac{2(1+\varepsilon)}{7+3\varepsilon}}\]
for any $\varepsilon>0$ small enough, which holds in the nondegenerate case. This is explained by some missing boundary points in $\Xi_{\Gamma'_2}$ in comparison to $\Xi_{6,2}$.
\begin{proof}
From Remark \ref{remarktriangle}, we already know that $\Xi_{\Gamma'_2} \neq \emptyset$.
Theorem \ref{th2} states that $T_m$ is bounded from $L^{p_1} \times \cdots \times L^{p_5} \rightarrow L^{p_6'}$ for any tuple $(p_1, \dots, p_6)$ such that \[ \frac{1}{p_1}+\cdots+ \frac{1}{p_6}=1,\]
\[ p_j>1 \text{ for }j=1,\dots,5,\]
and such that there exists $\alpha\in \Xi_{\Gamma'_2}$ such that 
\[ \frac{1}{p_j}< 1- \alpha_j.\]
 We are now interested in describing the full range obtained and show that it is equivalent to the conditions \eqref{calculXi_2} and \eqref{calculXi_3}. 
 \\ First notice that by definition of $\Xi_{\Gamma'_2}$, the Hölder tuples in Theorem \ref{th2} satisfy \eqref{calculXi_2}, \eqref{calculXi_3}. 
 \\ To prove that any Hölder tuple satisfying  \eqref{calculXi_2}, \eqref{calculXi_3} can be recovered from Theorem \ref{th2}, we proceed as before. First we compute the extreme points of $\Xi_{\Gamma'_2}$ and then the convex hull of the points $\tilde{\alpha}:=(1-\alpha_1, \dots, 1-\alpha_6)$ for $\alpha$ extreme point of  $\Xi_{\Gamma'_2}$. This will be enough to conclude, with a similar reasoning as in the proof of Proposition \ref{range1}, to the range of boundedness \eqref{calculXi_2} and \eqref{calculXi_3}.
\\ We first compute the extreme points of $\Xi_{\Gamma'_2}$. The degenerate directions for $\Gamma'_2$ are $(1,2), (1,3), (2,3)$ and $(4,5),(5,6), (4,6)$. Thus any $\alpha \in \Xi_{\Gamma'_2}$ satisfies 
\[ \left\{\begin{aligned} \alpha_1&= \theta_{1,4}+\theta_{1,5}+\theta_{1,6}, \\
\alpha_2&= \theta_{2,4}+\theta_{2,5}+\theta_{2,6},\\
\alpha_3&= \theta_{3,4}+\theta_{3,5}+\theta_{3,6}, \\
\alpha_4&= \theta_{1,4}+\theta_{2,4}+\theta_{3,4}, \\
\alpha_5&= \theta_{1,5}+\theta_{2,5}+\theta_{3,5}, \\
\alpha_6&= \theta_{1,6}+\theta_{2,6}+\theta_{3,6} \; . \\
\end{aligned} \right.\]
As we have said earlier, some boundary points from $\Gamma_{6,2}$ cannot be reached in $\Xi_{\Gamma'_2}$. Indeed, the point $(\frac{1}{2}, \frac{1}{2}, \frac{1}{2}, \frac{1}{2}, 0,0)$ and its permutations are boundary points of $\Xi_{6,2}$. However, if $(0, 0, \frac{1}{2}, \frac{1}{2}, \frac{1}{2}, \frac{1}{2})\in \bar{\Xi}_{\Gamma'_2}$, then there would exist a sequence $(\alpha^n)_n\in \Xi_{\Gamma'_2}^\mathbb{N}$ converging towards it. In particular we would have
\[ \alpha_6^n=\theta_{1,6}^n+\theta_{2,6}^n+\theta_{3,6}^n \longrightarrow 0 \]
and since each $\theta_{i,j} \geq 0$, we have that for $j=1,2,3$ 
$(\theta_{j,6}^n)_n$ converges to zero. Using the same argument, we get a similar result for any $\theta_{j,5}$, $j=1, 2,3$. 
\\ This implies since $\alpha^n_1= \theta_{1,4}+\theta_{1,5}+\theta_{1,6}\longrightarrow \frac{1}{2}$, that $(\theta_{1,4}^n)_n$ converges to $\frac{1}{2}$. Again with the same argument we get that $(\theta_{2,4}^n)_n$ and $(\theta_{3,4}^n)_n$ both converge to $\frac{1}{2}$. But this would mean that 
\[ \alpha_4^n=\theta_{1,4}^n +\theta_{2,4}^n +\theta_{3,4}^n \longrightarrow\frac{3}{2}\]
which contradicts the assumption that it converges to $\frac{1}{2}$. This proves that the point $(0, 0, \frac{1}{2}, \frac{1}{2}, \frac{1}{2}, \frac{1}{2})\notin \bar{\Xi}_{\Gamma'_2}$.
\\ By a similar reasoning, one can show that any boundary point with two zeros in either the first three coordinates or the last three does not belong to $\bar{\Xi}_{\Gamma'_2}$. 
\\ On the other hand, the point $(0, \frac{1}{2},\frac{1}{2},\frac{1}{2}, \frac{1}{2}, 0) \in \bar{\Xi}_{\Gamma'_2}$. Indeed, let any $\varepsilon>0$ and define 
\[ \theta_{1,4}=\theta_{1,5}=\theta_{1,6}=\theta_{2,6}=\theta_{3,6}=\varepsilon,\]
\[ \theta_{2,4}=\theta_{2,5}=\theta_{3,4}=\theta_{3,5}= \frac{1}{4}-\frac{5}{4}\varepsilon.\]
Then we get 
\[\alpha_1=\alpha_6=3\varepsilon,\]
\[ \alpha_2=\alpha_3=\alpha_4=\alpha_5=\frac{1}{2}-\frac{1}{4}\varepsilon.\]
Then $\alpha=(\alpha_1, \dots,\alpha_6)\in \Xi_{\Gamma'_2}$, for any $\varepsilon$ small enough. This proves that $(0, \frac{1}{2},\frac{1}{2},\frac{1}{2}, \frac{1}{2}, 0) \in \bar{\Xi}_{\Gamma'_2}$.
\\ With a similar reasoning we prove that the permutations of  $(0, \frac{1}{2},\frac{1}{2},\frac{1}{2}, \frac{1}{2}, 0)$ such that there is exactly one zero in the the three first coordinates belong to $\Xi_{\Gamma'_2}$.
\\ We now compute the convex hull $C$ of the points  $(1,\frac{1}{2},\frac{1}{2}, \frac{1}{2}, \frac{1}{2}, 1)$ and its permutations such that exactly one of the three first coordinates is equal to $1$. We claim that this convex hull is the set $B$ of $\beta \in \mathbb{R}^6$ such that 
\[\beta_{i_1}+\beta_{i_2} \leq \frac{3}{2},\]
\[ \beta_1 +\beta_2 +\beta_3=2,\]
\[ \beta_4 +\beta_5+ \beta_6=2,\]
for all $\{i_1, i_2 \}\subset \{1, 2,3\}$ or $\{i_1, i_2\} \subset\{4,5,6\}$.
Any $\beta$ in the convex hull $C$ satisfies the relations above. Conversely, we only need to prove that the extreme points of $B$ are exactly the permutations of $(1,\frac{1}{2},\frac{1}{2}, \frac{1}{2}, \frac{1}{2}, 1)$. Let $(\beta_1, \dots, \beta_6)$ be an extreme point $B$. By symmetry, we can assume that $\beta_1 \geq \beta_2 \geq \beta_3$ and $\beta_4 \geq \beta_5 \geq \beta_6$. Furthermore, if we assume that
\[ \beta_1 +\beta_2 <\frac{3}{2},\]
we get that $\beta':=(\beta_1,\beta_2+\varepsilon, \beta_3-\varepsilon, \beta_4, \beta_5, \beta_6)$ is also in $B$ if $\varepsilon$ is small enough. It is also the case for $\beta'':=(\beta_1,\beta_2-\varepsilon, \beta_3+\varepsilon, \beta_4, \beta_5, \beta_6)$
and therefore $\beta=\frac{1}{2}\beta'+\frac{1}{2}\beta''$ which contradicts the fact that it is an extreme point. Therefore $\beta_1+\beta_2=\frac{3}{2}$, which gives \[ \beta_3=\frac{1}{2}.\]
One can do a similar reasoning if we assume that 
\[ \beta_1<1\]
to conclude to the fact that $\beta_1=1$ and therefore $\beta_2=\frac{1}{2}$. 
By a similar argument, we can show that we also get 
\[\beta_4=1 \;\text{ and } \; \beta_5=\beta_6=\frac{1}{2}.\]
This proves that $C=B$. 
\\ This is enough to conclude in the same way as we did in the proof of Proposition \ref{range1}. 
\end{proof}
\end{example}
Examples \ref{ex1} and \ref{ex2} illustrate different behaviors of the boundedness range obtained from Theorem \ref{th2}. In the first case, there is only one degenerate direction and we show that it does not change the range of boundedness in comparison to the nondegenerate case. In the second case, however, the singularity has more degenerate directions and we get a boundedness range strictly included in the range obtained in the nondegenerate case. 
\section{A counter example in the case $k=1$}
In this section we construct a counterexample to a limit case of Theorems \ref{th1} and \ref{th2} when the number of nondegenerate directions is equal to two. To be more precise we extend a counter-example from \cite{Kesler2019} where the degenerate line is $\Gamma'=\left\{ (\xi,-\xi,0) \; |\; \xi \in \mathbb{R}\right\}$ in $\mathbb{R}^3$, to the more general case $\Gamma'=\left\{ (\xi,-\xi,0,\dots,0)\; |\;\xi \in \mathbb{R}\right\} \subset \mathbb{R}^n$. This shows that in any dimension, when the singularity is a line (i.e. $k=1$), the condition $n-s>2k$ is sharp. 
\\ We first begin by constructing an explicit counter example when $n=4$. Let $\Gamma'=\left\{(\xi,-\xi,0,0)\:|\: \xi \in \mathbb{R} \right\}$, degenerate in the two last directions. The idea is to construct a symbol $m$, sum of functions supported on Whitney cubes adapted to $\Gamma'$, which will fix the last two directions around $0$ and notice that even if these directions are fixed we can have too many Whitney cubes because the other directions aren't. This also explains why we leave out the degenerate directions in our proof of Theorem \ref{th1} as we encounter the same issue. We use the same cubes as in \cite{Kesler2019} while fixing the new direction. Since this direction is degenerate, this modification has little impact on the argument. One can see this as if we were projecting the singularity in the hyperplan $\left\{ \xi_3=0\right\}$. 
\\ Let's now dive into the details. Let $\rho \in \mathcal{C}^\infty([-\frac{1}{4},\frac{1}{4}])$ real, even such that $\hat{\rho}(0)>0$ and define $\phi=\rho \ast \rho$. Then $\phi \in  \mathcal{C}^\infty([-\frac{1}{2},\frac{1}{2}])$, $\phi$ is real, even, $\hat{\phi}(0)=\hat{\rho}(0)^2>0$ and $\hat{\phi}(x)=\hat{\rho}(x)^2 \geq 0$ for all $x \in \mathbb{R}$. Let $A\gg1$, then for $k\geq 8$, $m \in \mathbb{Z}$ and $-2^{k-8}<\lambda<2^{k-8}$, we define: 
\[ Q_{k,m,\lambda}^1=2^{-k}[2^km+\lambda-\frac{1}{2}, 2^km+\lambda +\frac{1}{2}],\]
\[  Q_{k,m,\lambda}^2=2^{-k}[-2^km-\lambda-\frac{1}{2}+A,- 2^km-\lambda +\frac{1}{2}+A],\]
\[ Q_{k}^3=2^{-k}[-\frac{1}{2},\frac{1}{2}].\]
Then the collection of cubes $Q_{k,m,\lambda}= Q_{k,m\lambda}^1\times Q_{m,k,\lambda}^2\times Q_k^3$ is similar to a collection of dyadic Whitney cubes, as in \eqref{eq:Whitneycubes}, relatively to $\left\{ (\xi,-\xi,0)\;| \; \xi \in \mathbb{R} \right\}$ with $C_0 \sim A$. We now consider the collection of cubes 
\[ \mathbb{Q}=\bigcup_{k\geq 8} \bigcup_{m \in \mathbb{Z}}\bigcup_{-2^{k-8}<\lambda< 2^{k-8}}Q_{k,m,\lambda}.\]
Notice that for a fixed $k$, we are looking at a collection of Whitney dyadic cubes that we shift with a large constant in the $m$ quantity and when $m$ is fixed that we shift by units in the $\lambda$ quantity. This in particular means our collection is not sparse. We now define 
\[ \eta_{k,m,\lambda}^1(x)=2^{-k}\hat{\phi}(2^{-k}x)e^{2\pi i(m+\lambda2^{-k})x},\]
\[ \eta_{k,m,\lambda}^2(x)=2^{-k}\hat{\phi}(2^{-k}x)e^{-2\pi i(m+\lambda2^{-k}-A2^{-k})x}e^{2i\pi A2^{-k}m}, \]
\[ \eta_k^3(x)=2^{-k}\hat{\phi}(2^{-k}x).\]
Notice that $\hat{\eta}_{k,m,\lambda}^j$ is supported on $Q_{k,m,\lambda}^j$ for $j=1,2$ and $\hat{\eta}_k^3$ is supported on $Q_k^3$. We now define our symbol $\mathfrak{m}\in \mathcal{M}_{\Gamma'}(\mathbb{R}^n)$ by:
\[ \mathfrak{m}(\xi_1,\xi_2,\xi_3)=\sum_{k\geq8} \underbrace{\sum_{|Q|=2^{-k}}\hat{\eta}^1_{k,m,\lambda}(\xi_1)\hat{\eta}^2_{k,m,\lambda}(\xi_2)\hat{\eta}_k^3(\xi_3)}_{:=\mathfrak{m}_k(\xi_1,\xi_2,\xi_3)} .\]
We claim that the operator $T_\mathfrak{m}$ defined by 
\[ T_{\mathfrak{m}}(f_1,f_2,f_3)(x)=\int_{\mathbb{R}^3}\mathfrak{m}(\xi_1,\xi_2,\xi_3)\hat{f}_1(\xi_1)\hat{f}_2(\xi_2) \hat{f}_3(\xi_3) e^{2i\pi x(\xi_1+\xi_2+\xi_3)}d\xi_1 d\xi_2 d\xi_3\]
is not bounded from any $L^{p_1} \times L^{p_2}\times L^{p_3} $ to $L^{p_4'}$ with $1<p_1,p_2,p_3,p_4'<\infty$ such that $\frac{1}{p_1}+\frac{1}{p_2}+\frac{1}{p_3}=\frac{1}{p_4'}$. To prove our claim we now need to find some good functions $f_1$, $f_2$, $f_3$. 
\\ Let $\delta =\frac{1}{100}$ and $\psi \in \mathcal{S}(\mathbb{R})$ such that:
\[ \mathbb{1}_{[-\frac{1}{2}+\delta, \frac{1}{2}-\delta]}\leq \hat{\psi} \leq \mathbb{1}_{[-\frac{1}{2}, \frac{1}{2}]} \]
and define for $N\gg1$, 
\[ f_1^N(x)=\sum_{1\leq n\leq N}\psi(x-n)e^{2i\pi nx},\]
\[ f_2^N(x)=\sum_{1\leq n\leq N}\psi(x-n)e^{-2i\pi nx},\]
\[ f_3^N(x)=\sum_{1\leq n\leq N} \psi(x-n).\]
We essentially chose these functions to localize our operator on the spatial interval $[1,N]$. Indeed notice that if $Q_{k,m,\lambda}^1 \cap [n_1-\frac{1}{2}, n_1+\frac{1}{2}] \neq \emptyset$, $Q_{k,m,\lambda}^2 \cap [n_2-\frac{1}{2}, n_2+\frac{1}{2}] \neq \emptyset$, then $n_1=m=-n_2$ and for all $k\geq C_{A}$, we have 
\[ Q_{k,m,\lambda}^1 \subset [m-\frac{1}{2}+\delta, m+\frac{1}{2}-\delta], \]
\[-Q^2_{k,m,\lambda}= 2^{-k}[2^km+\lambda-\frac{1}{2}-A,2^km+\lambda +\frac{1}{2}-A]\subset [m-\frac{1}{2}+\delta, m+\frac{1}{2}-\delta].\]
Fix now $k\geq 8$ and denote \[\mathbb{Q}^{k}:= \bigcup_{m\in \mathbb{Z}} \bigcup_{-2^{k-8}<\lambda<2^{k-8}}Q_{k,m,\lambda} \]
Then for a fixed $x\in \mathbb{R}$, we have: 
\begin{align*}
    T_{\mathfrak{m}_k}(f_1^N,f_2^N,f_3^N)(x)&=\int_{\mathbb{R}^3} \mathfrak{m}_k(\xi_1,\xi_2,\xi_3)e^{2i\pi x(\xi_1+\xi_2+\xi_3) }\hat{f_1}^N(\xi_1)\hat{f}_2^N(\xi_2) \hat{f}_3^N(\xi_3) d\xi_1 d\xi_2 d\xi_3 \\
    &= \sum_{Q_{k,m, \lambda} \in \mathbb{Q}^{k}} \prod_{j=1}^3 \int_{\mathbb{R}}\hat{\eta}_{k,m,\lambda}^j(\xi_j)\hat{f}_j^N(x)e^{2i \pi x\xi_j}d\xi_j \\
    &= \sum_{Q_{k,m, \lambda} \in \mathbb{Q}^{k}} \prod_{j=1}^3 \eta_{k,m,\lambda}^j \ast f_j^N(x) \\
    &=\left(\eta_k^3\ast f_3^N (x) \right) 
    \\ &\times\left(\sum_{\substack{m\in \mathbb{Z}\\ -2^{k-8}<\lambda<2^{k-8}}} \sum_{1\leq n_1,n_2\leq N} \eta_{k,m,\lambda}^1\ast (\psi(\cdot -n_1)e^{2i\pi n_1 \cdot}) (x) \times \eta_{k,m,\lambda}^2\ast (\psi(\cdot -n_2)e^{2i\pi n_2\cdot}) (x) \right)\\
    &=\left( \eta_k^3\ast f_3^N (x) \right) \times \left(  \sum_{\substack{1\leq m\leq N\\ -2^{k-8}<\lambda<2^{k-8}}}2^{-2k}\hat{\phi}((x-m)2^{-k})^2e^{2i\pi A 2^{-k}x} \right)\\
    &=\left(2^{-k} \sum_{1\leq m\leq N }\hat{\phi}(2^{-k}(x-m))\right) \times \left(e^{2i\pi A2^{-k}x}  (2^{k-7}-1)2^{-2k}\sum_{1\leq m\leq N} \hat{\phi}((x-m)2^{-k})^2\right),
\end{align*}
where $2^{k-7}-1 = \sharp \left\{\lambda \;| \; -2^{k-8}< \lambda  < 2^{k-8} \right\}$. 
We now want to find a lower bound for $|T_{\mathfrak{m}_k}(f_1^N,f_2^N,f_3^N)(x)|$. To do so we will first give a lower bound for the second bracket which comes from the counterexample from \cite{Kesler2019}. Since $\hat{\phi}(0)>0$, there exists an $\varepsilon>0$ such that $\hat{\phi}(x)\geq c>0$ for all $x\in [-\varepsilon,\varepsilon]$.Therefore for $m=1,\dots, N $, we have: 
\[ \hat{\phi}((x-m)2^{-k})^2\gtrsim \mathbb{1}_{[m-2^k\varepsilon,m+2^k\varepsilon]}(x).\]
Summing over $m$ we get: 
\[ \sum_{m=1}^N\hat{\phi}((x-m)2^{-k})^2\gtrsim \sum_{m=1}^N\mathbb{1}_{[m-2^k\varepsilon,m+2^k\varepsilon]}(x)\gtrsim \sum_{m=1}^N 2^k\mathbb{1}_{[m,m+1]}(x)\gtrsim_\varepsilon 2^k \mathbb{1}_{[1,N]}(x)\]
where the two last inequalities come from the fact that when $2^k \varepsilon \leq \frac{N}{2}$, and $x\in[m,m+1]$, then $x$ belongs to $\sim 2^k\varepsilon$ intervals of the kind $[m+j-2^k\varepsilon, m+j+2^k\varepsilon]$ for some $j$. 
\\ Now for the first bracket $ \displaystyle 2^{-k} \sum_{m=1}^N \hat{\phi}((x-m)2^{-k})$ we use the same reasoning using the property that $\hat{\phi}=\hat{\rho}^2\geq 0$. We get that $ \displaystyle |  2^{-k} \sum_{m=1}^N \hat{\phi}((x-m)2^{-k})| \gtrsim 2^{-k+k}\mathbb{1}_{[1,N]}(x)\gtrsim \mathbb{1}_{[1,N]}(x)$. This gives us the lower bound 
\[ |T_{\mathfrak{m}_k}(f_1^N,f_2^N,f_3^N)(x)|\gtrsim \mathbb{1}_{[1,N]}(x)\]
for $C_A\leq k \lesssim \log(N)$. Last but not least, if we take $A=100$, we get 
\[ \text{supp}(\mathcal{F}(T_{\mathfrak{m}_k}(f_1^N,f_2^N,f_3^N))) \subset [98\times 2^{-k},102 \times 2^{-k}]. \]
Using the Littlewood-Paley square function, we get 
\begin{align*}
    \| T_{\mathfrak{m}}(f_1^N,f_2^N,f_3^N)\|_{p_4'}&=\|\sum_{k\geq 8} T_{\mathfrak{m}_k}(f_1^N,f_2^N,f_3^N)\|_{p_4'} \\
    &\gtrsim \left\|\left( \sum_{k\geq 8} (T_{\mathfrak{m}_k}(f_1^N,f_2^N,f_3^N)^2\right)^{1/2} \right\|_{p_4'}\\
    &\gtrsim \left\|\left(\sum_{C_A\leq N\lesssim\log(N)} |T_{\mathfrak{m}_k}(f_1^N,f_2^N,f_3^N)|^2 \right)^{1/2}\right\|_{p_4'} \\
    &\gtrsim \log(N)^{1/2} N^{1/p_4'} \\
    &\gtrsim \log(N)^{1/2}N^{1/p_1 +1/p_2 +1/p_3}.
\end{align*}
Notice now that $\|f_j^N\|_{p_j}\lesssim N^{1/p_j}$ which is enough to conclude that $T_\mathfrak{m}$ is not bounded from $L^{p_1}\times L^{p_2} \times L^{p_3}$ to $L^{p_4'}$ for all $1<p_1,p_2,p_3,p_4'<\infty$ such that $\frac{1}{p_1}+\frac{1}{p_2}+\frac{1}{p_3}=\frac{1}{p_4'}$. 
\\ We now claim that this counterexample can be extended to the general case $\Gamma'=\left\{ (\xi,-\xi,0,\dots, 0) \;| \;\xi \in \mathbb{R}\right\}\subset \mathbb{R}^n$. It suffices to look at the collection \[ \mathbb{Q}=\bigcup_{k\geq 8}\bigcup_{m\in \mathbb{Z}}\bigcup_{-2^{k-8}<\lambda<2^{k-8}} Q_{k,m,\lambda}^1 \times Q_{k,m,\lambda}^2 \underbrace{\times Q_k^3 \times \dots \times Q_k^3}_{(Q_k^3)^{n-2}},\]
and define a new symbol $\mathfrak{m}$ by:
\[ \mathfrak{m}(\xi_1,\dots,\xi_{n-1})=\sum_{k\geq 8}\sum_{|Q|=2^{-k}}\hat{\eta}_{k,m,\lambda}^1(\xi_1)\hat{\eta}_{k,m,\lambda}^2(\xi_2)\prod_{j=3}^{n-1} \hat{\eta}_k^3(\xi_j).\]
Define the same functions $f_1^N$, $f_2^N$, $f_3^N$ as before and for $j\geq 4$, $f_j^N:=f_3^N$ and by the same reasoning, one can conclude that this operator $T_\mathfrak{m}$ is not bounded from $L^{p_1}\times \cdots \times L^{p_{n-1}}$ to $L^{p_n'}$ with $1<p_1,\dots,p_{n-1},p_n'<\infty$ and $\frac{1}{p_1}+\cdots +\frac{1}{p_{n-1}}=\frac{1}{p_n'}$. 
\bibliographystyle{plain}
\bibliography{refs}
\end{document}